\documentclass[UTF-8,reqno]{amsart}
\usepackage{enumerate}
\usepackage{mhequ}
\usepackage[margin=1.2in]{geometry}

\usepackage{microtype}
\usepackage{amssymb,url,color, booktabs,nccmath}
\usepackage{mathrsfs}
\usepackage{enumitem}
\usepackage{graphicx}
\usepackage{tikz}
\usetikzlibrary{shapes,snakes}
\usetikzlibrary{calc}
\usetikzlibrary{decorations.shapes}
\usepackage{color}
\usepackage[colorlinks=true]{hyperref}
\hypersetup{
    linkcolor=blue,          
    citecolor=red,        
    filecolor=blue,      
    urlcolor=cyan
}

\usepackage[colorinlistoftodos,prependcaption,color=yellow,textsize=tiny]{todonotes}

\definecolor{darkergreen}{rgb}{0.0, 0.5, 0.0}

\numberwithin{equation}{section}

\newcommand{\be}{\begin{eqnarray}}
\newcommand{\ee}{\end{eqnarray}}
\newcommand{\ce}{\begin{eqnarray*}}
\newcommand{\de}{\end{eqnarray*}}
\newtheorem{theorem}{Theorem}[section]
\newtheorem{lemma}[theorem]{Lemma}
\newtheorem{remark}[theorem]{Remark}
\newtheorem{definition}[theorem]{Definition}
\newtheorem{proposition}[theorem]{Proposition}

\newtheorem{corollary}[theorem]{Corollary}
\newtheorem{assumption}{Assumption}[section]

\newcommand{\update}[1]{{#1}}

\allowdisplaybreaks

\tikzset{
        dot/.style={circle,fill=black,inner sep=0pt, outer sep=0.7pt, minimum size=1mm},
        Phi/.style={white!40!red,thick,snake=coil,segment amplitude=0.6pt, segment length=2pt},
         Z/.style={black!40!green,thick,snake=coil,segment amplitude=0.6pt, segment length=2pt},
        C/.style={thick,black!20!blue},
          Cr/.style={thick,black!20!red},
            Cg/.style={thick,black!20!green},
       }

\begin{document}
\newcommand{\Q}{\mathbb{Q}}
\newcommand{\R}{\mathbb{R}}
\newcommand{\cG}{\mathcal{G}}
\newcommand{\I}{\mathbb{I}}
\newcommand{\Z}{\mathbb{Z}}
\newcommand{\N}{\mathbb{N}}
\newcommand{\F}{\mathcal{F}}
\newcommand{\p}{\mathbb{P}}
\newcommand{\XX}{\mathbb{X}}
\newcommand{\M}{\mathcal{M}}
\newcommand{\Pp}{\mathcal{P}}
\newcommand{\W}{\mathcal{W}}
\newcommand{\ED}[1]{{#1}^\epsilon_\delta}
\newcommand{\EDM}[1]{{#1}^\epsilon_{\delta,m}}
\newcommand{\EOM}[1]{{#1}^\epsilon_{0,m}}
\newcommand{\EDMg}[1]{{#1}^{\epsilon,g}_{\delta,m}}
\newcommand{\RR}{\mathcal{R}}
\newcommand{\E}{\mathbb{E}}
\newcommand{\h}{\mathcal{H}}
\newcommand{\cL}{\mathcal{L}}
\newcommand{\Ii}{\mathrm{I}}
\newcommand{\eps}{\epsilon}
\newcommand{\supp}{\mathrm{supp}}
\newcommand{\law}{\mathrm{Law}}
\newcommand{\leb}{\lambda}
\newcommand{\id}{\mathrm{id}}
\newcommand{\pr}{\mathrm{pr}}
\newcommand{\Ss}{\mathfrak{S}}
\newcommand{\inter}{\mathrm{int\,}}
\newcommand{\dhcomment}[1]{\textbf{\textcolor{red}{#1}}}
\newcommand{\PP}{\mathbb{P}}
\newcommand{\dhedit}[1]{{#1}}


\title[Landau-Lifshitz-Navier-Stokes Equations]{\Large Landau-Lifshitz-Navier-Stokes Equations: Large Deviations and Relationship to The Energy Equality}

\author{Benjamin Gess}
\author{Daniel Heydecker}
\author{Zhengyan Wu}

\address{B. Gess: Fakultat f\"ur Mathematik, Universit\"at Bielefeld, 33615 Bielefeld, Germany \\ Max Planck Institute for Mathematics in the Sciences, 04103 Leipzig, Germany.}
\email{Benjamin.Gess@mis.mpg.de}
\address{D. Heydecker: Max Planck Institute for Mathematics in the Sciences, 04103 Leipzig, Germany.}
\email{daniel.heydecker@mis.mpg.de}

\address{Z. Wu: Faculty of Mathematics, University of Bielefeld, 33615 Bielefeld, Germany.}
\email{zwu@math.uni-bielefeld.de}


\begin{abstract}The dynamical large deviations principle for the three-dimensional incompressible Landau-Lifschitz-Navier-Stokes equations is shown, in the joint scaling regime of vanishing noise intensity and correlation length. This proves the consistency of the large deviations in lattice gas models \cite{QY}, with Landau-Lifschitz fluctuating hydrodynamics \cite{LL87}. Secondly, in the course of the proof, we unveil a novel relation between the validity of the deterministic energy equality for the deterministic forced Navier-Stokes equations and matching large deviations upper and lower bounds. In particular, we conclude that time-reversible uniqueness to the forced Navier-Stokes equations implies the validity of the energy equality, thus generalising the classical Lions-Ladyzhenskaya result. Thirdly, we prove that no non-trivial large deviations result can be true for local-in-time strong solutions.

\end{abstract}

\maketitle

\setcounter{tocdepth}{1}
\tableofcontents
\section{Introduction}  \label{sec-1} This paper is dedicated to the study of the large deviations for $\varepsilon,\delta\to 0$ of the Landau-Lifschitz-Navier-Stokes equations
\begin{equation}\label{SNS-1}
\partial_t u=\Delta u-(u\cdot\nabla)u-\nabla p-\sqrt{\epsilon}\nabla\cdot\xi_{\delta}, \ \ {\rm{div}}(u)=0
\end{equation} posed on $x\in \mathbb{T}^3, t\ge 0$, and their relation to the open problem of anomalous dissipation for the forced Navier-Stokes equations
\begin{equation}\label{control}
\partial_tu=\Delta u-(u\cdot\nabla)u-\nabla p-\nabla\cdot g , \ \ {\rm{div}}(u)=0
\end{equation}
with force in the Leray class $\nabla\cdot g \in L^2_tH^{-1}_x$. Here, $p$ is the pressure, and $\xi_\delta$ denotes the mollification of a matrix-valued divergence-free white noise $\xi$ on a spatial scale $\delta$. 

\medskip\noindent
A central motivation of the present work, and the first main result, is to show the consistency of the macroscopic fluctuation theory (MFT) associated to the lattice gas model of the Navier-Stokes equations introduced by Quastel and Yau \cite{QY} with the large deviations associated to the equations  (\ref{SNS-1}) of fluctuating hydrodynamics (FHD). Generally speaking, MFT introduces a general framework for describing the fluctuations of observables far from equilibrium \cite{BDGJL,Derrida} by postulating large deviation rate principles for diffusive systems.  On the other hand, the theory of FHD provides a framework for modelling microscopic fluctuations in a manner consistent with statistical mechanics and non-equilibrium thermodynamics by means of stochastic differential equations (SPDEs) with conservative noise \cite{LL87,HS,FG22, FG21}.
In the context of fluid dynamics, the rate functions of MFT were introduced Quastel and Yau in \cite{QY}, and the FHD equations \eqref{SNS-1}, with $\delta=0$, were introduced \cite{LL87} to describe thermodynamic fluctuations in fluids. 
While on an informal level MFT and FHD are linked via large deviations principles for the FHD equations, the rigorous justification of this is largely open.  
\\  
This programme is achieved for fluctuations about the incompressible Navier-Stokes equations in\textbf{ Theorem \ref{thrm: LDP}} below, by showing that the large deviations of (\ref{SNS-1}) are governed by the same rate function $\mathcal{I}$ as in \cite[Theorem 2]{QY} on sufficiently integrable fluctuations.

\medskip\noindent  An important part of the large deviation lower bound arguments 
are uniqueness results for the perturbed limit equation, given in the present setting by (\ref{control}), see for example \cite{FG22,DH} for a discussion of the problem. When uniqueness is not available, one can argue instead via the explicit identification of the lower semicontinuous envelope of the rate function restricted to a smaller set of paths. For example, this program was achieved for purely hyperbolic scaling laws via a Young measure approach by Bellettini et al. \cite{bellettini2010gamma} and Mariani \cite{mariani2010large}. For the equation (\ref{SNS-1}), neither of these approaches are available. Indeed,  Albritton et al. \cite{albritton2022non} demonstrated nonuniqueness of solutions to (\ref{control}) in the Leray class $L^\infty_tL^2_x\cap L^2_tH^1_x$, while attempting to generalise the argument of \cite{bellettini2010gamma,mariani2010large} is hindered by the relative gain of regularity through the diffusive term. We instead develop an LDP analysis based on \emph{weak-strong uniqueness}. This extension of tools of weak-strong uniqueness to the field of large deviations appears to be new. 
As a result,  we prove matching large deviations bounds on sets of solutions that are characterised by certain extensions of weak-strong uniqueness classes, see Definition \ref{def: wk stron uniqueness class}. For concreteness we show that this includes the choice $\mathcal{C}:=L^4([0,T]\times\mathbb{T}^3, \mathbb{R}^3)$.  

\medskip\noindent  

One of the key technical difficulties in this approach to large deviations is taking the limit of the nonlinear convection $(u\cdot \nabla)u$ appearing in the (stochastic) PDE. In the present context, this is achieved through working in a well-chosen path space $\XX$, which transfers these difficulties into proving sufficiently strong compactness estimates on the solutions to (\ref{SNS-1}). This allows to concentrate the effort on the aspects which are specific to (\ref{control}) and the weak-strong uniqueness approach developed in this work. 

\medskip\noindent  Theorem \ref{thrm: LDP} below proves that the dynamic large deviations of (\ref{SNS-1}) are governed by the skeleton Navier-Stokes equation \eqref{control} with a control $g$ ranging over $L^2_tL^2_x$. In the course of the proof, we uncover a novel link between the problem of large deviations for (\ref{SNS-1}) 
and key problems in nonlinear PDEs, namely, the weak-strong uniqueness and the problem of anomalous dissipation for the forced Navier-Stokes equations (\ref{control}). Based on insight from the large deviations principle, this leads to a second main result, \textbf{Theorem \ref{thrm: Energy Equality}}, proving that for the deterministic, forced Navier-Stokes equations the energy equality holds on certain closures of weak-strong uniqueness classes. This provides a first systematic answer to the open problem posed in \cite{berselli2020energy} on how (weak-strong) uniqueness of solutions is related to the energy equality. We will next comment on this  challenge. 

\medskip\noindent 
Leray solutions to (\ref{control}) satisfy an energy \emph{inequality} \begin{equation}
	\label{eq: EI} \frac12\|u(t)\|_{L^2(\mathbb{T}^3)}^2+\int_0^t \|u(s)\|_{\dot{H}^1(\mathbb{T}^3)}^2 ds\le \frac12\|u(0)\|_{L^2(\mathbb{T}^3)}^2- \int_0^t \langle \nabla u, g\rangle ds 
\end{equation} for all $0\le t\le T$. This energy (in)equality may be formally derived by multiplying (\ref{control}) by $u$ itself and integrating by parts, but the available \emph{a priori} estimates do not offer enough regularity to make this calculation rigorous. In general, the question of which conditions, if any, are necessary to guarantee that Leray solutions satisfy the \emph{equality} case of (\ref{eq: EI}) is an open problem: if the inequality is strict, then the solution exhibits \emph{anomalous dissipation}. It is a celebrated result of Lions-Ladyzhenskaya \cite{Lionsjl,PG,KL} that $u\in L^4_tL^4_x$ is sufficient to guarantee the energy equality; other conditions of type $L^p_tL^q_x$ have been found by Shinbrot \cite{shinbrot1974energy} and da Veiga-Beirao-Yang \cite{da2020shinbrot} and conditions of type $L^p_tW^{\alpha, q}_x$ have been found by Bersilli-Chandroli \cite{berselli2020energy}, da Veiga-Beirao-Yang \cite{da2019energy} and Zhang \cite{zhang2019remarks}. \\
%
In comparison, the known sufficient conditions for weak-strong uniqueness of solutions are more restrictive:
for instance, the conditions in \cite{shinbrot1974energy,Lionsjl,PG,KL,da2020shinbrot} together cover all cases of the Ladyzhenskaya-Prodi-Serrin condition for uniqueness \cite{RRS,WE}. In addition, the counterexamples to uniqueness for the forced Navier-Stokes equations  \cite{albritton2022non} satisfy the energy equality. In particular the energy equality alone cannot guarantee uniqueness. \\
This leaves open the problem of whether a form of (weak-strong) uniqueness implies the energy equality, see \cite{berselli2020energy}. This is addressed by Theorem \ref{thrm: Energy Equality}, by demonstrating 
that the energy equality (\ref{eq: EI}) holds on certain closures of weak-strong uniqueness classes. The method and result are abstract, and in particular apply to any of the commonly investigated types $L^p_tL^q_x, W^{\alpha,p}_tW^{\beta,q}_x, L^p_tC^\alpha_x$. 

\medskip\noindent 
Another novel observation of this work is a direct link between the approach to the large deviations lower bound via lower semicontinuous envelopes, and the question of the energy (in)equality and {anomalous dissipation} in the forced Navier-Stokes equations (\ref{control}). As a consequence of Theorem \ref{thrm: Energy Equality}, in Remark  \ref{rmk: violations of EI}, we show
 that the existence of solutions to (\ref{control}) violating the energy equality at the terminal time $t=T$ 
  leads to violations of the large deviations lower bound. \\
 Notably, in the forced setting, solutions violating the energy equality with forces in $L^{2-\epsilon}_tH^{-1}_x$ are known, see \cite{CL}, which are thus only $\epsilon$ less regular than those  appearing in the large deviations.

 \medskip \noindent  
 The third main result of this work addresses the (im)possibility of using stronger notions of solutions to the Landau-Lifschitz-Navier-Stokes equations in the proof of large deviations principles.\\
 While in the deterministic theory, the problems of nonuniqueness and energy equalities have been resolved under additional restrictions such as short times or the smallness of the initial data \cite{Cao-Titti,kochtataru}, we demonstrate that the context of large-deviations for (\ref{SNS-1}) makes the use of any such theory impossible. In the third main result \textbf{Theorem \ref{thm: LDPstrong}}, we prove that at speed $\epsilon^{-1}$ with negligible large deviations cost, and in arbitrarily fast time, the $H^1(\mathbb{T}^3)$-norm of fluctuations of \eqref{SNS-1} becomes arbitrarily large. As a consequence, no non-trivial large deviations principle for strong solutions can be true at this speed. At any slower large deviations speed, however, the only limits with finite rate are initial segments of the maximal strong solution to the Navier-Stokes equations, again excluding any non-trivial large deviations principle for strong solutions.

   \medskip \noindent 
  The central problem in the study of blowup at the large deviation level is as follows: given a local, strong solution $\bar{u}(s), 0\le s<\tau$ to (\ref{control}) for a control $g$ and a time $0<t<\tau$, one must perturb $g$ to $g'$ such that the corresponding $\bar{u}'$ is close to $\bar{u}$ on $[0,t)$, but has become large in the stronger norm at time $t$. However, in general, in light of the supercritical nonlinearity, answering the question how much regularity solutions to (\ref{control})  for a given $g$ have appears intractable. The key observation here is that loss of regularity is related to the interplay between the dissipation term $\Delta u$ and the rough forcing $\nabla\cdot g, g\in L^2_t L^2_x$, rather than the nonlinear convective term $u\cdot \nabla u$. Indeed, for the same class of $g$, the linear evolution 
    \begin{equation}\label{eq: linear control} 
       \partial_t v = \Delta v-\nabla \cdot g
    \end{equation} 
    is already critical for $u\in L^2(\mathbb{T}^3)$, and supercritical for the norms $L^3(\mathbb{T}^3)$, ${\rm BMO}^{-1}(\mathbb{T}^3)$ which are scaling-critical for the Navier-Stokes equations. This observation will allow us to explicitly construct $g', \bar{u}'$ in Section \ref{sec-9}, which ultimately leads to a proof of Theorem \ref{thm: LDPstrong}.
    
    

\section{Statement of Results}

\subsection{Definitions} We now give some definitions of the objects which will be in use throughout the paper: function spaces of constant use, a path space $\XX$ of Leray-type regularity, and the precise meaning of the white noise and Leray-type solution theory of (\ref{SNS-1}). Two definitions, of local strong solutions to (\ref{SNS-1}) and an accompanying path space $\XX_{\rm st, loc}$, are defered until they are first used in Section \ref{sec-9}.
\paragraph{\textbf{Function Spaces}} Let us first introduce some function spaces of frequent use in the sequel. Denoting $\mathcal{X}_{\mathrm {d.f}}:=\{f\in C^{\infty}(\mathbb{T}^3;\mathbb{R}^3),\ \nabla\cdot f=0\}$ for the space of divergence-free test functions, let us take $H$ to be the closure of $\mathcal{X}_{\mathrm {d.f}}$ in $L^2(\mathbb{T}^3, \mathbb{R}^3)$. We denote by $\langle\cdot,\cdot\rangle$ the inner product of $H$, and write $\mathbf{P}$ for the Leray projection from $L^2(\mathbb{T}^3, \mathbb{R}^3)$ to $H$. We next define the space $\mathcal{M}$ as the space of divergence-free, matrix-valued $L^2$-functions: \begin{equation}
	\mathcal{M}:=\left\{g\in L^2(\mathbb{T}^3,\mathbb{R}^{3\times 3}): \sum_{j=1}^d \partial_j g_{ij}=0 \right\}
\end{equation}understanding the divergence-free condition in the sense of distributions, and we equip $\mathcal{M}$ with the restriction of the $L^2(\mathbb{T}^3, \mathbb{R}^{3\times 3})$ norm $\|g\|_{\mathcal{M}}:=\|g\|_{L^2(\mathbb{T}^3, \mathbb{R}^{3\times 3})}$. Finally, we define a further seminorm $
\|u\|_V:=\|\nabla u\|_{\mathcal{M}}$, allowing the value $\infty$. We write $V\cap H$ for the subspace of $H$ where this is finite, which we equip with the (Hilbert) norm\footnote{The norm $\|\cdot\|_{H\cap V}$ is thus exactly the Sobolev $H^1(\mathbb{T}^3)$-norm, and the seminorm $\|\cdot\|_V$ is the homoegenous part. The fact that $\|u\|_{L^2([0,T],V)}$ fails to be positive definite will not cause problems in the sequel.} $\|u\|_{V\cap H}^2:=\|u\|_H^2+\|u\|_V^2$. We write $L^2([0,T], H\cap V), L^2([0,T],V)$ for the vector spaces of measurable functions $[0,T]\to H$ for which the functionals $$ \|u\|_{L^2([0,T],V)}^2:=\int_0^T \|u(t)\|_V^2 dt; \qquad \|u\|_{L^2([0,T],H\cap V)}^2:=\int_0^T \|u(t)\|_{H\cap V}^2 dt $$ are finite. We will take a basis of $H$ consisting of the eigenvectors $\{e_\zeta\}_{\zeta\in \mathcal{B}}$ of the Stokes operator $A:=-\mathbf{P}\Delta$, indexed by $\mathcal{B}:=((\mathbb{Z}^3\setminus 0) \times \{\pm 1\})\sqcup (\{0\}\times\{\pm 1, 0\})$, where each $e_{\zeta}, \zeta=(k, \theta)$ is a normalised plane wave of frequency $2\pi k$; the corresponding eigenvalue is $\lambda_\zeta=4\pi^2 |k|^2$. We define a Galerkin projection $P_m$ onto the span of $e_\zeta, \zeta \in \mathcal{B}_m$ of those eigenfunctions with frequency $|k|\le m$:\begin{equation}\label{eq: Pm}
P_mf:=\sum_{\zeta\in \mathcal{B}_m} \langle f,e_{\zeta}\rangle e_{\zeta}.
\end{equation}  

\paragraph{\textbf{Path Space and its topologies}} We may now define the path space $\XX$ in which we consider large deviations. We define the path space to be \begin{equation}\label{eq: def X} \XX:=L^2([0,T],H)\cap C([0,T],(H,\mathrm{w}))\cap L^2_{\rm w}([0,T],H\cap V) \end{equation} and where $w$ denotes the weak topology of $H$, respectively $L^2([0,T],H\cap V)$, and where $C([0,T], (H,w))$ is the locally convex topology defined by the maps $u\mapsto \sup_{t\le T}|\langle \varphi, u\rangle|, \varphi\in H$.  We will occasionally work with a slight strengthening of the topology, in which we add the basic open sets induced by the seminorm $u\mapsto \|u(0)\|_H$. The spaces $\XX, \XX^+$ defined in this way are separable and Haussdorf. \paragraph{\textbf{Divergence-Free White Noise and Regularisation}}  \label{sec: approximation for WN}
Throughout, we will use the notation $(\Omega,\mathcal{F},\{\mathcal{F}_t\}_{t\in
[0,T]}, \mathbb{P})$ for a stochastic basis on which are defined a set of standard, one-dimensional Brownian motions $(\{\beta^{\zeta,i}(t)\}_{t\in[0,T]})_{\zeta\in \mathcal{B}, 1\le i\le 3})$. We write $\mathbb{E}$ for the expectation defined by $\mathbb{P}$. Without loss of generality, the filtration $\{\mathcal{F}_t\}_{t\in [0,T]}$ is assumed to be complete. \\\\We now formulate the noise term appearing in (\ref{SNS-1}). Let $\{\beta^{\zeta,i}\}_{\zeta \in \mathcal{B}, 1\le i\le 3}$ be a sequence of independent, standard Brownian motions in $\mathbb{R}$, and define $W=(W^i)_{1\le i\le 3}$ by
\begin{equation}\label{eq: Wi}
W^i(t):=\sum_{\zeta \in \mathcal{B}}\beta^{\zeta, i}(t)e_\zeta(x).
\end{equation}
 This will correspond to the noise $\xi$ appearing in (\ref{SNS-1}) through $\xi:=\frac{dW}{dt}$, although we will use $W$ in the sequel. We will use the terminology `divergence-free white noise' for the (matrix) white noise constructed in this way\footnote{It can be verified that the covariance structure of the Gaussian thus defined is the same as that in \cite{bell2007numerical}, which justifies our description of (\ref{SNS-1}) as the same equation studied therein.}. We remark that the Cameron-Martin space of the unreglarised, divergence-free white noise $W$ given by (\ref{eq: Wi}) is $L^2([0,T],\mathcal{M})$.   \medskip \\ We next specify a regularisation of the noise in space. We choose the regularisation $W_\delta:=\sqrt{Q_\delta}W$, where $Q_\delta$ is the operator on $H$ defined by defined by

\begin{equation}\label{Qdelta}
Q_{\delta}:=(I+\delta A^{2\beta})^{-1}.
\end{equation} \paragraph{\textbf{A Solution Class to (\ref{SNS-1})}} With above conditions, we define the class of weak solutions of (\ref{SNS-1}) with which we work. The solution theory mimics the classical arguments of Leray, and we therefore call the solution class \emph{stochastic Leray} solutions. 
\begin{definition}\label{weak-solution} Fix $\epsilon, \delta, T>0$. We say that a progressively measurable process $\ED{U}$, defined on a stochastic basis $(\Omega,\mathcal{F},\{\mathcal{F}(t)\}_{t\in[0,T]},\mathbb{P})$ equipped with a divergence-free white noise $W$ and an $H$-valued random variable $U_0$ independent of $W$, is a \emph{stochastic Leray solution started at $U_0$} if it ${\mathbb{P}}$-almost surely takes values in $\XX$, satisfies the equation (\ref{SNS-1}) in the sense that, for every $\varphi\in \mathcal{X}_{\rm d.f.}$, $\mathbb{P}$-almost surely, for all $0\le t\le T$,
\begin{equation} \begin{split}\label{weaksense}
\langle U^{\epsilon}_{\delta}(t),\varphi\rangle+\int_0^t\langle \nabla U^{\epsilon}_{\delta}(s),\nabla\varphi\rangle ds=\langle U_0,\varphi\rangle-\int_0^t\langle(U^{\epsilon}_{\delta}\cdot\nabla)U^{\epsilon}_{\delta},\varphi\rangle ds+\sqrt{\epsilon}\int_0^t \langle  \nabla \varphi, dW_\delta(r)\rangle 
\end{split} \end{equation} and satisfies the trajectorial energy inequality: $\mathbb{P}$-almost surely, for all $0\le t\le T$,
\begin{equation}\begin{split}\label{eq: TEI}
\frac{1}{2}\|U^{\epsilon}_{\delta}(t)\|_{H}^2\leq&\frac{1}{2}\|U_0\|_{H}^2-\int_{0}^{t}\| U^{\epsilon}_{\delta}(r)\|_{V}^2dr\\
&-\sqrt{\epsilon}\int_0^t\langle U^{\epsilon}_{\delta},\nabla\cdot dW_{\delta}(r)\rangle+ \frac{\epsilon}{2}\|A^{1/2}\circ\sqrt{Q_{\delta}}\|_{HS}^2t. \end{split}\end{equation}\end{definition}

\update{Let us remark that the solution theory will be slightly different from those of \cite{FG95,flandoliromito2008}; we defer a discussion of this point to the literature review.} 

\subsection{Main Results}
  We first give a result on existence. \begin{proposition}[Proposition \ref{existence}]\label{prop: main existence} For any $\epsilon, \delta>0$, there exists a stochastic basis $(\Omega, \mathcal{F}, (\mathcal{F}_t)_{t\ge 0}, \mathbb{P})$ on which are defined a divergence-free space-time white noise $W$ and a stochastic Leray solution $\ED{U}$ to (\ref{SNS-1}) for the white noise $W$. Under integrability conditions, $\ED{U}$ can be constructed so that $\ED{U}(0)$ has any given desired law, independently of the divergence-free white noise $W$.  \end{proposition}  Before giving the large deviations results, we introduce the necessary scaling relations on the parameters, and introduce the rate function. \paragraph{\textbf{Parameter Tuning}} The large deviations analysis may be carried out either at the level of the stochastic Leray solutions $\ED{U}$ as in Definition \ref{weak-solution}, or with the Galerkin approximations $\EDM{U}$ satisfying \begin{equation}
  	\label{eq: galerkin approx} \partial_t \EDM{U}=\Delta \EDM{U}-{P}_m((\EDM{U}\cdot \nabla)\EDM{U})-\sqrt{\epsilon}{P}_m(\nabla \cdot \xi_\delta)
  \end{equation} where ${P}_m$ is as in (\ref{eq: Pm}) and $m=m(\epsilon)\to \infty$ with $\epsilon$. In either case, we must impose a scaling relation on the parameters $(\epsilon, \delta)$ or $(\epsilon, \delta(\epsilon), m(\epsilon))$. In the former case, we impose the scaling relation that $\epsilon, \delta=\delta(\epsilon)$ must satisfy
\begin{equation}\label{scale-standing}
\epsilon\delta(\epsilon)^{-\frac{5}{4\beta}}\to 0
\end{equation} and that $\beta>\frac54$, where $\beta$ is the exponent in the definition (\ref{Qdelta}) of ${Q_\delta}$. In the latter case, we impose instead, \begin{equation}\label{scale-standing-galerkin} \epsilon \min\left(m(\epsilon)^5, \delta(\epsilon)^{-\frac{5}{4\beta}}\right) \to 0. \end{equation}

Let us remark that the two hypotheses (\ref{scale-standing}, \ref{scale-standing-galerkin}) allow both regularisation schemes in which one of the parameters is removed, either by taking $m\to \infty$ and tuning $\epsilon, \delta$ according to (\ref{scale-standing}), or by taking $\delta(\epsilon)=0, m(\epsilon)\ll \epsilon^{-\frac15}$. \\
    \paragraph{\textbf{Large Deviations of Initial Data}} We next give the assumption on the large deviations of the initial data. \begin{assumption}\label{hyp: initial data} Let $\epsilon, \delta(\epsilon)$ be scaling parameters satisfying (\ref{scale-standing}). In the case of solutions $\ED{U}$ to (\ref{SNS-1}), the laws $\mu^{\epsilon,0}_\delta$ of initial data $U^{\epsilon}_{\delta(\epsilon)}(0)$ satisfy a large deviations upper bound with rate function $\mathcal{I}_0$ in the weak topology of $H$, and a lower bound with the same rate in the \emph{norm} topology of $H$: whenever $A\subset H$ is closed for the weak topology and $U\subset H$ is open in the norm topology, we have \begin{equation} \limsup_{\epsilon} \epsilon\log \mu^{\epsilon, 0}_\delta(A)\le -\inf\left\{\mathcal{I}_0(u):u\in A\right\}\end{equation} and \begin{equation} \liminf_{\epsilon} \epsilon\log \mu^{\epsilon, 0}_\delta(U)\ge -\inf\left\{\mathcal{I}_0(u):u\in U\right\}. \end{equation} Furthermore, the function $\mathcal{I}_0(u)$ is continuous in the norm topology of $H$, has a unique zero $u_0^\star \in H$, and the initial data have exponential Gaussian moments satisfying \begin{equation}\label{eq: energy estimate at gaussian ID''}
    	\limsup_{\eta\to 0} \frac{Z(\eta)}{\eta}\le \|u_0^\star\|_H^2
    \end{equation} where $Z$ is the function \begin{equation}
	\label{eq: energy estimate at gaussian ID} Z(\eta):=\limsup_{\epsilon\to 0} \epsilon\log\mathbb{E}\left[\exp\left(\frac{\eta \|U^{\epsilon}_{\delta(\epsilon)}(0)\|_H^2}{\epsilon}\right)\right]. \end{equation} In particular, we require that $Z(\eta)$ is finite for sufficiently small $\eta>0$. In the case of solutions $\EDM{U}$ to the Galerkin approximation with $m(\epsilon)\to\infty, \delta(\epsilon)\to 0$, we assume that the same holds for the laws $\mu^{\epsilon, 0}_{\delta,m}$ of $\EDM{U}(0)$ in place of $\ED{U}$. \end{assumption} Throughout, the superscript $^0$ will be used to denote objects pertaining to the initial data, to distinguish them from the law $\ED{\mu}$ of the whole process. In order to show that this condition is not overly restrictive, we give the following example of such initial distributions. The proof is elementary, and so is omitted. \begin{lemma}\label{lemma: exponential calculation ID} Let $u_0\in H$. Assumption \ref{hyp: initial data} holds for Gaussian initial data $\ED{U}(0)\sim \cG(u_0, \epsilon Q_\delta/2)$ with mean $u_0$ and covariance $Q_\delta$, under the scaling relation (\ref{scale-standing}) on $\epsilon, \delta$, with a rate function $\mathcal{I}_0(v):=\|v-u_0\|_H^2$. The same holds for $\EDM{U}(0)\sim \cG(P_mu_0, \epsilon P_m Q_\delta/2)$, provided $\epsilon, \delta, m$ satisfy (\ref{scale-standing-galerkin}).

\end{lemma} The particular choice of $\EDM{U}(0)\sim \cG(0, \epsilon P_m Q_\delta/2)$ will be relevant for several remarks later, as it is a stationary distribution under which the law of $\EDM{U}$ is invariant under the time-reversal operation  \begin{equation}
	\label{eq: time reversal EDU} (\mathfrak{T}_Tu)(t,x):=-u(T-t,x),
\end{equation} see Lemma \ref{lemma: properties of galerkin}. \paragraph{\textbf{The Rate Function}} We now specify the rate function of the large deviations, related to the one reported in \cite{QY}. We define the dynamic cost $\mathcal{J}(u)$ to be \begin{equation}\label{dynamic} \mathcal{J}(u):=\frac12 \inf\left\{\|g\|_{L^2([0,T],\mathcal{M})}^2: u\text{ satisfies (\ref{control}) for }g\right\}\end{equation}
 where we understand the solution theory for (\ref{control}) in a weak sense. Let us remark that this is the dynamic rate function $\mathcal{I}_1$ in \cite{QY}. We set $\mathcal{I}$ to be the whole rate function\begin{equation}\label{rate-3-1}
\mathcal{I}(u)=\mathcal{I}_0(u(0)) + \mathcal{J}(u)
\end{equation} where $\mathcal{I}_0$ is the large deviation rate function of the initial data in Assumption \ref{hyp: initial data}. We next characterise the spaces on which we can restrict in the lower bound. \begin{definition}\label{def: wk stron uniqueness class} We say that $\mathcal{C}\subset \XX$ satisfies the weak-strong uniqueness property if, whenever $u\in \mathcal{C}, v\in \XX$ are weak solutions to the skeleton equation (\ref{control}) for the same $g$ and the same initial data $u(0)=v(0)$, and $v$ satisfies the energy inequality \begin{equation}\label{energy-5}
\frac{1}{2}\|v(t)\|_{H}^2+\int_0^t\| v(s)\|_{V}^2ds\leq\frac{1}{2}\|v_0\|_{H}^2+\int_0^t\langle\nabla v,g\rangle ds
\end{equation} then $u=v$. The $\mathcal{I}$-closure of a set $\mathcal{A}\subset \XX$ is defined by \begin{equation}
	\label{eq: def closure} \overline{\mathcal{A}}^{\mathcal{I}}:=\left\{u\in \XX: \exists u^{(n)}\in \mathcal{A}, \quad u^{(n)}\to u,\qquad  \mathcal{I}(u^{(n)})\to \mathcal{I}(u) \right\}
\end{equation} so that $\overline{\mathcal{A}}^{\mathcal{I}}$ is the maximal set on which $\mathcal{I}$ agrees with the lower semicontinuous envelope of its restriction to $\mathcal{A}$.\end{definition} Concretely, we show in Lemma \ref{weak-strong} that the Ladyzhenskaya-Prodi-Serrin condition \cite{RRS,WE} defines a weak-strong uniqueness class, and in Lemma \ref{l4 recovery} that its $\mathcal{I}$-closure contains $L^4([0,T]\times\mathbb{T}^3, \mathbb{R}^3)$.  We can now state the large deviations principle.
\begin{theorem}(Propositions \ref{proposition: final lower bound}, \ref{upresult-1})\label{thrm: LDP}
Let $\epsilon, \delta(\epsilon)$ be scaling parameters satisfying (\ref{scale-standing}), and for all $\epsilon>0$, let $\ED{\mu}$ be the law of a stochastic Leray solution $\ED{U}$ to (\ref{SNS-1}) started from initial data $\ED{U}(0)$ satisfying Assumption \ref{hyp: initial data}, and let $\mathcal{C}$ be the $\mathcal{I}$-closure of a class $\mathcal{C}_0$ satisfying the weak-strong uniqueness property. Then $\ED{\mu}$ satisfy the restricted large deviation lower bound, for every open subset $\mathcal{U}$ of $\XX$,
\begin{equation}\label{LB}
\liminf_{\epsilon\rightarrow0}\epsilon \log \ED{\mu}(\mathcal{U})\geq-\inf_{v\in \mathcal{U}\cap \mathcal{C}}\mathcal{I}(v),
\end{equation}
and the upper bound, for every closed subset $\mathcal{V}$ of $\XX$,
\begin{equation}\label{UB}
\limsup_{\epsilon\rightarrow0}\epsilon \log \ED{\mu}(\mathcal{V})\leq-\inf_{v\in \mathcal{V}}\mathcal{I}(v),
\end{equation}
where the rate function $\mathcal{I}$ is defined by (\ref{rate-3-1}). Moreover, $\mathcal{C}$ can concretely be chosen to $L^4([0,T]\times\mathbb{T}^3,\mathbb{R}^3)$ in (\ref{LB}). Finally, the same holds for $\EDM{U}$, if $\epsilon, \delta=\delta(\epsilon)\to 0, m=m(\epsilon)\to \infty$ are chosen according to (\ref{scale-standing-galerkin}).
\end{theorem}

This achieves the goal, discussed in the introduction, of showing that the stochastic Leray solutions of (\ref{SNS-1}) have the same large deviations as the particle model in \cite{QY} on a large subset of $\XX$. We will present several mild sharpenings of the upper bound, related to the energy inequality (\ref{energy-5}). The first result in this direction which we will present is a mild sharpening of the upper bound. The rate function $\mathcal{I}$ appearing in (\ref{UB}) vanishes on all \emph{weak} solutions $u\in \mathbb{X}$ to the Navier-Stokes equations starting at $u_0^\star$, while the Leray theory also has the energy inequality (\ref{eq: EI}) in the special case $g=0$. We will now show a strengthening with a new rate function which vanishes only on Leray solutions to Navier-Stokes starting at $u_0^\star$. \begin{theorem}\label{thm: improved ub}For any $u\in \mathbb{X}$, set $\mathcal{E}^\star(u)$ to be the maximal violation of the energy inequality (\ref{eq: EI}) for $u(0)=u_0^\star$: \begin{equation}\label{eq: Estar}
	\mathcal{E}^\star(u):=\sup_{t\le T}\left(\frac12\|u(t)\|_H^2+\int_0^t \|u(s)\|_V^2-\frac12 \|u_0^\star\|_H^2\right)^+
\end{equation} with $x^+$ denoting the nonnegative part of $x\in \mathbb{R}$. Then, under the same hypotheses as Theorem \ref{thrm: LDP}, there exists a function $S:[0,\infty)\to [0,\infty)$, which vanishes only at $S(0)=0$, such that the upper bound (\ref{UB}) holds with the new rate function \begin{equation}
	\mathcal{I}'(u):=\max\left(\mathcal{I}(u), S(\mathcal{E}^\star(u))\right)
\end{equation} which vanishes only on \emph{Leray} solutions to the Navier-Stokes equations starting at $u_0^\star$. The function $S$ may be made explicit in terms of the function $Z(\eta)$ defined in (\ref{eq: energy estimate at gaussian ID}). \end{theorem}   We next give the theorem relating the classes $\mathcal{C}$ appearing in the lower bound of Theorem \ref{thrm: LDP} to the energy equality (\ref{eq: EI}).  \begin{theorem}
\label{thrm: Energy Equality}  Let $\mathcal{C}\subset \XX$ be as in the lower bound of Theorem \ref{thrm: LDP}, and $\mathcal{R}\subset \mathcal{C}$ be such that $\mathcal{R}=\mathfrak{T}_T \mathcal{R}$, where the time-reversal operator $\mathfrak{T}_T$ is defined in (\ref{eq: time reversal EDU}). Then, whenever $u\in \mathcal{R}$ has finite rate $\mathcal{I}(u)<\infty$, the energy equality holds: \begin{equation}
	\label{eq: general energy equality} \frac12 \|u(T)\|_H^2+\int_0^T \|u(s)\|_V^2 = \frac12 \|u(0)\|_H^2+\int_0^T\langle \nabla u, g\rangle ds,
\end{equation} where $g\in L^2([0,T],\mathcal{M})$ is such that $u$ solves (\ref{control}).   \end{theorem} As discussed in the introduction, this theorem partially resolves the problem of connecting uniqueness criteria to the energy equality \cite{berselli2020energy}. 

We also find, via a direct and probabilistic argument the following connection of the PDE problem to the sharpness of the large deviation principle. \begin{remark}\label{rmk: violations of EI} Set $\delta(\epsilon)=0$ and choose $m(\epsilon)$ according to (\ref{scale-standing-galerkin}). Let $\EDM{U}$ be solutions to (\ref{eq: galerkin approx}) starting from equilibrium $\EDM{U}(0)\sim \cG(0, \epsilon P_m Q_\delta/2)$. Then violations of the energy identity are related to the LDP of $U^\epsilon_{0,m}$ as follows:
  	\begin{enumerate}[label=\roman*)] \item \emph{(Non-Leray Solutions Violate Lower Bound)} Suppose $u\in \mathbb{X}$ is a solution to (\ref{control}) for some $g\in L^2([0,T],\mathcal{M})$, which violates the energy inequality at the terminal time $T$: \begin{equation}
  		\frac12 \|u(T)\|_H^2 + \int_0^T \|u(s)\|_V^2 > \frac12 \|u(0)\|_H^2+\int_0^T \langle \nabla u, g\rangle ds.
  	\end{equation} Then $u$ violates the large deviations lower bound: \begin{equation}\label{eq: failure of LB}
  		\inf_{\mathcal{U}\ni u} \liminf_{\epsilon\to 0} \epsilon \log \mu^\epsilon_{0, m}(\mathcal{U}) < - \mathcal{I}(u)
  	\end{equation} where the infimum runs over all open sets $\mathcal{U}\subset \XX$ containing $u$. \item \emph{(Anomalous dissipation violates lower bound for reversed trajectory)} Suppose instead that  $u\in \XX$ is a solution to the controlled equation (\ref{control}) for which the equality (\ref{eq: general energy equality}) is a strict inequality. Then the lower bound fails, in the sense of (\ref{eq: failure of LB}), for the trajectory $v:=\mathfrak{T}_Tu$. \end{enumerate}
  \end{remark} \dhedit{Let us remark on the difference between the two different situations in this result, and in Theorem \ref{thm: improved ub}. Remark \ref{rmk: violations of EI} gives (implicitly) an improved upper bound on non-Leray solutions to (\ref{control}), provided that the energy inequality remains violated at the terminal time $T$, and for the special case $\delta=0, m=m(\epsilon)\ll \epsilon^{-\frac15}$ and with initial data $\EOM{U}(0)$ started in equilibrium. In contrast, the upper bound in Theorem \ref{thm: improved ub} applies to more general scalings on $\epsilon, \delta, m$ and more general initial conditions, and penalizes violations of the energy inequality at any time $0\le t\le T$, but only in comparison to the unforced energy inequality.} \medskip \\  Finally, we give our result on the large deviations of the strong solution. The definition of strong solution is given in Definition \ref{strong-solution} and the definition of the strong local path space $\XX_{\rm st, loc}$ and its topology are given in Section \ref{subsec: XXstloc}.  \begin{theorem}[Triviality of large deviations of the strong solution]\label{thm: LDPstrong}Let $\ED{\overline{U}}(0)$ be a $V$-valued random variable, satisfying a large deviation principle with speed $\epsilon^{-1}$ in the weak topology of $H$, with a good rate function $\mathcal{I}_0$ vanishing at a unique $u^\star_0$. For a divergence-free white noise $W$, let $(\ED{\overline{U}}(t), 0\le t<\tau^\epsilon)$ be the maximal strong solution to (\ref{SNS-1}) starting at this initial data. Let $\delta=\delta(\epsilon)$ be chosen to satisfy (\ref{scale-standing}) and let $r_\epsilon=r(\epsilon, \delta(\epsilon))$ be any putative large deviations speed, which diverges as $\epsilon\to 0$. For large deviations in $\XX_{\rm st, loc}$, we have the following dichotomy. \begin{enumerate}[label=\roman*).]
	\item If $\inf_{\epsilon} \epsilon r_\epsilon>0$, then the large deviations at speed $r_\epsilon$ include spontaneous blowup with zero large-deviations cost: for any partial path $\overline{u}$ and any initial segment $\overline{v}$ of $\overline{u}$, it holds that \begin{equation}
	\label{eq: cheap blowup}	\inf_{\overline{\mathcal{U}}\ni \overline{v}}\liminf_{\epsilon\to 0} r_\epsilon^{-1} \log \mathbb{P}\left(\ED{\overline{U}}\in \overline{\mathcal{U}}\right) \ge \inf_{\overline{\mathcal{U}}\ni \overline{u}}\liminf_{\epsilon\to 0} r_\epsilon^{-1} \log \mathbb{P}\left(\ED{\overline{U}}\in \overline{\mathcal{U}}\right)
	\end{equation}where, on both sides, the infimum runs over open sets $\overline{\mathcal{U}}$ containing $\overline{v}$, respectively $\overline{u}$. For the trivial path $\star$, we observe instantaneous blowup with zero large deviations cost:\begin{equation}\label{eq: instant blowup} \inf_{\overline{\mathcal{U}} \ni \star} r_\epsilon^{-1} \liminf_{\epsilon\to 0}\log \PP\left(\ED{\overline{U}}\in \overline{\mathcal{U}}\right) = 0.
		\end{equation} \item Let $\overline{u}^\star$ be the maximal strong solution to the Navier-Stokes equation starting from $u^\star_0$. If $\epsilon r_\epsilon\to 0$, then for any $\overline{v}$ which is not an initial segment of $\overline{u}^\star$, it holds that \begin{equation}\label{eq: inf rate} \inf_{\overline{\mathcal{U}}\ni \overline{v}} \limsup_{\epsilon\to 0} r_\epsilon^{-1}\log \PP\left(\ED{\overline{U}} \in \overline{\mathcal{U}}\right) = -\infty. \end{equation} \end{enumerate} Both statements hold, no matter what scaling (at least as strong as (\ref{scale-standing})) is imposed on $\epsilon, \delta$. \end{theorem}

\subsection{Structure of the paper}
The paper is structured as follows. The remainder of this section discusses connections of the results to the literature. In Section \ref{sec-2}, we introduce some objects of frequent use, including the noise $\xi$ and the regularisation. In Section \ref{sec-4}, we prove some tightness estimates in the space $\XX$, based on the energy inequality; these are used in Section \ref{sec: exist} to prove Proposition \ref{prop: main existence}. In Section \ref{sec-5}, we prove some properties of the rate function $\mathcal{I}$. The lower bound of Theorem \ref{thrm: LDP} is proven in Section \ref{sec-6}, and the upper bound is proven in Section \ref{sec-7}. Section \ref{sec: energy} relates the large deviations to the energy equality to prove Theorem \ref{thrm: Energy Equality} and Remark \ref{rmk: violations of EI}. In Section \ref{sec-9}, we discuss the triviality of large deviations of the local-in-time strong solution as a justification for working with the class of stochastic Leray solutions. 
\subsection{Comments on the literature}
\paragraph{\textbf{Stochastic Navier Stokes and Large Deviations}} By now, the analysis of classical Navier-Stokes equations has attracted significant attention, see, for example,  \cite{Cao-Titti,kochtataru} for the local-in-time existence of regular solutions, \cite{RRS,WE} for weak-strong uniqueness, \cite{albritton2022non,guillod2017numerical,jia2015incompressible,Bu20} for a recent non-uniqueness theorem in the Leray space $\XX$,  \cite{Lionsjl,PG,KL,shinbrot1974energy,da2020shinbrot,berselli2020energy,da2019energy,zhang2019remarks} for conditions related to the energy (in)equality, and Leslie-Shvydkoy \cite{leslie2018energy} for regularity analysis related to failures of the energy equality. The stochastic Navier-Stokes equations were studied by Flandoli \cite{FF}, Flandoli-Gatarek \cite{FG95}; let us also refer to the review articles \cite{mattingly2003recent,bensoussan1995stochastic,breit2018introduction} and references therein for an overview of the area. The Landau-Lifshitz-Navier-Stokes equations have attracted significant interest in relationship to numerical schemes for fluctuations, see for instance \cite{bell2007numerical,russo2021finite}.

Within the  literature on stochastic perturbations of the Navier-Stokes equations, there is by now a substantial literature on large deviation principles (LDP). Dynamic large deviation principles were obtained in \cite{SS,CD,ZZ,Z,RZ2,AX}. Related large deviations for the invariant measure were obtained by \cite{MD,BC}, and LDP for the occupation measure was studied in \cite{N,GM}. We point out that the notion of stochastic Leray solution in Definition \ref{weak-solution} differs from the classical definition of martingale solutions, see for example \cite[Definition 3.1]{FG95}. Similar solutions were shown to exist by Flandoli-Romito \cite[Theorem 3.7]{flandoliromito2008}; in this definition the martingale appearing in the energy identity is not explicitly identified, and only characterised by upper bounds on the quadratic variation. For the large deviation theory, it will be important to have the explicit characterisation of the martingale term. This has been achieved for the case of compressible Navier-Stokes in \cite{BFH}, see \cite[Defintion 3.4.1]{BFH}; for completeness we give the full argument.

 \paragraph{\textbf{Independence of Uniqueness Theory}}  Let us remark that we make no statements regarding the uniqueness or uniqueness-in-law of the stochastic Leray solutions; consequently, Theorems \ref{thrm: LDP} - \ref{thm: improved ub} apply to any choice stochastic Leray solutions $\ED{U}$. The pathwise uniqueness of solutions appears to be as hard as the problem of uniqueness for the deterministic Navier-Stokes equations, which remains open in the Leray class; see \cite{hofmanova2019non} for a discussion. In \cite{hofmanova2019non,hofmanova2023global}, the authors use a stochastic convex integration method to prove non-uniqueness in law for \emph{weak} solutions to the stochastic Navier-Stokes equations, where the pathwise energy inequality is removed. It is not yet clear whether a related construction would show non-uniqueness in the stronger setting of stochastic Leray solutions, and we will not attempt to address this problem.
  
\paragraph{\textbf{Large Deviations of Singular SPDE}} Regarding the large deviations of singular SPDEs, we refer, for example, to Hairer and Weber \cite{HW} and the references therein. Let us remark in particular \cite{HW} shows that the renormalisation constants for the stochastic Allen-Cahn equation appear in the large deviation rate functional. The situation we consider, where the noise is of the form $\nabla \cdot \xi_\delta$, is the setting of the conservative singular stochastic PDE studied by Dirr, Fehrman and the first author \cite{DFG}, and Fehrman and the first author \cite{FG22}, which establish a link between the large deviations of conservative SPDEs and those of particle systems. As already remarked, in these cases \cite{DFG}, \cite{FG22}, there is no renormalisation procedure that allows us to make sense of the SPDE with spatially uncorrelated noise. In \cite{DFG,FG16,GH}, the kinetic solution approach and the $L^1$-theory play important roles in the well-posedness of the perturbed `skeleton' equation, but, as already mentioned, there is no hope for such a strong uniqueness result in the present setting. 
\paragraph{\textbf{LDP for Particle Systems and the Quastel-Yau Lattice Gas}} A microscopic particle model, which produces the Navier-Stokes equations in $d=3$ in the incompressible limit, was introduced in \cite{QY} and the large deviations identified, and as remarked under Theorem \ref{thrm: LDP}, the model (\ref{SNS-1}) produces the same large deviations. This connection allows two different interpretations of Theorem \ref{thrm: LDP}. On the one hand, if one accepts the physical reasoning of \cite{LL87} leading to (\ref{SNS-1}), then the connection shows that the large deviations of the particle model are physically correct. On the other hand, if one accepts the particle model as fundamental, we can see ( \ref{SNS-1}) as producing a computationally tractable numerical approximation to the particle model, which improves the regularity of the sample paths from $L^2_tH^{-2}_x$ to the Leray class $\XX$ while maintaining the correct large deviation structure.
\paragraph{\textbf{Uniqueness and the Energy Equality}} While a number of different techniques have arisen to guarantee uniqueness or the energy equality, it is not known whether solutions persist in these classes. Various regularity conditions have been proposed for these problems, such as the Ladyzhenskaya-Prodi-Serrin condition for uniqueness (see \cite{RRS,WE}) and the Lions-Ladyzhenskaya criterion for the energy equality (see \cite{Lionsjl,PG,KL}). At the time of writing, the possible connection between the energy equality and uniqueness of weak solutions of the three-dimensional Navier-Stokes equations still remains open \cite{berselli2020energy}. This is resolved by Theorem \ref{thrm: Energy Equality}.

  \paragraph{\textbf{LDP for local-in-time solutions}} Finally, let us remark on literature relevant to Theorem \ref{thm: LDPstrong}. Large deviations of the stochastic Allen-Cahn equation, including the possibility of blowup, was studied by Hairer-Weber \cite{HW}, although their framework differs from the one underlying Theorem \ref{thm: LDPstrong} in assigning the whole trajectory to a single graveyard state if there is blowup, rather than tracking the behaviour up to the blowup time. The analysis in Section \ref{sec-9} is also different from the SDE case (see, for example Azencott \cite{AZ}). In moving to the PDE setting, it is possible for the process to blow up in one sense (in the present case, the divergence of stronger norms) while remaining bounded in another: there is therefore no contradiction between showing blowup in Theorem \ref{thm: LDPstrong} and the global existence required for Theorem \ref{thrm: LDP}. This is different from the SDE setting, where there is an unambiguous notion of blowup.

\section{Preliminary}\label{sec-2}

 We now introduce some further preliminary notions and calculations. We will use the notation $\lesssim$ to avoid repeatedly writing inequalities up to constants; the parameters on which the constant is allowed to depend will always be specified.

The scaling relation (\ref{scale-standing}) is imposed by the following elementary calculation.
\begin{lemma}
	\label{lemma: eigenvalue sum calculation} Suppose that (\ref{scale-standing}) holds. Then  \begin{equation}\label{eq: trace ed 1}
		\epsilon \|A^{1/2}\sqrt{Q_\delta}\|_\mathrm{HS}^2=\epsilon\mathrm{Tr}\left[AQ_\delta\right]\to 0.
	\end{equation}If instead (\ref{scale-standing-galerkin}) holds, we have \begin{equation}\label{eq: trace ed 2}
		\epsilon \|P_mA^{1/2} \sqrt{Q_\delta}\|_\mathrm{HS}^2=\mathrm{Tr}\left[P_mAQ_\delta\right]\to 0
	\end{equation} where $P_m$ is the Galerkin projection (\ref{eq: Pm}).  \end{lemma}
\begin{proof}[Sketch Proof]This follows from observing that $\sqrt{Q_\delta}, P_m, A^{1/2}$ are simultaneously diagonalised by the basis of eigenfunction $\{e_\zeta, \zeta\in \mathcal{B}\}$ and $$ A^{1/2}\sqrt{Q_\delta}e_\zeta = \left(\frac{\lambda_\zeta}{1+\delta|\lambda_\zeta|^{2\beta}}\right)^{1/2}e_\zeta. $$ The conclusion follows by summing $\frac{\lambda_\zeta}{1+\delta|\lambda_\zeta|^{2\beta}}$, using the exact expression for $\lambda_{(k, \theta)} \sim |k|^2$ and $\beta>\frac54$.  \end{proof} We note, for future reference, that $\sqrt{Q_\delta}$ regularises by more than one derivative, at the cost of a $\delta$-dependent constant: for all $u\in H$, \begin{equation} \label{gainonederivative}
	\|\sqrt{Q_\delta}\nabla u\|_\mathcal{M}\le C(\delta)\|u\|_H.
\end{equation}
		We will use the following equivalent characterisation of large deviations lower bounds, which mildly extends \cite[Lemma 7]{mariani2010large}.
\begin{lemma}\label{entropymethod}
Let $E$ be a \dhedit{separable and Hausdorff} space, $I: E\rightarrow[0,+\infty]$ be lower semicontinuous and let $\{\mu^{\epsilon}\}\subset\mathcal{P}(E)$ be a sequence of probability measures. Then $\{\mu^{\epsilon}\}$ satisfies the large deviation lower bound with speed $\epsilon^{-1}$ and rate $I$ if, and only if, for every $x\in E$, there exists a sequence $\{\pi^{\epsilon,x}\}\subset\mathcal{P}(E)$ such that $\pi^{\epsilon,x}\rightarrow\delta_x$ weakly in $\mathcal{P}(E)$, and
\begin{equation}\label{entropy}
\limsup_{\epsilon\rightarrow0}\epsilon{\rm{Ent}}(\pi^{\epsilon,x}|\mu^{\epsilon})\leq I(x).
\end{equation}

\end{lemma}
The `if' statement is \cite[Lemma 7]{mariani2010large}, and the `only if' statement follows from a simple calculation.\\ \\ 
Finally, we give a basic result regarding the Galerkin truncated equation (\ref{eq: galerkin approx}). Since the proof is entirely elementary, it is omitted.  \begin{lemma}\label{lemma: properties of galerkin}
	Fix a stochastic basis $(\Omega, \mathcal{F}, \{\mathcal{F}(t)\}_{t\in [0,T]}, \mathbb{P})$ equipped with a divergence-free white noise $W$. For any random variable $V_m$ taking values in $P_mH$ and which is independent of $W$, there exists a unique strong solution $U^{\epsilon}_{\delta,m}(t)$ to (\ref{eq: galerkin approx}) starting at $v_0$. The solutions satisfy, almost surely, \begin{align}\label{m-energy}
\frac{1}{2}\|U^{\epsilon}_{\delta,m}(t)\|_{H}^2=&\frac{1}{2}\|U^{\epsilon}_{\delta,m}(0)\|_{H}^2-\int_{0}^{t}\| U^{\epsilon}_{\delta,m}(r)\|_{V}^2dr\\
&-\sqrt{\epsilon}\int_0^t\langle U^{\epsilon}_{\delta,m},\nabla\cdot dW_{\delta}(r)\rangle+ \frac{\epsilon}{2}\|P_m\circ A^{1/2}\circ\sqrt{Q_{\delta}}\|_{HS}^2t.\notag
\end{align} In the special case $\delta=0$, the distribution $\cG(0, \epsilon P_m/2)$ is stationary, and if $U^\epsilon_{0,m}(0)\sim \cG(0, \epsilon I_m/2)$, then for any fixed $T>0$, the law of the process is preserved under the transformation \begin{equation}\label{eq: time reversal}
		\mathfrak{T}_T{U}^\epsilon_{0,m}:=(-U^{\epsilon}_{0,m}(T-t):0\le t\le T) =_\mathrm{d} (U^{\epsilon}_{0,m}(t):0\le t\le T).
	\end{equation}
\end{lemma} 
\section{Exponential estimates}\label{sec-4}
We will now show some exponential concentration inequalities for the Galerkin approximations  (\ref{eq: galerkin approx}) and for stochastic Leray solutions to (\ref{SNS-1}) in the sense of Definition \ref{weak-solution}. We recall the definition of the topology on $\XX$ in (\ref{eq: def X}).
\begin{proposition}
	\label{exponential3} There exists a functional ${F}:\mathbb{X}\to [0,\infty]$ with compact sublevel sets, a function $C:[0,\infty)\times [0,\infty)\times [0,\infty) \to (0,\infty)$ and an explicit $\kappa=\kappa(\eta)$ such that the following holds. Whenever $\ED{U}$ is a stochastic Leray solution to (\ref{SNS-1}) and $\EDM{U}$ is a solution to (\ref{eq: galerkin approx}), defined on a stochastic basis $(\Omega, \mathcal{F}, (\mathcal{F}_t)_{t\ge 0}, \mathbb{P}), W$ and with initial data $\EDM{U}(0), \ED{U}$ satisfying, for some $\eta>0$, \begin{equation}\label{eq: initial moment} \epsilon \log \mathbb{E}\left[\exp\left(\eta \|\ED{U}(0)\|_H^2 \right) \right]\le z, \qquad \epsilon \log \mathbb{E}\left[\exp\left(\eta \|\EDM{U}(0)\|_H^2 \right) \right]\le z\end{equation} then it holds that \begin{equation}\label{eq: FET 1}
		\epsilon \log \mathbb{E}\left[\exp\left(\kappa(\eta) F(\ED{U}) \right) \right]\le C(\epsilon\|A^{1/2}\sqrt{Q_\delta}\|_\mathrm{HS}^2,\eta,z)
	\end{equation} respectively \begin{equation}\label{eq: FET 2} \epsilon \log \mathbb{E}\left[\exp\left(\kappa(\eta) F(\EDM{U}) \right) \right]\le C(\epsilon\|P_mA^{1/2}\sqrt{Q_\delta}\|_\mathrm{HS}^2,\eta,z).
	\end{equation} \end{proposition} Thanks to the compactness of the sets $\mathcal{K}_n:=F^{-1}([0,n])$, the conclusions can immediately be postprocessed into the following various tightness properties. \begin{corollary}\label{ET'}
	For all $\epsilon, \delta, m$, let $\ED{U}, \EDM{U}$ be solutions to (\ref{SNS-1}, \ref{eq: galerkin approx}) respectively, with initial data satisfying, for some $\eta>0$, \begin{equation}
\sup_\epsilon \epsilon \log \mathbb{E}\left[\exp\left(\eta \|\ED{U}(0)\|_H^2\right) \right]<\infty; \qquad \sup_\epsilon \sup_m \epsilon \log \mathbb{E}\left[\exp\left(\eta \|\EDM{U}(0)\|_H^2\right) \right]<\infty.
	\end{equation} Let $\ED{\mu}, \EDM{\mu}$ be the laws of $\ED{U}, \EDM{U}$ on $\mathbb{X}$. Then: \begin{enumerate}[label=\roman*)]
		\item for any fixed $\epsilon, \delta>0$, $\EDM{\mu}$ satisfies \begin{equation}\label{eq: tightness m} \limsup_n \sup_m \EDM{\mu}(\mathcal{K}_n^\mathrm{c})=0\end{equation} where $\mathcal{K}_n:=F^{-1}([0,n])$, and hence the laws $\EDM{\mu}$ are tight on $\XX$. Moreover, for all $p\ge 1$, \begin{equation}\label{eq: integrability tightness}
			\sup_m \mathbb{E}\left(\sup_{t\le T} \|\EDM{U}(t)\|_H^{2p}\right)<\infty;
		\end{equation} \item if $(\epsilon, \delta(\epsilon))$ satisfy (\ref{scale-standing}), then we have the exponential tightness \begin{equation}
	\label{eq: exp tightness galerkin} \limsup_{n\to \infty} \limsup_{\epsilon\to 0}  \sup_m \left(\epsilon\log \EDM{\mu}(\mathcal{K}_n^\mathrm{c})\right) = -\infty
\end{equation} and \begin{equation}
	\label{eq: exp tightness wo galerkin} \limsup_{n\to \infty} \limsup_{\epsilon\to 0}  \left(\epsilon\log \ED{\mu}(\mathcal{K}_n^\mathrm{c})\right) = -\infty;
\end{equation} \item if $(\epsilon, \delta(\epsilon), m(\epsilon))$ satisfy (\ref{scale-standing-galerkin}), then \begin{equation}
	\label{eq: exp tightness galerkin galerkin} \limsup_{n\to \infty} \limsup_{\epsilon\to 0}  \left(\epsilon\log {\mu}^\epsilon_{\delta(\epsilon), m(\epsilon)}(\mathcal{K}_n^\mathrm{c})\right) = -\infty.
\end{equation} \end{enumerate}
\end{corollary} \begin{remark}\label{rmk: generalise ET proof} Throughout this section, we will assume the notation of Proposition \ref{exponential3}.  To avoid writing very similar arguments, we will present the details of the argument for $\EDM{U}$, based on the energy inequality (\ref{m-energy}). The arguments for $\ED{U}$ will be identical, up to replacing $\|P_m A^{1/2}\sqrt{Q_\delta}\|_\mathrm{HS}^2$ by $\|A^{1/2}\sqrt{Q_\delta}\|_\mathrm{HS}^2$. \end{remark} \begin{proof}[Proof of Proposition \ref{exponential3} and Corollary \ref{ET'}] Let us first define some function spaces. We define divergence-free Sobolev spaces $H^\gamma(\mathbb{T}^3)$ to be the completion of $\mathcal{X}_{\rm d.f.}$ under the norm $$ \|\varphi\|_{H^\gamma(\mathbb{T}^3)}^2:=\sum_{\zeta \in \mathcal{B}}(1+\lambda_\zeta)^\gamma \langle e_\zeta, \varphi\rangle^2.$$ For $\alpha\in (0,1)$ and a function $u:[0,T]\to H^{-\gamma}(\mathbb{T}^3)$, we write $ [u]_{\dot{C}^\alpha([0,T],H^{-\gamma}(\mathbb{T}^3))}$ for the homogenous H\"older seminorm $$ [u]_{\dot{C}^\alpha([0,T],H^{-\gamma}(\mathbb{T}^3))}:=\sup_{0\le s<t\le T}\frac{\|u(t)-u(s)\|_{H^{-\gamma}(\mathbb{T}^3)}}{|t-s|^\alpha}. $$ With this notation, we fix $\alpha\in (0,\frac12), \gamma>\frac52$, and take $F$ to be the functional \begin{equation}
	\label{eq: prototype mathcalf} F(u):=\|u\|^2_{L^2([0,T];V)}+\|u\|^2_{L^\infty([0,T],H)} + \|u\|_{W^{\alpha, 2}([0,T],H^{-\gamma}(\mathbb{T}^3))}+[u]_{\dot{C}^\alpha([0,T],H^{-\gamma}(\mathbb{T}^3))}.
\end{equation}  We first remark on the compactness of the sublevel sets. By definition of $\XX$, compactness of a subset $\mathcal{K}\subset \XX$ is equivalent to of compactness of $\mathcal{K}\subset L^2([0,T],H)$, (weak) compactness in $A\subset L_\mathrm{w}^2([0,T],H \cap V)$ and compactness $\mathcal{K}\subset C([0,T],(H,\mathrm{w}))$; each of these will follow from classical theorems. The spaces $H\cap V\hookrightarrow H \hookrightarrow H^{-\gamma}(\mathbb{T}^3)$ are reflexive, and by Rellich-Kondrachov, the embedding $H\cap V\hookrightarrow H$ is compact. We are therefore in the setting of the Aubin-Lions lemma \cite{aubin1963theoreme}, so the embedding
\begin{equation*}
L^2([0,T];H\cap V)\cap W^{\alpha,2}([0,T];H^{-\gamma}(\mathbb{T}^3)) \hookrightarrow L^2([0,T],H)
\end{equation*} is compact, and hence sublevel sets $F^{-1}([0,M])$ are compact in $L^2([0,T],H)$. On the set $F^{-1}([0,M])$, the control on $\|u\|_{L^\infty([0,T],H)}, \|u\|_{L^2([0,T],V)}$ imply a control on the norm $\|u\|_{L^2([0,T],H\cap V)}$, so $F^{-1}([0,M])$ is bounded in $L^2([0,T],H\cap V)$, and hence precompact in the weak topology by the Banach-Alaoglu theorem; compactness follows on observing the lower semicontinuity of $F$ with respect to this topology. Finally, we observe that the topology of $C([0,T],(H,w))$ restricted to $F^{-1}([0,M])$ coincides with that induced by $d(u,v):=\sup_{t\le T}\|u(t)-v(t)\|_{H^{-\gamma}(\mathbb{T}^3)}$, since the $H^{-\gamma}(\mathbb{T}^3)$-norm induces the weak topology on the (weakly) compact subset $\{u\in H: \|u\|_H\le M\}$. Using the Arzel\`a-Ascoli \cite[Theorem 4.10]{KS}, the set $F^{-1}([0,M])$ is compact in the metric $d$, and hence is compact in the topology of $C([0,T],(H,w))$.  \bigskip \\ We will now discuss the argument for (\ref{eq: FET 1}); the case for (\ref{eq: FET 2}) is identical. We will prove estimates of the same form for the first two terms appearing in $F(\EDM{U})$ in Lemma \ref{exponential1} with functions $\kappa_i(\eta), i=1,2$, and upper bounds $C_i$, and a similar estimate will be proven for the third and fourth term in Lemmas \ref{exponential2} - \ref{exponential2'} with new functions $\kappa_i(\eta), C_i, i=3,4$. Setting $\kappa(\eta):=\min(\kappa_i(\eta),1\le i\le 4)/4>0$, we use H\"older's inequality to decompose the left-hand side of (\ref{eq: FET 1}) into the simpler terms \begin{equation}\begin{split}\label{eq: F est}
	\epsilon \log \mathbb{E}\left[\exp\left(\kappa(\eta) F(\EDM{U}) \right) \right] & \le \frac\epsilon 4 \log \mathbb{E}\left[\exp\left(4\kappa(\eta )\|\EDM{U}\|_{L^\infty([0,T],H)}^2 \right) \right] \\ & \hspace{1cm} + \frac\epsilon 4 \log \mathbb{E}\left[\exp\left(4\kappa(\eta )\|\EDM{U}\|_{L^2([0,T],V)}^2 \right) \right] \\ & \hspace{1cm} + \frac\epsilon 4 \log \mathbb{E}\left[\exp\left(4\kappa(\eta )\|\EDM{U}\|_{W^{\alpha, 2}([0,T],H^{-\gamma}(\mathbb{T}^3))} \right) \right] \\ & \hspace{1cm} + \frac\epsilon 4 \log \mathbb{E}\left[\exp\left(4\kappa(\eta )[\EDM{U}]_{\dot{C}^{\alpha}([0,T],H^{-\gamma}(\mathbb{T}^3))} \right) \right].
\end{split} \end{equation} Since we chose $\kappa(\eta)$ so that $4\kappa(\eta)\le \kappa_i(\eta), 1\le i\le 3$, we can replace the three terms on the right-hand side by the functions $C_1, C_2, C_3, C_4$ produced in Lemmata \ref{exponential1}, \ref{exponential2}, \ref{exponential2'}, and the desired bound (\ref{eq: FET 1}) holds with the function $C:=(C_1+C_2+C_3+C_4)/4$. \bigskip \\ In Corollary \ref{ET'}, the main points (\ref{eq: tightness m}, \ref{eq: exp tightness galerkin}, \ref{eq: exp tightness wo galerkin}, \ref{eq: exp tightness galerkin galerkin}) follow from (\ref{eq: FET 1}, \ref{eq: FET 2}) by Chebyshev's inequality, noting in (i) that the right-hand side of (\ref{eq: FET 1}) remains bounded as $m\to \infty$ with $\epsilon, \delta$ fixed, and in (ii), (iii) that the right-hand sides of (\ref{eq: FET 1}, \ref{eq: FET 2}) remain bounded under respectively (\ref{scale-standing}) in (ii), and under (\ref{scale-standing-galerkin}) in (iii) thanks to Lemma \ref{lemma: eigenvalue sum calculation}. The point (\ref{eq: integrability tightness}) follows from (\ref{eq: FET 1}) since $F(u)\ge \sup_{0\le t\le T} \|u(t)\|_H$.    \end{proof}

The remainder of the section is dedicated to the proofs of the estimates postponed from (\ref{eq: F est}).  We first prove the estimates for $L^\infty_tH$ and $L^2_tV$.

\begin{lemma}\label{exponential1}
Under the hypotheses of Proposition \ref{exponential3}, we have the estimates
\begin{equation}\label{eq: E11}
 \epsilon\log\mathbb{E}\Big(\exp\Big\{\frac{\min(\eta,\frac12)}{\epsilon}\|\EDM{U}\|_{L^2([0,T];V)}^2\Big\}\Big)\leq z+\frac{T}{2}\epsilon\|P_mA^{1/2}\sqrt{Q_\delta}\|_\mathrm{HS}^2
\end{equation} and \begin{equation}\label{eq: E12}\epsilon\log\mathbb{E}\Big(\exp\Big\{\frac{\min(\eta,\frac12)}{2\epsilon}\sup_{t\le T} \|\EDM{U}(t)\|_{H}^2\Big\}\Big)\leq \frac{z}{2}+\frac{T}{4}\epsilon\|P_mA^{1/2}\sqrt{Q_\delta}\|_\mathrm{HS}^2. \end{equation}
The same holds for $\ED{U}$, replacing $P_mA^{1/2}\sqrt{Q_\delta}$ by $A^{1/2}\sqrt{Q_\delta}$.
\end{lemma}
\begin{proof}
We will write out the proof only in the case $\EDM{U}$; the case for $\ED{U}$ is identical.  Applying It\^o's formula to $\|\EDM{U}\|_H^2$,  we find that, $\mathbb{P}^{}$-almost surely,
\begin{align}\label{energyineq03}
\frac{1}{2}\|\EDM{U}(t)\|_{H}^2\leq&\frac{1}{2}\|\EDM{U}(0)\|_{H}^2-\int_{0}^{t}\| \EDM{U}(r)\|_{V}^2dr\notag\\
&-\sqrt{\epsilon}\int_0^t\langle \EDM{U}(r),P_m \nabla\cdot dW_{\delta}(r)\rangle+\frac12 \epsilon \|P_m A^{1/2}\sqrt{Q_\delta}\|_\mathrm{HS}^2t.
\end{align}
Let us set $\theta:=\min(\eta, \frac12)$. Applying the exponential function to (\ref{energyineq03}), it follows that, $\mathbb{P}$-almost surely,
\begin{align}\label{expineq}
&\exp\Big\{\frac{\theta}{\epsilon}\Big(\frac{1}{2}\int_{0}^{t}\| \EDM{U}(r)\|_{V}^2dr+\frac12 \|\EDM{U}(t)\|_{H}^2\Big)\Big\}\\ & \hspace{1cm}\le \exp\Big\{\frac{1}{\epsilon}(\theta \|U^\epsilon_{\delta}(0)\|_{H}^2+\frac{\theta}{2}\|P_m A^{1/2}\sqrt{Q_\delta}\|_\mathrm{HS}^2 T)-\frac{\theta}{\sqrt{\epsilon}}\int_{0}^{t}\langle \EDM{U}(r),\nabla\cdot dW_{\delta}(r)\rangle-\frac{\theta}{2\epsilon}\int_0^t\| \EDM{U}\|_{V}^2dr\Big\}.\notag
\end{align}
The quadratic variance of the martingale $\frac{\theta}{\sqrt{\epsilon}}\int_{0}^{t}\langle \EDM{U},\nabla\cdot dW_\delta(r)\rangle$ is bounded by
\begin{align*}
\Big\langle-\frac{\theta}{\sqrt{\epsilon}}\int_{0}^{\cdot}\langle \EDM{U},\nabla\cdot dW_{\delta}(r)\rangle\Big\rangle_t & = \frac{\theta^2}{\epsilon}\sum_{\zeta\in \mathcal{B}_m, 1\le i\le 3}\Big\langle\int_0^{\cdot}\langle \partial_i\EDM{U},\sqrt{Q_{\delta}}e_\zeta\rangle d\beta^{\zeta,i}(r)\Big\rangle_t\\
 & =\frac{\theta^{2}}{\epsilon}\sum_{\zeta \in \mathcal{B}_m, 1\le i\le 3}\int_0^t\langle\sqrt{Q_{\delta}}\partial_i{\EDM{U}(r)},e_\zeta\rangle^2dr\\
 & =\frac{\theta^{2}}{\epsilon}\int_{0}^{t}\|\sqrt{Q_{\delta}}\nabla \EDM{U}\|_{\mathcal{M}}^2dr\le \frac{\theta}{2\epsilon}\int_{0}^{t}\|\EDM{U}\|_{V}^2dr
\end{align*}
where the last line follows because $\theta\le \frac12$. Since the quadratic variation is at most half of the negative term, the properties of exponential martingales (see \cite{KS}) show that $(M^\epsilon_t)^p$ is a nonnegative, continuous supermartingale starting at 1, for any $p\in [0,1]$, where
\begin{equation}\label{p-supermartingale}
M^{\epsilon}_t:=\exp\Big\{-\frac{\theta}{\sqrt{\epsilon}}\int_0^t\langle \EDM{U}(r),\nabla\cdot dW_{\delta}(r)\rangle-\frac{\theta}{2\epsilon}\int_0^t\|\EDM{U}(r)\|_{V}^2dr\Big\}.
\end{equation}
 Returning to (\ref{expineq}) and taking expectations at $t=T$, we use the case $p=1$ and the hypothesis (\ref{eq: initial moment}) to estimate the right-hand side, to obtain \begin{equation}
	\begin{split}\label{eq: E11 follows}
		\epsilon \log \mathbb{E}\left(\exp\left(\frac{\theta}{2\epsilon}\int_0^T \|\EDM{U}(r)\|_V^2 dr\right)\right) \le z+\frac{T}{2}\epsilon\|P_mA^{1/2}\sqrt{Q}_\delta\|_{\rm HS}^2
	\end{split}
\end{equation} from which (\ref{eq: E11}) follows. To control the $L^\infty_tH$ norm, we apply Doob's maximal inequality to the submartingale $$ Z^\epsilon_t:=\exp\Big\{-\frac{\theta}{4\sqrt{\epsilon}}\int_0^t\langle \EDM{U}(r),\nabla\cdot dW_\delta(r)\rangle\Big\}$$ and control the expectation of the square by \begin{equation}\begin{split}
	\mathbb{E}\left[\sup_{t\le T}(Z^\epsilon_t)^2\right]\le 4\mathbb{E}\left[(Z^\epsilon_T)^2\right] &=4\mathbb{E}\left[(M^\epsilon_T)^{1/2}\exp\left(\frac{\theta}{4\epsilon}\int_0^T\|\EDM{U}(r)\|_V^2 dr\right)\right] \\ & \le 4\mathbb{E}[M^{\epsilon}_T]^{1/2}\mathbb{E}\left[\exp\left(\frac{\theta}{2\epsilon}\int_0^T \|\EDM{U}(r)\|_V^2 dr\right)\right]^{1/2} \\ & \le 4\exp\left(\frac{1}{2\epsilon}\left(z+\frac{T\epsilon}{2}\|P_mA^{1/2}Q_\delta\|_{\rm HS}^2\right)\right)\end{split}
\end{equation} where we have used $\mathbb{E}M^{\epsilon}_T\le 1$ and (\ref{eq: E11 follows}). Substituting this bound into (\ref{expineq}) gives an upper bound for the $L^\infty_tH$ norm on the left-hand side, and (\ref{eq: E12}) is proven.\end{proof}
Since it is closely related to the previous proof, we now give a lemma which will eventually allow us to prove the improved upper bound in Theorem \ref{thm: improved ub}. \begin{lemma}\label{lemma: improved UB'}
	Under the same hypotheses as Theorem \ref{thrm: LDP}, under the scaling (\ref{scale-standing}), there exists a function $S:[0,\infty)\to [0,\infty)$, vanishing only at $0$, such that, for all $u\in \XX$, there exists an open set $\mathcal{U}\ni u$ such that \begin{equation}
		\limsup_{\epsilon\to 0} \epsilon\log \ED{\mu}(\mathcal{U})\le -S(\mathcal{E}^\star(u))
	\end{equation} where $\mathcal{E}^\star(u)$ is defined by (\ref{eq: Estar}). The same holds for $\EDM{\mu}$ under the scale (\ref{scale-standing-galerkin}).
\end{lemma}
\begin{proof}
We argue the case for $\ED{U}, \ED{\mu}$; the case for $\EDM{U}, \EDM{\mu}$ is identical. Let us fix $u\in \XX$ and shorten notation by setting $\theta:=\frac14\mathcal{E}^\star(u)$. The case $\mathcal{E}^\star(u)=0$ is trivial, by setting $S(0)=0, \mathcal{U}=\mathbb{X}$, so we will consider only $\mathcal{E}^\star(u)>0$ in the sequel. By definition, there exists $0\le t\le T$ such that \begin{equation}\label{eq: EI violation}
		\frac12\|u(t)\|_H^2+\int_0^t\|u(s)\|_V^2ds>\frac12 \|u^\star_0\|_H^2+3\theta.
	\end{equation} For this particular $t$, let us now take $\mathcal{U}$ to be the set \begin{equation}\label{eq: choice of U}
		\mathcal{U}:=\left\{v\in \XX: \frac12\|v(t)\|_H^2+\int_0^t\|v(s)\|_V^2 ds > \frac12\|u_0^\star\|_H^2+3\theta\right\}
	\end{equation} which is open in $\XX$ thanks to lower semicontinuity of $v\mapsto \|v(t)\|_H, v\mapsto \int_0^t \|v(s)\|_V^2$ in the topology of $\XX$. We now use the trajectorial energy inequality (\ref{eq: TEI}), which is valid for $\ED{U}$ by definition. Let us write $M^\epsilon_t$ for the martingale $$ M^\epsilon_s:=-\sqrt{\epsilon} \int_0^s\langle \ED{U}, \nabla\cdot dW_\delta(r)\rangle.$$ Using Lemma \ref{lemma: eigenvalue sum calculation} and the scaling hypothesis (\ref{scale-standing}), we may choose $\epsilon_0>0$ small enough that, for all $0<\epsilon<\epsilon_0$, $\epsilon \|A^{1/2}\sqrt{Q_\delta}\|_{\rm HS}^2T<\theta$. For some $\lambda>0$ to be chosen, and all $\epsilon<\epsilon_0$, we may break up the event $\{\ED{U}\in \mathcal{U}\}$ as  \begin{equation}\begin{split}
		\label{eq: break up violation of EI} \left\{\ED{U}\in \mathcal{U}\right\} &\subset \left\{\|\ED{U}\|^2_{L^2([0,T],V)}>\lambda\right\} \cup \left\{M^\epsilon_t>\theta, \|\ED{U}\|_{L^2([0,T],V)}^2\le \lambda\right\} \\&\hspace{5cm} \cup \left\{\|\ED{U}(0)\|_H^2>\|u^\star_0\|_H^2+\theta\right\}\end{split}
	\end{equation} where the reference time $t \in [0,T]$ in the second event is the same one as in the definition of $\mathcal{U}$ in (\ref{eq: choice of U}) according to (\ref{eq: EI violation}).  We now estimate the rate of decay of the probabilities of these events for general $\lambda>0$ in terms of the cumulant generating function $Z(\eta)$ defined at (\ref{eq: energy estimate at gaussian ID}), and will see that the optimal exponential rate, given by optimising in $\lambda$, is strictly positive when $\theta>0$. Returning to the definition of $\theta:=\frac13\mathcal{E}^\star(u)$, this gives a definition of the function $S$, implicit in the function $Z$, and concludes the claims of the lemma. \medskip \\ For the first event $\{\|\ED{U}\|^2_{L^2([0,T],V)}>\lambda\}$, we repeat the arguments leading to (\ref{eq: E11}) for the stochastic Leray solution $\ED{U}$. Taking the limit $\epsilon\to 0$, the second term on the right-hand side vanishes, and we find, for all $\eta\in (0, \frac12]$, \begin{equation}
		\limsup_{\epsilon}\epsilon \log \mathbb{E}\left(\exp\left\{\frac{\eta}{\epsilon}\|\ED{U}\|^2_{L^2([0,T],V)}\right\}\right) \le Z(\eta)
	\end{equation} so that, using a Chebyshev estimate, \begin{equation}
		\limsup_{\epsilon\to 0} \epsilon \log\PP\left(\|\ED{U}\|^2_{L^2([0,T],V)}>\lambda\right) \le -\lambda\eta+Z(\eta).
	\end{equation} To control the probability of the second event, we recall that the quadratic variation of $M^\epsilon$ is controlled by $$ \langle M^\epsilon\rangle_s\le \epsilon\int_0^s \|\ED{U}(r)\|_V^2dr \le \epsilon \|\ED{U}\|_{L^2([0,T],V)}^2 $$ so the probability of second event may be estimated by elementary calculations to find $$ \mathbb{P}\left(M^\epsilon_t>\theta, \|\ED{U}\|_{L^2([0,T],V)}^2\le \lambda\right) \le \mathbb{P}\left(M^\epsilon_t>\theta, \langle M^\epsilon\rangle_t\le \lambda \epsilon\right)\le e^{-\frac{\theta^2}{2\lambda}}.$$ For the final term, which concerns the large deviations of the initial data, we use the definition of the function $Z$ and a Chebyshev estimate: for any $\eta'>0$, \begin{equation}
		\limsup_{\epsilon\to 0}\epsilon\log \mathbb{P}\left(\|\ED{U}(0)\|_H^2>\|u^\star_0\|_H^2+\theta\right) \le Z(\eta')-\eta'(\|u_0^\star\|_H^2+\theta)
	\end{equation} Combining the previous three displays, it follows that, for any $\eta, \eta', \lambda>0$, \begin{equation}\label{eq: penalise violations of EI 1} \limsup_{\epsilon\to 0}\epsilon \log \ED{\mu}(U) \le -\min\left(\lambda \eta -Z(\eta), \frac{\theta^2}{2\lambda}, \eta'(\|u_0^\star\|_H^2+\theta)-Z(\eta')\right)\end{equation} Since $\limsup_{\eta'\to 0} Z(\eta')/\eta'\le \|u_0^\star\|_H^2$ by (\ref{eq: energy estimate at gaussian ID''}), it follows that the right-hand side is strictly negative if we choose $\eta'>0$ small enough and $\lambda>0$ large enough. We can thus characterise $S$ by optimising over $\eta, \lambda, \eta' $ to find $$ S(4\theta):=\max_{\lambda>0} \min\left(Z^\star(\lambda), \frac{\theta^2}{2\lambda}, Z^\star(\|u^\star_0\|_H^2+\theta)\right) $$ where $Z^\star$ denotes the Legendre transform.   \end{proof}

Finally, we show the exponential estimates for the time-continuity, where the choices of parameter $\gamma>\frac52$, $\alpha\in (0,\frac12)$ will become important. We will first prove the estimate in $W^{\alpha, 2}([0,T],H^{-\gamma}(\mathbb{T}^3))$ and give details of the proof; the proof of the estimate in $\dot{C}^{\alpha}([0,T],H^{-\gamma}(\mathbb{T}^3))$ in Lemma \ref{exponential2'} is nearly identical, and only the necessary modifications will be discussed.
\begin{lemma}\label{exponential2}
With the notation and hypotheses of Proposition \ref{exponential3}, fix $\alpha\in(0,\frac12)$ and $\gamma>\frac52$. Then there exists an explicitable function $\kappa_3(\eta)>0$ and a constant $C=C(\alpha, \gamma, T)$ such that
\begin{equation}\label{wa2 est}
\epsilon\log\mathbb{E}\Big(\exp\Big\{\frac{\kappa_3(\eta)}{\epsilon}\|\EDM{U}\|_{W^{\alpha,2}([0,T];H^{-\gamma}(\mathbb{T}^3))}\Big\}\Big)\leq C+z+\epsilon\|P_mA^{1/2}\sqrt{Q_\delta}\|_\mathrm{HS}^2.
\end{equation}
The same holds for the stochastic Leray solution $\ED{U}$, still in the notation of and under all hypotheses in Proposition \ref{exponential3}.
\end{lemma}
\begin{proof}

As before, we argue only the case for $\EDM{U}$. For $\kappa_3=\kappa_3(\eta)>0$ to be chosen later, we multiply by $\kappa_3$ and take the ${W^{\alpha,2}([0,T];H^{-\gamma}(\mathbb{T}^3))}$-norm of the identity $$ \EDM{U}(t) =\EDM{U}(0) - \int_0^t A\EDM{U}(s)ds - \int_0^t P_m ((\EDM{U}(s)\cdot \nabla)\EDM{U}(s)) ds - \sqrt{\epsilon}\int_0^t P_m \nabla \cdot dW_\delta(dr).$$ As in the proof of Proposition \ref{exponential3}, we can decompose the norm into the norm of each term:
\begin{equation}\label{decompose}
\epsilon \log \mathbb{E}\exp\Big\{\frac{\kappa_3}{\epsilon}\|\EDM{U}\|_{W^{\alpha,2}([0,T];H^{-\gamma}(\mathbb{T}^3))}\Big\} \le \sum_{\ell=1}^4 \EDM{\mathcal{T}}(\ell) \end{equation} where the 4 terms are \begin{equation}
	\EDM{\mathcal{T}}(1):=\epsilon \log \mathbb{E}\exp\Big\{\frac{4\kappa_3}{\epsilon}\|\EDM{U}(0)\|_{W^{\alpha,2}([0,T];H^{-\gamma}(\mathbb{T}^3))}\Big\};
\end{equation} \begin{equation}
	\EDM{\mathcal{T}}(2):=\epsilon \log \mathbb{E}\exp\Big\{\frac{4\kappa_3}{\epsilon}\Big\|\int_0^{\cdot}\Delta \EDM{U} ds\Big\|_{W^{\alpha,2}([0,T];H^{-\gamma}(\mathbb{T}^3))}\Big\};
\end{equation}
\begin{equation}\EDM{\mathcal{T}}(3):=\epsilon\log \mathbb{E}\exp\Big\{\frac{4\kappa_3}{\epsilon}\Big\|\int_0^{\cdot}P_m (\EDM{U}\cdot\nabla)\EDM{U}ds\Big\|_{W^{\alpha,2}([0,T];H^{-\gamma}(\mathbb{T}^3))}\Big\}; \end{equation} \begin{equation}
	\EDM{\mathcal{T}}(4):= \epsilon\log \mathbb{E}\exp\Big\{\frac{4\kappa_3}{\sqrt{\epsilon}}\|P_m \nabla\cdot W_{\delta}\|_{W^{\alpha,2}([0,T];H^{-\gamma}(\mathbb{T}^3))}\Big\}.
\end{equation}
We now deal with the terms one by one. In $\EDM{\mathcal{T}}(1)$, there is no time dependence, so that $W^{\alpha, 2}$-norm is bounded by $2(1+\|\EDM{U}\|^2_H)$, and if we choose $8\kappa_3\le \eta$, we may use (\ref{eq: initial moment}) to find a bound $z+2$ on the first term. For the viscous term $\EDM{\mathcal{T}}(2)$, we can use the definition to check that $\|\Delta u\|_{H^{-1}(\mathbb{T}^3)}\le \|u\|_V$, and use Lemma \ref{exponential1} and Young inequality to see that
\begin{align}\label{es-1}
\EDM{\mathcal{T}}(2) & \leq\epsilon \log \mathbb{E}\exp\Big\{\frac{4\kappa_3}{\epsilon}\Big(\int_0^T\|\Delta \EDM{U}\|_{H^{-1}(\mathbb{T}^3)}^2ds\Big)^{1/2}\Big\}\\
& \le \epsilon \log \mathbb{E}\exp\Big\{\frac{2\kappa_3}{\epsilon}\Big(\int_0^T\| \EDM{U}\|_{V}^2ds+1\Big)\Big\} \le 2+z+T+\frac{T}{2}\|P_mA^{1/2}\sqrt{Q_\delta}\|_{\mathrm{HS}}^2.
\notag
\end{align}
 For the nonlinear term $\EDM{\mathcal{T}}(3)$, since $\alpha\in(0,\frac12)$, $\gamma>\frac32$, Kondrachov's embedding theorem implies that the inclusion $$W^{1,1}([0,T];L^1(\mathbb{T}^3))\subset W^{\alpha,2}([0,T];H^{-\gamma}(\mathbb{T}^3))$$ is continuous, and we observe that $P_m$ is nonexpansive in the norm of $H^{-\gamma}(\mathbb{T}^3)$. Using H\"older's inequality and allowing an absolute constant $c$ to change from line to line,
\begin{equation}\begin{split}\label{walpha-1}
\EDM{\mathcal{T}}(3) & \le \epsilon \log \mathbb{E}\exp\Big\{\frac{4\kappa_3}{\epsilon}\Big\|\int_0^{\cdot} (\EDM{U}\cdot\nabla)\EDM{U} ds\Big\|_{W^{\alpha,2}([0,T];H^{-\gamma}(\mathbb{T}^3))}\Big\} \\ & \le \epsilon \log \mathbb{E}\exp\Big\{\frac{c\kappa_3}{\epsilon}\Big\|\int_0^{\cdot} (\EDM{U}\cdot\nabla)\EDM{U} ds\Big\|_{W^{1,1}([0,T];L^1(\mathbb{T}^3))}\Big\}\\
& \leq \epsilon \log \mathbb{E}\exp\Big\{\frac{c\kappa_3}{\epsilon}\int_0^T\|(\EDM{U}\cdot\nabla)\EDM{U} \|_{L^1(\mathbb{T}^3)}ds\Big\} \\& \leq\epsilon \log \mathbb{E}\exp\Big\{\frac{c\kappa_3}{\epsilon}\int_0^T\|\EDM{U}\|_{H}\| \EDM{U}\|_{V}ds\Big\}.
\end{split}\end{equation} We now bound the final expression above by $$ \epsilon \log \mathbb{E}\exp \left\{\frac{c\kappa_3}{2\epsilon}\left( \|\EDM{U}\|_{L^\infty([0,T];H)}^2 + \|\EDM{U}\|_{L^2([0,T];V)}^2\right)\right\}$$ which we control by Lemma \ref{exponential1}, if we shrink $\kappa_3$ (if necessary) so that ${c\kappa_3}\le \min(\eta, \frac14)$ and separating the two terms with H\"older's inequality in the same way as in the proof of Proposition \ref{exponential3}. All together, we conclude that, \begin{align}\label{walpha-1'}
\EDM{\mathcal{T}}(3) \le 2+z+\epsilon \|P_mA^{1/2}\sqrt{Q_\delta}\|_\mathrm{HS}^2.
\end{align}
We finally turn to the noise term $\EDM{\mathcal{T}}(4)$. For $\delta>0$, $\gamma>\frac52$, the maps $\sqrt{Q_{\delta}}$, $P_m$ are nonexpansive in $H^{-\gamma}(\mathbb{T}^3)$, whence $P_m \sqrt{Q_{\delta}}$ will only shrink the norm of  $\|\cdot\|_{W^{\alpha,2}([0,T];H^{-\gamma}(\mathbb{T}^3))}$. By using Young's inequality, and for $c>0$ to be chosen later,
\begin{align*}
\EDM{\mathcal{T}}(4)\le &\epsilon \log \mathbb{E}\exp\Big\{\frac{4\kappa_3}{\sqrt{\epsilon}}\|\sqrt{Q_{\delta}}W\|_{W^{\alpha,2}([0,T];H^{-\gamma+1}(\mathbb{T}^3))}\Big\}\notag\\
\le &\epsilon \log \mathbb{E}\exp\Big\{\frac{4\kappa_3}{\sqrt{\epsilon}}\|W\|_{W^{\alpha,2}([0,T];H^{-\gamma+1}(\mathbb{T}^3))}\Big\}\\
\leq&\epsilon \log \mathbb{E}\exp\Big\{c\|W\|_{W^{\alpha,2}([0,T];H^{-\gamma+1}(\mathbb{T}^3))}^2+\frac{16\kappa_3^2}{c\epsilon}\Big\}.
\end{align*}
For $\gamma>\frac{5}{2}$, $\alpha\in(0,\frac{1}{2})$, by the H\"older regularity of Brownian motion, $W$ induces a Gaussian measure on $W^{\alpha,2}([0,T];H^{-\gamma+1}(\mathbb{T}^3))$. Fernique's theorem therefore implies that, for sufficiently small $c>0$, $$ L(\alpha):=\log \mathbb{E}\exp\left\{c\|W\|^2_{W^{\alpha,2}([0,T],H^{-\gamma+1}(\mathbb{T}^3))}\right\}<\infty. $$ We note that this last display is independent of $\epsilon, \delta, m$. Combining, we have proven that
\begin{equation}
		\label{es-3}
\EDM{\mathcal{T}}(4) \leq  \frac{16\kappa_3^2}{c}+L(\alpha)\end{equation}
Gathering the estimates (\ref{es-1}), (\ref{walpha-1'}) and (\ref{es-3}) and returning to (\ref{decompose}), we conclude (\ref{wa2 est}) as claimed. \end{proof}

\begin{lemma}\label{exponential2'}
Under the same notation and hypotheses of Lemma \ref{exponential2} and Proposition \ref{exponential3}, there exists an explicitable function $\kappa_4(\eta)>0$ and a constant $C=C(\alpha, \gamma, T)$ such that
\begin{equation}\label{ca est}
\epsilon\log\mathbb{E}\Big(\exp\Big\{\frac{\kappa_3(\eta)}{\epsilon}[\EDM{U}]_{\dot{C}^{\alpha}([0,T];H^{-\gamma}(\mathbb{T}^3))}\Big\}\Big)\leq C+z+\epsilon\|P_mA^{1/2}\sqrt{Q_\delta}\|_\mathrm{HS}^2.
\end{equation}
The same holds for the stochastic Leray solution $\ED{U}$, still in the notation of and under all hypotheses in Proposition \ref{exponential3}.
\end{lemma}

\begin{proof}[Sketch Proof] We follow the same ideas as the proof of Lemma \ref{exponential2} above, arguing only the case for $\EDM{U}$. We start from the analogue of (\ref{decompose}), to decompose into the same 4 terms $\EDM{\mathcal{T}}(k), 1\le k\le 4$, now measuring the time continuity in the H\"older rather than Sobolev sense. Of these, $\EDM{\mathcal{T}}(1)$ vanishes, because there is no time dependence in the initial data. For the contributions from the dissipation and convection terms, one estimates, for all $\varphi \in H^\gamma(\mathbb{T}^3)$ and all $0\le s<t\le T$, \begin{equation}
	\left|\left\langle \varphi, \int_s^t A\EDM{U}(r)dr\right\rangle \right| \le \|\varphi\|_{H^1(\mathbb{T}^3)}(t-s)^{1/2}\|\EDM{U}\|_{L^2([0,T],V)} \end{equation} and using the Sobolev embedding $H^\gamma(\mathbb{T}^3)\hookrightarrow L^\infty(\mathbb{T}^3)$, for some absolute constant $c$, \begin{equation}
		\left|\left\langle \varphi, \int_s^t P_m(\EDM{U}\cdot \nabla)\EDM{U}(r)dr\right\rangle \right| \le c\|\varphi\|_{H^\gamma(\mathbb{T}^3)}(t-s)^{1/2} \left(\|\EDM{U}\|_{L^\infty([0,T],H)}^2+\|\EDM{U}\|_{L^2([0,T],V)}^2\right).
	\end{equation} Using the duality of $H^{\gamma}(\mathbb{T}^3), H^{-\gamma}(\mathbb{T}^3)$ and that $\alpha<\frac12$, these imply bounds on $\EDM{\mathcal{T}}(2), \EDM{\mathcal{T}}(3)$ using Lemma \ref{exponential1}, up to absorbing constants by making $\kappa_4=\kappa_4(\eta)$ smaller. For the new definition of $\EDM{\mathcal{T}}(4)$, we observe that $\nabla \cdot W_\delta$ defines a Gaussian random variable on $C^{\alpha}([0,T],H^{-\gamma}(\mathbb{T}^3))$ for the values $\alpha \in (0,\frac12), \gamma>\frac52$ imposed. On this space, the H\"older seminorm is induced by a countable set of linear functionals of unit norm: $$ [u]_{\dot{C}^\alpha([0,T],H^{-\gamma}(\mathbb{T}^3))}=\sup\left(\frac{\langle \varphi_n, u(t_m)-u(s_m)\rangle}{|t_m-s_m|^{\alpha}}: m, n\ge 1 \right)$$
	where $\{\varphi_n, n\ge 1\}$ is a countable dense subset of $H^{\gamma}(\mathbb{T}^3)$ and $(t_m, s_m)$ is an enumeration of distinct rational times $0\le s_m<t_m\le T$. In particular, we may apply Fernique's theorem as in \cite[Example 3, Theorem 1]{CIS} to obtain, for some absolute constant $c$, \begin{equation}
		\mathbb{E}\left[\exp\left(c[\nabla\cdot W]^2_{\dot{C}^\alpha([0,T],H^{-\gamma}(\mathbb{T}^3))}\right)\right]<\infty
	\end{equation} which implies an estimate on $\EDM{\mathcal{T}}(4)$ by the same argument as before.
\end{proof}

\section{Existence of Stochastic Leray Solutions}\label{sec: exist}
In this section, we will prove Proposition \ref{prop: main existence}, with the definitions we have made precise in Definition \ref{weak-solution}. We will use a tightness argument in the space $\XX$, similar to \cite{FG95}, using the estimates already established in the previous section.
\begin{proposition}\label{existence}
Let $\epsilon, \delta>0$. For each $\epsilon, \delta>0$, let ${\ED{\tilde{U}}}(0)$ be an $H$-valued random variable satisfying, for some constant $\eta>0$,  \begin{equation}\label{eq: initial data existence}
	 \mathbb{E} \left[\exp\left(\eta\|\ED{\tilde{U}}(0)\|_H^2/\epsilon\right)\right]<\infty.
\end{equation} Then, on a new stochastic basis $(\Omega, \mathcal{F}, (\mathcal{F}_t)_{t\ge 0}, \mathbb{P})$ equipped with a divergence-free white noise $W$, there exists a stochastic Leray solution $(\ED{U}(t))_{t\ge 0}$ to (\ref{SNS-1}) such that $\ED{U}(0)$ has the law of $\ED{\tilde{U}}(0)$ and is independent of $W$.
\end{proposition} \begin{proof} On the same stochastic basis as $\ED{\tilde{U}}(0)$, let $\EDM{\tilde{U}}(0):=P_m \ED{\tilde{U}}(0)$, and extending the stochastic basis if necessary, let $\tilde{W}$ be a divergence-free white noise, independently of $\ED{\tilde{U}}(0)$, and let $\EDM{\tilde{U}}$ be the solutions to (\ref{eq: galerkin approx}) starting from $\EDM{\tilde{U}}(0)$ and driven by $\tilde{W}$.  The argument will proceed using the usual argument of compactness and identification of limits. We consider the tuples $(\EDM{\tilde{U}}, \tilde\beta, \tilde M^m)$ of solutions to (\ref{eq: galerkin approx}), together with the Wiener process $\tilde \beta=\{\tilde{\beta}^{\zeta, i}: \zeta\in \mathcal{B}, 1\le i\le 3\}$ corresponding to the driving white noise $\tilde W$, and martingale $\tilde M^m$ appearing in the energy inequality \begin{equation}
	\tilde{M}^m(t)=-\int^t_0\sqrt{\epsilon}\langle \EDM{\tilde{U}},\nabla\cdot d\tilde{W}_{\delta}(s)\rangle.
\end{equation}Using Corollary \ref{ET'}i), the construction of $\EDM{\tilde{U}}(0)$ and using Burkholder-Davis-Gundy together with (\ref{eq: integrability tightness}), (\ref{gainonederivative}), the tuples $(\EDM{\tilde{U}}, \tilde{\beta},\tilde{M}^m)$ are tight in the space $\XX^+\times C([0,T],\mathbb{R}^{\mathcal{B}\times \{1,2,3\}})\times C([0,T])$ which justifies moving to a subsequence, without relabelling, converging in distribution relative to this topology. We use Skorokhod's representation theorem to find a new probability space $(\Omega, \mathcal{F}, (\mathcal{F}(t))_{t\ge 0}, \mathbb{P})$ and realisations $(\EDM{U}, \beta^m, M^m)$ which converge almost surely to a limit $(\ED{U}, \beta, M)$ and, using same estimates for $M^m$ as in tightness, one can also gain convergence $M^m\to M$ in the norm of $L^p(\mathbb{P};C([0,T]))$ for any $1\le p<\infty$. \\ \\  Using the L\'evy characterisation,  $\beta$ is a collection of independent Brownian motions, and it is standard to identify $\ED{U}$ as a weak solution to (\ref{SNS-1}) in the sense of equation (\ref{weaksense}), see \cite{FG95}. To conclude, we only need to show the trajectorial energy inequality. To ease notation, we will not reiterate the dependence of constants on the (fixed) parameters $\epsilon, \delta$ in the sequel. \\ \\ We start from the approximate energy equality (\ref{m-energy}) for $\EDM{{U}}$, which implies that for every $\psi\in L^\infty(\mathbb{P};L^\infty([0,T],\mathbb{R}_+))$,
\begin{align}\label{Itoineq}
{\mathbb{E}}\Big[\int_0^T \psi(t) \Big(&\frac{1}{2}\|\EDM{U}(t)\|_H^2-\frac{1}{2}\|\EDM{U}(0)\|_H^2+\int_0^t\| \EDM{U}(s)\|_V^2ds\notag\\
&+\int^t_0\sqrt{\epsilon}\langle \EDM{U},\nabla\cdot dW^{m}_\delta(s)\rangle-\frac{\epsilon}{2}\|P_m\circ A^{1/2}\circ\sqrt{Q_{\delta}}\|_{HS}^2t\Big)dt\Big]=0
\end{align}
where $W^m$ is the white noise associated to the Wiener processes $\beta^m$. We also recall that the stochastic integral appearing in the second line is nothing other than $M^m(t)$. We now take the limit $m\rightarrow\infty$ along the subsequence of all terms aside from the stochastic integral. The negative term depending on the initial data is dealt with using (\ref{eq: initial moment}) and the convergence in $\XX^+$, while limits of all other terms can be taken using the lower semicontinuity of Sobolev norms relative to the topology of $\XX^+$, the nonnegativity of $\psi(t)$ and Fatou's lemma

\begin{align}\label{Itoineq'}
\mathbb{E}\Big[\int_0^T \psi(t) \Big(&\frac{1}{2}\|\ED{U}(t)\|_H^2-\frac{1}{2}\|\ED{U}(0)\|_H^2+\int_0^t\|\ED{U}(s)\|_V^2ds+M(t)-\frac{\epsilon}{2}\|A^{1/2}\circ\sqrt{Q_{\delta}}\|_{HS}^2t\Big)dt\Big]\le 0.\end{align} Since $\psi\in L^\infty(\mathbb{P}, L^\infty([0,T],\mathbb{R}_+))$ was arbitrary, the factor in parenthesis is almost surely nonpositive for almost all times: $\mathbb{P}$-almost surely, there exists a set $\mathfrak{s}\subset[0,T]$ of full $dt$-measure, such that, for all $t\in  \mathfrak{s}$, \begin{equation}\label{eq: pwenergy}
	\frac{1}{2}\|\ED{U}(t)\|_H^2\le \frac{1}{2}\|\ED{U}(0)\|_H^2-\int_0^t\|\ED{U}(s)\|_V^2ds-M(t)+\frac{\epsilon}{2}\|A^{1/2}\circ\sqrt{Q_{\delta}}\|_{HS}^2t.
\end{equation}  We now extend this to all $t\in [0,T]$; to simplify notation, let $E^+(t)$ be the random function which appears on the right-hand side of (\ref{eq: pwenergy}).   Since $\mathfrak{s}$ is dense in $[0,T]$, we can (almost surely) construct, in a $\mathcal{F}$-measurable way, $\theta_n:[0,T]\to \mathfrak{s}$ such that $|\theta_n(t)-t|<\frac1n$ for all $t$. From (\ref{eq: pwenergy}), it holds that, almost surely \begin{equation}\label{eq: as small time change} \frac12\|\ED{U}(\theta_n(t))\|^2_H \le E^+(\theta_n(t)) \qquad \text{for all }0\le t\le T.   \end{equation} We now take $n\to \infty$. On the left-hand side, the fact that $\ED{U}$ takes values in $\XX$ implies that, almost surely, for all $t$, $\ED{U}(\theta_n(t))\to \ED{U}(t)$ weakly in $H$, and on the same almost sure event, $$\frac12\|\ED{U}(t)\|_H^2 \le \liminf_n \frac12 \|\ED{U}(\theta_n(t))\|_H^2.$$ On the other hand, $E^+$ is continuous because $M$ is, and so, almost surely, $E^+(\theta_n(t))\to E^+(t)$ uniformly in $t\in [0,T]$. We may therefore take the limits of both sides of (\ref{eq: as small time change}) as $n\to \infty$ to see that, $\PP$-almost surely, (\ref{eq: pwenergy}) holds for all times $0\le t\le T$, as desired. \medskip \\
To conclude the trajectorial energy inequality (\ref{eq: TEI}), we now identify $M(t)$ with the stochastic integral $-\sqrt{\epsilon}\int_0^t \langle \ED{U}(r), \nabla \cdot dW_\delta(r)\rangle $. First, since each $M^m$ is a martingale with respect to the complete, adapted filtration $\{{\mathcal{F}}^m(t)\}_{t\ge 0}$ generated by $\EDM{U},M^m, {\beta}^m$, it follows that for any $0\le s\le t$ and bounded, continuous $F: L^2([0,s],L^2(\mathbb{T}^3))\times C([0,s])\times C([0,s],\mathbb{R}^\mathcal{B})\to\mathbb{R}$, $$ {\mathbb{E}}\left[F\left(\left.\EDM{U}\right|_{[0,s]},\left.M^m\right|_{[0,s]},\left.{\beta}^m\right|_{[0,s]}\right)(M^m(t)-M^m(s))\right]=0.$$ Taking limits using the almost sure convergence and the boundedness and continuity of $F$  shows that the same holds for $\ED{U}, M, {\beta}$, which implies that $M$ is a martingale with respect to the complete and adapted filtration $\{{\mathcal{F}}(t)\}_{t\ge 0}$ generated by $\ED{U},M, {\beta}$. Similarly, $$ {\mathbb{E}}\left[F\left(\left.\EDM{U}\right|_{[0,s]},\left.M^m\right|_{[0,s]},\left.{\beta}^m\right|_{[0,s]}\right)\left(M^m(t)^2-M^m(s)^2-\epsilon\int_s^t\|\sqrt{Q_\delta}\nabla \EDM{U}(r) dr\|_\mathcal{M}^2\right)\right]=0.$$ Thanks to (\ref{gainonederivative}) and the definition of $\XX$,  $\|\sqrt{Q_\delta}\nabla \EDM{U}\|_\mathcal{M} \to \|\sqrt{Q_\delta}\nabla \EDM{U}\|_\mathcal{M}$, ${\mathbb{P}}$-almost surely in $L^2([0,T])$, and thanks to the integrability implied by Lemma \ref{exponential1}, the same convergence holds in $L^2({\mathbb{P}}, L^2([0,T]))$. Taking the limit of the other terms, we conclude that \begin{equation}\label{mar-1}
{\mathbb{E}}\Big[F\left(\left.\ED{U}\right|_{[0,s]},\left.M\right|_{[0,s]},\left.{\beta}\right|_{[0,s]}\right)\Big(M(t)^2-M(s)^2-\epsilon\int^t_s\|\sqrt{Q_{\delta}}\nabla \ED{U}(r)\|_\mathcal{M}^2dr\Big)\Big]=0
\end{equation} which implies that \begin{equation}\label{eq: QV1} M(t)^2-\epsilon\int_0^t\|\sqrt{Q_{\delta}}\nabla \ED{U}\|_\mathcal{M}^2dr\end{equation} is a mean-0 continuous martingale with respect to ${\mathcal{F}}(\cdot)$. Similarly, the covariation of $M$ against the Brownian motion ${\beta}$ is \begin{equation}
	d\left\langle M,{\beta}^{\zeta,i}\right\rangle (t)=\epsilon \langle \sqrt{Q_\delta} \partial_i \ED{U}, e_\zeta\rangle^2 dt.
\end{equation}  For a sequence of previsible random functions $\alpha_{\zeta,i}(t)$ satisfying $\mathbb{E}[\sum_{\zeta\in \mathcal{B}, 1\le i\le 3}\int_0^T \alpha_{\zeta,i}^2(t)dt]<\infty$, we may sum the previous display to see that \begin{equation}
	\label{eq: covariation of sum} M(t)\left(\sum_{\zeta\in \mathcal{B}, 1\le i\le 3} \int_0^t \alpha_{\zeta,i}(s)d{\beta}^{\zeta,i}_s\right)-\epsilon\int_0^t
	\sum_{\zeta \in \mathcal{B}, 1\le i\le 3} \langle \sqrt{Q_\delta}\partial_i \ED{U}(s), e_\zeta\rangle \alpha_{\zeta,i}(s)ds
\end{equation}
is a continuous, mean-0 martingale with respect to $\mathcal{F}(\cdot)$. In particular, using Parseval's identity, we may take $\alpha_{\zeta,i}(t)=\langle \partial_i \ED{U}(t), \sqrt{Q_\delta}e_\zeta\rangle$ to see that the conclusion applies to \begin{equation}\begin{split}
	\label{eq: covariation of sum'} & M(t)\int_0^t \langle \nabla \ED{U}, dW_\delta(r)\rangle -\epsilon\int_0^t \sum_{\zeta \in \mathcal{B}, 1\le i\le 3}  \langle \partial_i \ED{U}(s), \sqrt{Q_\delta}e_\zeta\rangle^2 ds\\& \hspace{1cm}=M(t)\int_0^t \langle \nabla \ED{U}, dW_\delta(r)\rangle -\epsilon\int_0^t\|\sqrt{Q_\delta}\nabla \ED{U}(s)\|_\mathcal{M}^2 ds.
\end{split}\end{equation} By a simple computation, the same conclusion also holds for \begin{equation}\label{eq: QV2}
	\left(\int_0^t \langle \nabla \ED{U}, dW_\delta(r)\rangle\right)^2 -\int_0^t \|\sqrt{Q_\delta}\nabla \ED{U}\|_\mathcal{M}^2 dr.
\end{equation} Gathering (\ref{eq: QV1}, \ref{eq: covariation of sum'}, \ref{eq: QV2}), we find
\begin{equation*}
{\mathbb{E}}\Big(M(t)-\sqrt{\epsilon}\int^t_0\langle\nabla \ED{U}(r), dW_\delta(r)\rangle\Big)^2=0,
\end{equation*}
which implies that, for any $t\ge 0$, $\mathbb{P}$-almost surely, $$M(t)=-\sqrt{\epsilon}\int^t_0\langle \ED{U}(r),\nabla\cdot dW_{\delta}(r)\rangle = -\sqrt{\epsilon}\int^t_0\langle \ED{U}(r),\sqrt{Q_\delta}\nabla\cdot dW(r)\rangle$$ as desired; the conclusion extends to holding almost surely for all times $0\le t\le T$ by continuity. Returning to (\ref{eq: pwenergy}), we have proven that $(\ED{U}, {W})$ satisfies the trajectorial energy estimate (\ref{eq: TEI}), and all of the conditions of Definition \ref{weak-solution} are satisfied.
\end{proof}
\section{Properties of the rate function}\label{sec-5}
We now give some properties of the rate function $\mathcal{I},\mathcal{J}$ defined in Section \ref{sec-1}. \begin{remark}\label{rmk: optimal g} For any $u\in \XX$ for which $\mathcal{J}(u)<\infty$, there exists a unique $g$ such that (\ref{control}) holds and which attains the infimum in (\ref{dynamic}). The optimal $g$ is uniquely characterised by the property that $\int_{\mathbb{T}^3} g(t,x)dx=0$ for almost all $t\in [0,T]$. \end{remark} This can be shown by elementary Hilbert space considerations; see \cite[Lemma 3.5]{GH} for a similar argument. As is common in large deviations theory, we will show that the rate function $\mathcal{I}$ can be expressed as a dual formulation. Let us define $C^\infty_{\rm d.f.}([0,T]\times \mathbb{T}^3)$ to be the space of all smooth functions $\varphi: [0,T]\times \mathbb{T}^3\to \mathbb{R}^3$ such that ${\rm{div}}(\varphi)=0$ for all $t,x$. 
\begin{lemma}\label{ske-lem}
The rate functions $\mathcal{I}, \mathcal{J}$ defined by (\ref{rate-3-1}, \ref{dynamic}) have the representation, for all $u\in \XX$,
\begin{equation}\label{eq: variational form J}
\mathcal{J}(u):=\sup_{\varphi\in C^{\infty}_\mathrm{d.f.}([0,T]\times\mathbb{T}^3)}\Big\{ \Lambda(\varphi,u)\Big\};
\end{equation}
\begin{equation}\label{eq: variational form}
\mathcal{I}(u):=\sup_{\varphi\in C^{\infty}_\mathrm{d.f.}([0,T]\times\mathbb{T}^3), \psi\in C_b(H,w)}\Big\{ \Lambda_0(\psi,u(0))+\Lambda(\varphi,u)\Big\},
\end{equation} where  $\Lambda(\cdot, u):C^{\infty}_\mathrm{d.f.}([0,T]\times\mathbb{T}^3)\rightarrow\mathbb{R}$ and $\Lambda_0(\cdot, u(0)):C_b(H,w)\rightarrow\mathbb{R}$ are defined by
\begin{align}\label{Fu}
\Lambda(\varphi,u):=\langle  u(T),\varphi(T)\rangle-\langle u(0),\varphi(0)\rangle &- \int_0^T
\int_{\mathbb{T}^3} (\partial_t \varphi\cdot u -\nabla \varphi \cdot \nabla u- \varphi\cdot((u\cdot \nabla u)))dtdx \\ & \hspace{1cm}-\frac12\|\nabla \varphi\|_{L^2([0,T],\mathcal{M})}^2,\notag
\end{align}
for any $\varphi\in C^{\infty}_\mathrm{d.f.}([0,T]\times\mathbb{T}^3)$, and for any $\psi\in C_b(H,w)$,
\begin{align}\label{eq: Lambda0}
\Lambda_0(\psi,u(0)):=\psi(u(0)) - \mathcal{I}_0^\star(\psi);\qquad \mathcal{I}_0^\star(\psi):=\sup_{v\in H}\left\{\psi(v)-\mathcal{I}_0(v)\right\}
\end{align} where $\mathcal{I}_0$ is the static large deviations functional in Assumption \ref{hyp: initial data}.

\end{lemma}
\begin{proof}
The variational form (\ref{eq: variational form J}) can be found in \cite[Equations (8.3), (8.5)]{QY}; see also \cite[Lemma 5.1]{kipnis1989hydrodynamics}, \cite[Section 10]{kipnis1998scaling}. The claim follows on observing that the suprema over $\varphi, \psi$ in (\ref{eq: variational form}) can be taken separately, and by Bryc's lemma, \begin{equation}
	\sup_{\psi \in C_b(H,w)} \Lambda_0(\psi, u(0))= \mathcal{I}_0(u(0))
\end{equation} so that (\ref{eq: variational form}) follows from (\ref{eq: variational form J}).  \end{proof}

The following are easy consequences of the same argument. \begin{lemma}\label{ratelem}
Each functional $\Lambda_0(\psi, u(0)), \Lambda(\varphi, u)$ is continuous with respect to the topology of $\XX$, and the rate functions $\mathcal{I}, \mathcal{J}$ defined by (\ref{rate-3-1}, \ref{dynamic}) are lower semi-continuous with respect to the topology of $\XX$.
\end{lemma}  \begin{lemma}[Convergence of optimal controls]\label{lemma: convergence of g}
	Let $u^{(n)}, u\in \XX$ and suppose that \begin{equation}
		\label{eq: I convergence} u^{(n)}\to u \text{ in the topology of }\XX; \qquad \mathcal{I}(u^{(n)})\to \mathcal{I}(u)<\infty.
	\end{equation} Then it follows that the optimal controls $g^{(n)}$ for $u^{(n)}$ converge strongly in $L^2([0,T],\mathcal{M})$ to the optimal control $g$ for $u$.
\end{lemma} 
 We will also use the following lemma, which describes the time regularity of fluctuations $u$ with $\mathcal{I}(u)<\infty$. \begin{lemma}\label{lemma: time regularity}
	Let $u\in \XX$ and suppose that $\mathcal{I}(u)<\infty$. Then $u\in W^{1,2}([0,T],H^{-1}(\mathbb{T}^3))$, so admits a weak derivative $\partial_t u\in L^2([0,T],H^{-1}(\mathbb{T}^3))$. \end{lemma} \begin{proof} Fix $u$ as given; thanks to Lemma \ref{ske-lem}, there exists $g\in L^2([0,T],\mathcal{M})$ such that the skeleton equation (\ref{control}) holds. The conclusion follows from the same calculations as in lemma \ref{exponential2}, replacing the stochastic term by \begin{equation}
	\left\|\int_0^\cdot \nabla \cdot g\right\|_{W^{1,2}([0,T], H^{-1}(\mathbb{T}^3))} = \|\nabla \cdot g\|_{L^2([0,T],H^{-1}(\mathbb{T}^3))} = \|g\|_{L^2([0,T],\mathcal{M})}<\infty.
\end{equation} \end{proof} In the next two lemmata, we show that $L^4([0,T]\times \mathbb{T}^3, \mathbb{R}^3)$ satisfies the conditions on $\mathcal{C}$ in Theorem \ref{thrm: LDP}. We will take the weak-strong uniqueness class $\mathcal{C}_0=\mathcal{N}$ to be the set of $u\in \XX$ which additionally satisfy $u\in L^{\infty}_\text{loc}((0,T];H^1(\mathbb{T}^3))\cap L^2_\text{loc}((0,T];H^2(\mathbb{T}^3))$, $\partial_tu\in L^2_\text{loc}((0,T];L^2(\mathbb{T}^3))$ and, for some $s, r$ allowed to depend on $u$, the Ladyzhenskaya-Prodi-Serrin condition (see \cite{SJ}) holds:
\begin{align}\label{serrin}
 u\in L^r([0,T];L^s(\mathbb{T}^3)),\qquad  \text{for} \qquad \frac{3}{s}+\frac{2}{r}=1,\ 3<s<\infty.\end{align}

The additional regularity will allow us to take $u$ as a test function in the Definition of Leray-Hopf weak solution, see \cite[Lemma 6.6, Lemma 8.18]{RRS}.

\begin{lemma}[Weak-strong uniqueness]\label{weak-strong}
 For $T>0$, fix $u \in \mathcal{N}$, and suppose $g\in L^2([0,T],\mathcal{M})$ is such that $u$ solves (\ref{control}). Let $v$ be another weak solution of (\ref{scontrol-0}) on $[0,T]$, with the same initial data $v(0)=u(0)\in H$, and suppose further that $v$ satisfies
\begin{equation}\label{energy-2}
\frac{1}{2}\|v(t)\|_{H}^2+\int_0^t\| v(s)\|_{V}^2ds\leq\frac{1}{2}\|v_0\|_{H}^2+\int_0^t\langle\nabla v,g\rangle ds,
\end{equation}
for all $t\in[0,T]$. Then $v(t)=u(t)$ for all $t\in[0,T]$.
\end{lemma}
\begin{proof}
The proof follows the argument of \cite[Theorem 4.2]{WE}. We note that the rough control $g$ only appears linearly in (\ref{control}), and in particular drops out when considering the difference $u-v$ as in the cited proof. 
\end{proof} We continue with a second lemma, which allows us to recover the rate function on $L^4([0,T],\times \mathbb{T}^3, \mathbb{R}^3)$ as the lower semicontinuous envelope of its restriction on $\mathcal{N}$. \begin{lemma}
\label{l4 recovery} Let $v\in L^4([0,T],\times \mathbb{T}^3, \mathbb{R}^3)$	with $\mathcal{I}(v)<\infty$. Then there exists a sequence $v^{(n)}\in \mathcal{N}$ such that $v^{(n)}\to v$ in the topology of $\XX$ and $\mathcal{I}(v^{(n)})\to \mathcal{I}(v)$.
\end{lemma}
\begin{proof} Using $u\in L^\infty([0,T],H)$ and Lemma \ref{lemma: time regularity}, it follows that the trajectory $v^{(n)}=v\star \lambda_{1/n}$ given by convolving $v$ with a smooth mollifier $\lambda_{1/n}$ at scale $n^{-1}$ satisfies $v^{(n)}\in L^\infty([0,T];L^\infty(\mathbb{T}^3))$ and $\partial_t v^{(n)} \in L^2([0,T],L^\infty(\mathbb{T}^3))$, and hence $v^{(n)}\in \mathcal{N}$. \bigskip \\   To conclude, we must show that $\limsup_n \mathcal{I}(v^{(n)})\le \mathcal{I}(v)$. The convergence of the initial cost is immediate thanks to the convergence in norm. Using Lemma \ref{ske-lem}, we pick an optimal control $g$ for $v$, and observe that the optimal control for $v^{(n)}$ is given by $g^{(n)}=g\star\lambda_{1/n}+\eta^{(n)}$, where $\eta^{(n)}:=\frac{1}{2}(v\otimes v)\star \lambda_{1/n}-\frac{1}{2}(v\star \lambda_{1/n})\otimes (v\star \lambda_{1/n})$ is the commutator. The hypothesis that $v\in L^4([0,T],L^4(\mathbb{T}^3))$ guarantees that $\eta^{(n)}\to 0$ in $L^2([0,T],\mathcal{M})$, whence $\limsup_n \mathcal{J}(v^{(n)})\le \mathcal{J}(v)$ and the proof is complete.
	\end{proof}
	
\section{Lower bound of the Large Deviation Principle}\label{sec-6}
We now prove the large deviations lower bound. We will argue, throughout this section, in the case of solutions to (\ref{SNS-1}), with scaling parameters $\epsilon, \delta$ satisfying (\ref{scale-standing}). The case of solutions to (\ref{eq: galerkin approx}) under (\ref{scale-standing-galerkin}) is essentially identical.

We recall, from the setting of the theorem, that $\mathcal{C}_0$ is a weak-strong uniqueness class for (\ref{control}) and $\mathcal{C}$ is its $\mathcal{I}$-closure in the sense of Definition \ref{def: wk stron uniqueness class}, which we can concretely take to be $\mathcal{C}_0=\mathcal{N}, \mathcal{C}=L^4([0,T]\times\mathbb{T}^3,\mathbb{R}^3)$ thanks to Lemmata \ref{serrin}, \ref{l4 recovery}. We will show, for $v\in \mathcal{C}_0 \subset \XX$ satisfying $\mathcal{I}(v)<\infty$, how measures $\pi^{\epsilon, v}$ can be constructed to apply the forward direction of Lemma \ref{entropymethod}. Under these assumptions, by Lemma \ref{ske-lem}, there exists $g\in L^2([0,T],\mathcal{M})$ such that $v$ satisfies the skeleton equation \begin{equation}
	\label{eq: sk} \partial_t u-\Delta u+\mathbf{P}((u\cdot \nabla)u)=-\nabla\cdot g.
\end{equation} {By Assumption \ref{hyp: initial data} and the converse part of Lemma \ref{entropymethod}, there exist families of random variables $\ED{V}(0)$, converging to $v(0)$ in law with respect to the norm topology of $H$, such that \begin{equation}\label{eq: ent initial}
	\limsup_{\epsilon \to 0} \epsilon {\rm{Ent}}\left(\pi^{\epsilon, v(0), 0}_\delta|\mu^{\epsilon, 0}_\delta\right)\le  \mathcal{I}_0(v),
\end{equation} where we write $\mu^{\epsilon, 0}_\delta$ for the law of $\ED{U}(0)$ in Theorem \ref{thrm: LDP}, which satisfies the standing hypothesis Assumption \ref{hyp: initial data}, and $\pi^{\epsilon, v(0), 0}_\delta$ for the law of $\ED{V}(0)$.} We now construct $\pi^{\epsilon,g}, \epsilon>0$ as the law of stochastic Leray solutions to
\begin{equation}\label{scontrol-0}
\partial_t\ED{V}=-A \ED{V}-\mathbf{P}(\ED{V}\cdot\nabla)\ED{V}-\sqrt{\epsilon}\nabla\cdot\xi_\delta-\nabla\cdot \Big(\sqrt{Q_{\delta}}g\Big)
\end{equation} with the usual $\delta=\delta(\epsilon)$ satisfying (\ref{scale-standing}), and extending the notion of stochastic Leray solutions in Definition \ref{weak-solution} by including the corresponding terms involving $g$ to both the equation (\ref{weaksense}) and the trajectorial inequality (\ref{eq: TEI}). We prove the concentration of measures required to apply Lemma \ref{entropymethod} by showing tightness for the weak solutions $\{\ED{V}\}_{\epsilon>0}$, showing that any subsequential limits satisfy the skeleton equation (\ref{eq: sk}), and using weak-strong uniqueness for the original solution $v$, which follows from the definition of $\mathcal{C}_0$.
\begin{lemma}
	\label{lemma: existence of tilted} Fix $g\in L^2([0,T],\mathcal{M})$ and $v(0)\in H$, and let $\pi^{\epsilon,v(0),0}_{\delta}$ be as above. For all $\epsilon, \delta>0$, let $\ED{\mu}$ be the law of a stochastic Leray solution to (\ref{SNS-1}) with initial data satisfying Assumption \ref{hyp: initial data}. Then, on a different probability space, there exists a stochastic Leray solution $\ED{V}$ to (\ref{scontrol-0}) with initial data $\ED{V}(0)\sim \pi^{\epsilon, v(0)}_{\delta}$, satisfying the trajectorial energy estimate \begin{equation}\begin{split}
		\label{eq: energy ineq for V} \frac{1}{2}\|\ED{V}(t)\|_{H}^2+\int_{0}^{t}\| \ED{V}(r)\|_{V}^2dr \leq & \frac{1}{2}\|\ED{V}(0)\|_{H}^2-\sqrt{\epsilon}\int_0^t\langle \ED{V},\nabla\cdot dW_{\delta}(r)\rangle \\& \hspace{1cm} + \frac{\epsilon}{2}\|A^{1/2}\circ\sqrt{Q_{\delta}}\|_{HS}^2t +\int_0^t \langle \nabla \ED{V}, \sqrt{Q_\delta}g\rangle ds \end{split}
\end{equation} and such that the law  $\pi^{\epsilon,g}_\delta$ satisfies the entropy estimate \begin{equation}
	\label{eq: entropy controlled solution} \limsup_{\epsilon\to 0} \epsilon{\rm{Ent}}\left(\pi^{\epsilon, g}_\delta|\ED{\mu}\right)\le \mathcal{I}_0(v(0)) + \frac12\|g\|_{L^2([0,T],\mathcal{M})}^2.\end{equation}
\end{lemma} \begin{proof}
 Let us fix a probability space $({\Omega},{\mathcal{F}},\widetilde{\PP})$ on which are defined a stochastic Leray solution $\ED{V}$ to (\ref{SNS-1}) started at $\ED{V}(0) \sim \mu^{\epsilon,0}_{\delta}$, driven by a cylindrical Wiener process  $\tilde \beta$, corresponding to a divergence free white noise $\tilde{W}$, independently of $\ED{V}(0)$. We define \begin{equation}
	Y_0:=\frac{d\pi^{\epsilon,v(0),0}_{\delta}}{d\mu^{\epsilon,0}_{\delta}}(\EDM{V}(0));
\end{equation} which is licit in view of the finiteness of (\ref{eq: ent initial}), and \begin{equation}\label{Zt}
	{Z_t:=\exp\left(-\frac{1}{\sqrt{\epsilon}}\int_0^t \langle \sqrt{Q_\delta}g, d\tilde{W}(s)\rangle-\frac{1}{2\epsilon} \int_0^t \|\sqrt{Q_\delta}g(s)\|_{\mathcal{M}}^2 ds \right)Y_0.}
\end{equation} Noting that $\sqrt{Q_\delta}g$ belongs to the Cameron-Martin space for $\sqrt{Q_\delta}\tilde{W}$, we apply the Cameron-Martin Theorem \cite[Theorem 2.26]{da2014stochastic} to identify $Z_T$ with the density $\frac{d\PP}{d\widetilde{\PP}}$ of a new probability measure $\mathbb{P}$ under which $W:=\tilde{W}+\epsilon^{-1/2}\int_0^\cdot gds $ is a divergence-free white noise and $\ED{V}(0)\sim \pi^{\epsilon, v(0),0}_{\delta}$. The equation (\ref{scontrol-0}) and energy estimate (\ref{eq: energy ineq for V}) under the new measure $\mathbb{P}$ follow from the corresponding properties under $\widetilde{\mathbb{P}}$. For the entropy estimate, a simple computation using the information processing inequality yields \begin{equation}
	\begin{split}
		\epsilon {\rm{Ent}}\left(\ED{\pi}|\ED{\mu}\right)\le \epsilon \rm{Ent}\left(\PP|\widetilde{\PP}\right) & = \epsilon \mathbb{E}\left[\log Z_T\right] \\ & =\frac{1}{2}\int_0^T \|g\|_{\mathcal{M}}^2ds + \epsilon \int_{H} \log \frac{d\pi^{\epsilon,v(0),0}_{\delta}}{d\mu^{\epsilon,0}_{\delta}}(v)\pi^{\epsilon,v(0),0}_{\delta}(dv).
	\end{split}
\end{equation} The final term is exactly the definition of $\epsilon {\rm{Ent}}(\pi^{\epsilon,v(0),0}_{\delta}|{\mu^{\epsilon,0}_{\delta}})$, of which the limit superior is at most $\mathcal{I}_0(v(0))$ by (\ref{eq: ent initial}), and we have proven (\ref{eq: entropy controlled solution}).  \end{proof}

\begin{lemma}\label{stationary}
Let $v \in \mathcal{C}_0$ be a solution of (\ref{control}) with control $g\in L^2([0,T],\mathcal{M})$ attaining the minimum cost in $\mathcal{J}(v)$. For each $\epsilon, \delta$, let $\pi^{\epsilon, g}_{\delta}$ be as constructed in Lemma \ref{lemma: existence of tilted} with this choice of $g$, and with initial data $\ED{V}(0)$ constructed as at (\ref{eq: ent initial}). Under the scaling relation (\ref{scale-standing}), $\pi^{\epsilon,g}_\delta$ converges to $\delta_v$ in the weak topology of measures on $\XX$.
\end{lemma}
\begin{proof} We use tightness and an identification of the limit points. Tightness follows from Corollary \ref{ET'} (ii) and an Orlicz inequality; thanks to the tightness in the norm topology of $H$ remarked above (\ref{eq: ent initial}), this extends to $\XX^+$. For an identification of the limit points, we use the same arguments as \cite{FG95} to see that all subsequential limit points $V$ almost surely satisfy (\ref{control}) in the weak sense: $\mathbb{P}$-almost surely, for any $\varphi\in C^\infty_{\rm d.f.}(\mathbb{T}^3)$, for all $0\le t\le T$
\begin{equation*}
\langle{V}(t),\varphi\rangle=\langle v(0),\varphi\rangle-\int_0^t\langle\nabla{V},\nabla\varphi\rangle ds-\int_0^t\langle({V}\cdot\nabla){V},\varphi\rangle ds+\int_0^t\langle g,\nabla\varphi\rangle ds.
\end{equation*} For the energy inequality, we repeat the arguments of Section \ref{existence}. The martingale term vanishes in the limit, possibly up to extracting a further subsequence, and the It\^o correction vanishes thanks to the scaling hypothesis (\ref{scale-standing}). The only new term is the control term, which can be dealt with using weak convergence of $\nabla V^\epsilon \to \nabla V$, the strong convergence $\sqrt{Q_\delta}g\to g$ and the strong compactness of the functions $g_t(s,x):=g(s,x)1[s\le t]$ in $L^2([0,T],\mathcal{M})$ to see that, ${\mathbb{P}}$-almost surely, \begin{equation}
	\label{control term energy ineq} \sup_{t\le T}\left|\int_0^t \langle \nabla V^\epsilon, \sqrt{Q_\delta}g\rangle ds - \int_0^t \langle \nabla V, g\rangle ds \right|\to 0.
\end{equation} One finally finds that, $\PP$-almost surely, $V$ satisfies the energy inequality associated to (\ref{eq: sk}) for all $0\le t\le T$. Since $\mathcal{C}$ is a weak-strong uniqueness class, we conclude that $V=v$ almost surely, and we are done.
\end{proof}

\begin{remark}\label{rmk: existence skeleton}
Since we only used weak-strong uniqueness at the end, the arguments of the lemma actually produce the following. Given $v_0\in H$ and $g\in L^2([0,T],\mathcal{M})$, there exists a solution $v\in \XX$ to (\ref{eq: sk}) which satisfies the energy inequality.
\end{remark}
We can now prove the large deviations lower bound restricted to $\mathcal{C}_0$.
\begin{lemma}[Lower bound on $\mathcal{C}_0$]\label{lowerbound-1}
Under the standing scaling assumption (\ref{scale-standing}), let $U^{\epsilon}_\delta, \epsilon>0$ be stochastic Leray solutions of (\ref{SNS-1}) with parameters $\epsilon, \delta(\epsilon)$ and initial condition $U^\epsilon_\delta(0)\sim \mu^{\epsilon,0}_{\delta}$ satisfying Assumption \ref{hyp: initial data}.  Then the induced measures $\mu^{\epsilon}_\delta=\text{Law}(U^\epsilon_{\delta(\epsilon)})$ on $\XX$ satisfy, for each open subset $\mathcal{U}$ of $\XX$,
\begin{equation*}
\liminf_{\epsilon\rightarrow0}\epsilon \log \mu^{\epsilon}_\delta (\mathcal{U})\geq-\inf_{v\in \mathcal{U} \cap \mathcal{C}_0}\mathcal{I}(v)
\end{equation*}

\end{lemma}
\begin{proof}
This follows from applying Lemma \ref{entropymethod} with the function $\mathcal{I}_1$, modifying $\mathcal{I}$ to be infinite outside $\mathcal{C}_0$. The conditions for this application follow from Lemma \ref{ske-lem}, Lemma \ref{lemma: existence of tilted} and Lemma \ref{stationary}.
\end{proof}
Using the definition of the rate-closure, we now conclude the final lower bound
\begin{proposition}[Restricted lower bound]\label{proposition: final lower bound}
Under the same hypotheses as Lemma \ref{lowerbound-1}, for any open set $\mathcal{U}\subset \XX$, and for $\mathcal{C}$ as in Theorem \ref{thrm: LDP},
\begin{equation*}
\liminf_{\epsilon\rightarrow0}\epsilon \log \mu^{\epsilon}_\delta(\mathcal{U})\geq-\inf_{v\in \mathcal{U}\cap \mathcal{C}}\mathcal{I}(v).
\end{equation*}

\end{proposition}

\section{Upper bound of the Large Deviation Principle}\label{sec-7} We now prove the corresponding upper bound, and Theorem \ref{thm: improved ub}. Since the arguments are the very similar, we consolidate them into a single proposition.
\begin{proposition}[Large Deviations upper bound]\label{upresult-1}
 Let $\epsilon, \delta=\delta(\epsilon)$ be scaling parameters satisfying (\ref{scale-standing}), and for each $\epsilon>0$, let $\ED{\mu}$ be the law of a stochastic Leray solution $\ED{U}$ to (\ref{SNS-1}), starting from $\ED{U}(0)$ satisfying Assumption \ref{hyp: initial data}. Then, for any closed subset $\mathcal{V}$ of $\XX$, it holds that
\begin{equation*}
\limsup_{\epsilon\rightarrow0}\epsilon \log \ED{\mu}(\mathcal{V})\leq-\inf_{v\in \mathcal{V}}\max\left(\mathcal{I}(v), S(\mathcal{E}^\star(v))\right)
\end{equation*} where $\mathcal{E}^\star$ is defined by (\ref{eq: Estar}), and the function $S$ is constructed in Lemma \ref{lemma: improved UB'}.

\end{proposition} Let us remark that this implies both the upper bound of Theorems \ref{thrm: LDP} and \ref{thm: improved ub}.
\begin{proof} Thanks to exponential tightness in Proposition \ref{exponential3}, it is sufficient to prove the estimate in the case when $\mathcal{V}\subset \XX$ is additionally compact.  Recalling the variational form of $\mathcal{I}$ in Lemma \ref{ratelem}, let us now fix $\varphi \in C^\infty_\mathrm{d.f.}([0,T]\times\mathbb{T}^3), \psi\in C_b(H,w)$ and define $\mathcal{N}^{\varphi}:[0,T]\times \XX \rightarrow\mathbb{R}$  by
\begin{align*}
\mathcal{N}^{\varphi}(t,u):=&\langle u(t),\varphi(t)\rangle-\langle u(0),\varphi(0)\rangle-\int_{[0,t]\times\mathbb{T}^3}\langle u,\partial_s\varphi\rangle ds dx +\int_{[0,t]\times \mathbb{T}^3}\langle\nabla u,\nabla\varphi\rangle dt dx\\
&\hspace{3cm}+\int_{[0,t]\times \mathbb{T}^3}\langle (u\cdot\nabla)u,\varphi\rangle dt dx.\end{align*}

From the criterion (\ref{weaksense}) in the Definition \ref{weak-solution} of stochastic Leray solutions, it follows that $\mathcal{N}^{\varphi,\psi}(t, \ED{U})=\mathcal{N}^{\varphi}(t,\ED{U})+\Lambda_0(\psi, \ED{U}(0))$ is a martingale with quadratic variation
\begin{equation}
\langle\mathcal{N}^{\varphi,\psi}(\cdot,\ED{U})\rangle(t)=\epsilon\int_0^t\|\sqrt{Q_{\delta}}\nabla\varphi\|_{\mathcal{M}}^2ds \le \epsilon\int_0^t\|\nabla\varphi\|_{\mathcal{M}}^2ds.
\end{equation}

We define a map $Q^{\epsilon}_{\varphi,\psi}:[0,T]\times \XX \rightarrow\mathbb{R}$ by
\begin{equation*}
Q^{\epsilon}_{\varphi,\psi}(t,u):=\exp\Big\{\epsilon^{-1}\left(\mathcal{N}^{\varphi}(t,u)+\Lambda_0(\psi, u(0))-\frac{1}{2}\int_0^t\|\nabla\varphi\|_{\mathcal{M}}^2ds\right)\Big\}
\end{equation*}

so that $Q^{\epsilon}_{\varphi,\psi}(\cdot,U^{\epsilon}_{\delta})$ is a supermartingale. Further, using the definition of $\Lambda_0, \mathcal{I}_0^\star$ and Varadhan's integral lemma, \begin{equation}\label{eq: Varadahn} \limsup_{\epsilon}\epsilon \log \mathbb{E}^\epsilon\left[Q^{\epsilon}_{\varphi,\psi}(0, \ED{U})\right]=0. \end{equation} Since $Q^{\epsilon}_{\varphi,\psi}(\cdot,U^{\epsilon}_{\delta})$ is a nonnegative supermartingale, we can use Chebyshev's inequality to see that, for any subset $\mathcal{U}\subset \XX$,
\begin{align*}
\mu^{\epsilon}_{\delta}(\mathcal{U})=\int_{\XX}I_{\{u\in \mathcal{U}\}}Q^{\epsilon}_{\varphi,\psi}(T,u)(Q^{\epsilon}_{\varphi,\psi}(T,u))^{-1}\mu^{\epsilon}_{\delta}(du)& \le \sup_{u\in \mathcal{U}}\left(Q^{\epsilon}_{\varphi,\psi}(T,u)^{-1}\right)\int_{\XX }Q^{\epsilon}_{\varphi,\psi}(T,u)\mu^{\epsilon}_{\delta}(du)\\
=&\sup_{u\in \mathcal{U}}\left(Q^{\epsilon}_{\varphi,\psi}(T,u)^{-1}\right)\mathbb{E}^{\epsilon}\left[Q^{\epsilon}_{\varphi,\psi}(T, u^\epsilon)\right]\end{align*}
whence, using (\ref{eq: Varadahn}),
\begin{align}\label{eq: UB with phi psi}
\limsup_{\epsilon\to 0}\epsilon \log \mu^{\epsilon}_{\delta}(\mathcal{U})  \le  - \inf_{u\in \mathcal{U}}\Big(\mathcal{N}^{\varphi}(T,u)+\Lambda_0(\psi, u(0))-\frac{1}{2}\int_0^T\|\nabla\varphi\|_{H}^2ds\Big).
\end{align}
Recalling the definition (\ref{Fu}), the $\varphi$-dependent terms $\mathcal{N}^\varphi(T,u)-\int_0^T\|\nabla \varphi\|_{\mathcal{M}}^2ds$ are exactly the definition of $\Lambda(\varphi, u)$. Given a trajectory $u\in \mathcal{V}$ and $\lambda>0$, we may first use Lemma \ref{ske-lem} to choose $\varphi, \psi$ so that $$\Lambda_0(\psi, u(0))+\Lambda(\varphi, u)>\mathcal{I}(u)-\lambda$$ and then use continuity of the maps $v\mapsto \Lambda_0(\psi, v(0)), \Lambda(\varphi, v)$ to choose an open $\mathcal{U}\ni u$ such that $$ \inf_{v\in \mathcal{U}}\left(\mathcal{N}^{\varphi}(T,v)+\Lambda_0(\psi, v(0))-\frac{1}{2}\int_0^T\|\nabla\varphi\|_{H}^2ds\right)>\mathcal{I}(u)-2\lambda. $$ In particular, we may combine with Lemma \ref{lemma: improved UB'} to find, for any $u$ and any $\lambda>0$ an open $\mathcal{U}(u)\ni u$ such that \begin{equation}
	\label{eq: LUB} \limsup_{\epsilon\to 0}\epsilon \log \ED{\mu}(\mathcal{U}(u))\le -\max\left(\mathcal{I}(u)-2\lambda, S(\mathcal{E}^\star(u))\right) \le -\max\left(\mathcal{I}(u), S(\mathcal{E}^\star(u))\right)+2\lambda.
\end{equation} Since $\mathcal{V}$ is compact, we may cover $\mathcal{V}$ with finitely many $\mathcal{U}(u)$ with $u$ running over a finite subset of $\mathcal{V}$. Keeping only the most favourable exponent, it follows that \begin{equation}
	\limsup_{\epsilon\to 0} \epsilon \log \ED{\mu}(\mathcal{V})\le -\inf \left(\max\left(\mathcal{I}(u), S(\mathcal{E}^\star(u))\right): u\in F\right)+2\lambda.
\end{equation} Since $\lambda>0$ was arbitrary, we may now send $\lambda\to 0$, and the proof is complete.
\end{proof}

\section{From Large Deviations to the Energy Equality}\label{sec: energy} We now give the proof of Theorem \ref{thrm: Energy Equality} and Remark \ref{rmk: violations of EI}. The relationship sharpness of the large deviation bounds to the energy equality arises because, in the stationary and reversible initial condition $\EDM{U}(0)\sim \cG(0, \epsilon I_m/2)$, $\delta=0, m(\epsilon)\ll \epsilon^{-\frac15}$, the energy $\|u(0)\|_{H}^2$ appears as the rate function for the large deviations of the initial data $\EDM{U}(0)$, see Lemma \ref{lemma: exponential calculation ID}. The key point is, by the time-symmetry, $\mathfrak{T}_T{U}^\epsilon_{0,m}$ have the same large deviations as ${U}^\epsilon_{0,m}$, and that, on the time-reversal of paths where the upper and lower bound match, we obtain an identity \begin{equation*}
\mathcal{I}(u)=\mathcal{I}(\mathfrak{T}_Tu)\end{equation*} which can be manipulated into the claimed energy \emph{equality}. This probabilistic argument has been used in other contexts by the first and second authors  \cite[Section 13]{GH}. For the main result Theorem \ref{thrm: Energy Equality}, we extract out the elements of the proof to give a shorter argument which does not use any probability theory; we will use the probabilistic argument above in the course of proving Remark \ref{rmk: violations of EI}. We begin with a short lemma which characterises trajectories in the class $\mathcal{C}$ in the theorem. \begin{lemma}\label{lemma: GEI on cR} Let $\mathcal{C}$ be as in Theorem \ref{thrm: Energy Equality}. Then, for all $u\in \mathcal{C}$ with $\mathcal{I}(u)<\infty$, there exists $g\in L^2([0,T],\mathcal{M})$ such that $u$ solves (\ref{control}), and additionally satisfies, for all $0\le t\le T$, the generalised energy \emph{in}equality on $[0,t]$:
  	\begin{equation}
	\label{eq: general energy inequality intervals} \frac12 \|u(t)\|_H^2+\int_0^t \|u(s)\|_V^2 ds \le  \frac12 \|u(0)\|_H^2+\int_0^t\langle \nabla u, g\rangle ds.
\end{equation}
  \end{lemma} Let us remark that this is already not obvious from the definition, since the notion of solutions to (\ref{control}) in the definition of $\mathcal{I}$ does not impose such an energy inequality, and the hypotheses on $\mathcal{C}_0, \mathcal{C}$ concern only uniqueness. \begin{proof}
  	The finiteness of $\mathcal{I}(u)$ implying that $u$ solves (\ref{control}) for some $g$ is already part of the content of Lemma \ref{ske-lem}, and we only need to argue (\ref{eq: general energy inequality intervals}); throughout the rest of the proof, we will fix the optimal such $g$. We first argue in the special case $u\in \mathcal{C}_0$, and then extend to $u\in \mathcal{C}$. \paragraph{\textbf{Case 1. $u\in \mathcal{C}_0$}} Let us first suppose that $u\in \mathcal{C}_0$ and $\mathcal{I}(u)<\infty$. In this case, we can directly follow the arguments of Lemma \ref{lemma: existence of tilted} as in Remark \ref{rmk: existence skeleton} to construct solutions $v(t)$ to (\ref{scontrol-0}) for our choice of $g$ and with initial data $v(0)=u(0)$ converging to $u(0)$ and which further satisfy the energy inequality (\ref{eq: general energy inequality intervals}) as desired. Since $\mathcal{C}_0$ is a weak-strong uniqueness class, it follows that $v=u$, and so $u$ satisfies the energy inequality (\ref{eq: general energy inequality intervals}). \paragraph{\textbf{Case 2. $u\in \mathcal{C}$}} We now deal with the more general case stated in the lemma, in which we only assume that $u\in \mathcal{C}$. By definition of $\mathcal{C}= \overline{\mathcal{C}_0}^{\mathcal{I}}$, it follows that there exists a sequence $u^{(n)}\in \mathcal{C}_0$ with $u^{(n)}\to u$ in the topology of $\XX$ and $\mathcal{I}(u^{(n)})\to \mathcal{I}(u)$. This is the setting of Lemma \ref{lemma: convergence of g}, and so we see that the optimal controls $g^{(n)}$ associated to $u^{(n)}$ converge in $L^2([0,T],\mathcal{M})$ to the optimal control $g$ for $u$. For each $n$, case 1 shows that the inequality (\ref{eq: general energy equality}) holds for $u^{(n)}, g^{(n)}$, and we take a limit, using the same arguments as in Proposition \ref{existence} and Lemma \ref{stationary}, now in a deterministic setting and using $$ \sup_{t\le T} \left|\int_0^t \langle \nabla u^{(n)}, g^{(n)}\rangle ds -  \int_0^t \langle \nabla u, g\rangle ds\right| \to 0 $$ thanks to the weak convergence $\nabla u^{(n)}\to \nabla u$ and the strong convergence $g^{(n)}\to g$. \end{proof}
\begin{proof}[Proof of Theorem \ref{thrm: Energy Equality}] Let $\mathcal{C}_0, \mathcal{C}, \mathcal{R}$ be as given, and let $u\in \mathcal{R}$ be of finite rate. By Lemma \ref{lemma: GEI on cR}, there exists $g\in L^2([0,T],\mathcal{M})$ such that $u$ solves (\ref{control}) with the energy inequality (\ref{eq: general energy equality}). We now observe that the time-reversal $v:=\mathfrak{T}_T u\in \mathcal{R}$, since $\mathcal{R}$ was assumed to be time-symmetric, and solves (\ref{control}) for the new choice \begin{equation}
  	\label{eq: reversed control} g_\mathrm{r}(t,x):={g}(T-t,x)+2\nabla v(t,x)
  \end{equation} which belongs to $L^2([0,T],\mathcal{M})$ thanks to the regularity $u\in \XX$. By Lemma \ref{lemma: GEI on cR} again, $v$ satisfies the energy inequality \begin{equation}
  	\frac12\|v(t)\|_H^2+\int_0^t\|v(s)\|_V^2ds \le \frac12 \|v(0)\|_H^2+\int_0^t \langle \nabla v, g_r\rangle ds.
  \end{equation} Using the definitions of $v$, $g_\mathrm{r}$ and taking $t=T$ in the previous display, we find \begin{equation}
  	\frac12\|u(0)\|_H^2+\int_0^T\|u(s)\|_V^2ds \le \frac12 \|u(T)\|_H^2-\int_0^T \langle \nabla u, g\rangle ds + 2\int_0^T\|u(s)\|_V^2 ds
  \end{equation} and combining with the usual energy inequality, the conclusion follows.\end{proof} We now give the proof of Remark \ref{rmk: violations of EI}. \begin{proof}
  As in the statement of the result, we consider the large deviations of $\EOM{U}$, started in the invariant distribution $\EOM{U}(0)\sim \mathcal{G}(0, \epsilon P_m/2)$, under the scaling (\ref{scale-standing-galerkin}) By Lemma \ref{lemma: exponential calculation ID}, Assumption \ref{hyp: initial data} holds with $\mathcal{I}_0(u(0))=\|u(0)\|_H^2$. For both items i, ii), we will argue the contrapositive. For item i), suppose that $u$ is a weak solution to (\ref{control}) for some (optimal) $g\in L^2([0,T],\mathcal{M})$, such that the lower bound holds at $u$: \begin{equation}
  		\label{eq: loc LB} \inf_{\mathcal{U}\ni u}\liminf_{\epsilon\to 0} \epsilon \log \EOM{\mu}(\mathcal{U})\ge -\mathcal{I}(u). \end{equation} Recalling that the topology of $\XX$ is separable and Haussdorf, let $\mathcal{U}_n, n\in \mathbb{N}$ be a decreasing sequence of open sets such that $\cap_n \overline{\mathcal{U}}_n=\{u\}$. Using the time-reversal property of $\EOM{U}$, it follows that $(\mathfrak{T}_T)_\#\EOM{\mu}=\EOM{\mu}$, whence, for all $n$, \begin{equation}
  			\begin{split}
  				-\mathcal{I}(u) & \le \liminf_{\epsilon\to 0} \epsilon \log \EOM{\mu}(\mathcal{U}_n) \\ & \le \limsup_{\epsilon\to 0} \epsilon \log ((\mathfrak{T}_T)_\#\EOM{\mu})(\overline{\mathcal{U}}_n) \\&  =  \limsup_{\epsilon\to 0} \epsilon \log (\EOM{\mu})(\mathfrak{T}_T \overline{\mathcal{U}}_n) \le -\inf\left\{\mathcal{I}(v): \mathfrak{T}_Tv\in \overline{\mathcal{U}}_n\right\}
  			\end{split}
  		\end{equation} where the final line follows from the global upper bound in Theorem \ref{thrm: LDP}. Thanks to the lower semicontinuity of $\mathcal{I}$, the infimum on the right-hand side converges to $\mathcal{I}(\mathfrak{T}_Tu)$ as $n\to \infty$, and we conclude that \begin{equation}
  			\label{eq: estimate on RF reversal} \mathcal{I}(u)\ge \mathcal{I}(\mathfrak{T}_Tu).
  		\end{equation} To compute $\mathcal{I}(\mathfrak{T}_Tu)$, we observe that if $g$ is optimal for $u$, then the new choice $$g_\mathrm{r}(t,x):=g(T-t,x)-2\nabla u(T-t,x)$$ chosen at (\ref{eq: reversed control}) is also optimal thanks to Remark \ref{rmk: optimal g}, and $\mathcal{I}_0(\mathfrak{T}_Tu(0))=\|u(T)\|_H^2$. Expanding out (\ref{eq: estimate on RF reversal}), we have \begin{equation} \begin{split}
  			\mathcal{I}(\mathfrak{T}_Tu)&=\|u(T)\|_H^2+\frac12\|g\|_{L^2([0,T],\mathcal{M})}^2-2\int_0^T\langle \nabla u, g\rangle dt + 2 \int_0^T\|u(t)\|_V^2 dt \\ & \le \mathcal{I}(u)=\|u(0)\|_H^2+\frac12 \|g\|_{L^2([0,T],\mathcal{M})}^2. \end{split}
  		\end{equation} Subtracting the term in $g$ produces the energy inequality at time $T$, and we have proven the contrapositive of the claim in part i) of the remark. Part ii) follows from a further time-reversal, since if $u$ is a solution to (\ref{control}) for which the energy inequality is strict, then the time-reversal $u:=\mathfrak{T}_Tv$ is a solution to (\ref{control}) for the new choice $g_\mathrm{r}$ which violates the energy inequality.
  \end{proof}

\section{Large Deviations of the strong solution}\label{sec-9}

In this section, we will discuss the large deviation properties of the \emph{stochastic strong solution} and the \emph{strong} solution to the Navier-Stokes equations \begin{equation}\label{DNSE}
\partial_t\overline{u}=-A\overline{u}-\mathbf{P}(\overline{u}\cdot\nabla)\overline{u}.
\end{equation} \subsection{Definition of Strong Solutions} We begin with a definition. \begin{definition}\label{strong-solution}(Local in time strong solution)
Let $\epsilon>0$, $U_0\in L^2(\Omega;{V})$. A couple $(\overline{U}^{\epsilon},\tau^{\epsilon})$ is a strong solution of  (\ref{SNS-1}) with initial data $U_0$, if

\item[(i)]$\tau^{\epsilon}$ is a stopping time with respect to $\{\mathcal{F}(t)\}_{t\in[0,T]}$ such that $\tau^{\epsilon}\in(0,T]$, $\mathbb{P}$-almost surely.
	\item[(ii)] $\mathbb{P}$-almost surely,
\begin{align*}
\overline{U}^{\epsilon}\in C([0,\tau^{\epsilon});H^1(\mathbb{T}^3))\cap L^2_{\mathrm{loc}}([0,\tau^{\epsilon});H^2(\mathbb{T}^3)) \end{align*} {and, $\mathbb{P}$-almost surely, ${\rm{div}} \overline{U}(t)=0$ for all $0\le t\le \tau^\epsilon$.}
\item[(iii)] For some $C=C(T)>0$, $\overline{U}^{\epsilon}$ satisfies the estimate
\begin{align*}
\mathbb{E}\Big(\sup_{t\in[0,T]}\|\overline{U}^{\epsilon}(t\wedge\tau^{\epsilon})\|_{L^2(\mathbb{T}^3)}\Big)^2+\mathbb{E}\int^{T\wedge\tau^{\epsilon}}_0\int_{\mathbb{T}^3}\| \overline{U}^{\epsilon}(s)\|_{V}^2ds\leq C(T)\cdot(\mathbb{E}\|U_0\|_{L^2(\mathbb{T}^3)}^2+1). 	
\end{align*}

\item[(iv)] $\mathbb{P}$-almost surely, for every $\varphi\in C^{\infty}_\mathrm{d.f.}([0,T]\times\mathbb{T}^3)$ and all $0\le t\le T$,
\begin{equation} \begin{split} & \langle \overline{U}^{\epsilon}(t\wedge\tau^{\epsilon}),\varphi\rangle+\int_0^{t\wedge\tau^{\epsilon}}\langle A \overline{U}^{\epsilon}(s),\varphi\rangle ds\\ & \hspace{3cm}= \langle U_0,\varphi(0)\rangle-\int_0^{t\wedge\tau^{\epsilon}}\langle(\overline{U}^{\epsilon}\cdot\nabla)\overline{U}^{\epsilon},\varphi\rangle ds+\sqrt{\epsilon}W_{t\wedge\tau^{\epsilon}}\Big(\sqrt{Q_{\delta}}\nabla\varphi\Big).
\end{split} \end{equation}
In the same way, we use the terminology `local strong solution' for a divergence-free vector field $(\overline{u}(t): 0\le t\le \tau)$ solving (\ref{DNSE}) and additionally satisfying $\overline{u}\in C([0,\tau),H^1(\mathbb{T}^3))\cap L^2_{\mathrm{loc}}([0,\tau),H^2(\mathbb{T}^3))$.
\end{definition}

The existence of local-in-time strong solutions to (\ref{SNS-1}) and (\ref{DNSE}) is established in \cite[Theorem 2.3]{DA} and \cite[Theorem 1.2]{KJ}. In particular, in \cite{KJ}, the author dealt with the local in time strong solution of the stochastic Navier-Stokes equations on $\mathbb{R}^3$ with fractional differentiable initial data $U_0\in H^{\frac12+\alpha}(\mathbb{R}^3)$, for the case of $\alpha\in(0,\frac{1}{2})$. By following the proof of \cite[Theorem 1.2]{KJ} line by line, one can show that, on a sufficiently rich probability space, there exists a stochastic strong solution. Furthermore, using similar arguments to those in Lemma \ref{weak-strong} or in Lemma \ref{lemma: wsuniqueness} below, one can show that two strong solutions with the same initial data agree so long as both exist, which implies the existence of a \emph{maximal} strong solution $(\overline{U}^\epsilon, \tau^\epsilon)$ in the sense that all other strong solutions are initial segments of it. \subsection{\textbf{A Space $\XX_{\rm{st,loc}}$ of regular local paths}}\label{subsec: XXstloc} We now define a topological space $\XX_{\rm st, loc}$ of regular partial trajectories in order to describe the large deviations of $\overline{U}^\epsilon$. For any $0<\tau\le T$, let $\XX_{{\rm st},\tau}$ be the paths regular on $[0,\tau)$: $$ \XX_{{\rm st},\tau}:=C([0,\tau),H^1(\mathbb{T}^3))\cap L^2_{\rm loc}([0,\tau),H^2(\mathbb{T}^3)).$$  We then set \begin{equation}\label{eq: def XXstloc} \XX_{\rm st,loc}:=\left\{\star\right\}\sqcup \bigsqcup_{\tau\in (0,T]} \XX_{{\rm st},\tau}\end{equation} where we have included a trivial trajectory $\star$, defined only on $0$ time. As a convention, we will understand the time of existence $\tau=\tau(u)$ to be part of the data of an element $\overline{u} \in \sqcup_{\tau\in (0,T]} \XX_{{\rm st},\tau}$. Finally, we write $\preceq$ for the partial order on $\XX_{\rm st, loc}$ for partial segments: $\overline{v}\preceq \overline{u}$ if $\tau(\overline{v})\le \tau(\overline{u})$ and $\overline{v}(t)=\overline{u}(t)$ for all $0\le t<\tau(\overline{v})$.   \medskip \\ We equip the space $\XX_{\rm st, loc}$ with a family of functions which describe the (potential) loss of regularity close to the maximum time of definition. For every $M>0$, let $\tau_M: \sqcup_{0<\tau\le T}\XX_{{\rm st},\tau}\rightarrow [0,T]$ be defined by
\begin{align}\label{stoppingtimemap}
	\tau_{M}(\overline{u}):=&\inf\Big\{t\in[0,T]:\|\overline{u}(t)\|_{H^1(\mathbb{T}^3)}\geq M,\ \text{or }\|\overline{u}\|_{L^2([0,t];H^2(\mathbb{T}^3))}\geq M\Big\}\wedge T
\end{align}
and define the $H^1(\mathbb{T}^3)$-$H^2(\mathbb{T}^3)$ blow-up time by
\begin{align}\label{explosiontime}
\tau^\star(\overline{u}):=&\lim_{M\rightarrow+\infty}\tau_{M}(\overline{u}).
\end{align} We set $\tau_M(\star)=\tau^\star(\star)=0$. By an abuse of notation, we will use the same notation for the functions $\tau_M, \tau^\star$ defined in the same way on $\XX$ with $u$ in place of $\overline{u}$. For all elements of $\XX_{\rm st,loc}$, it holds that $\tau(\overline{u})\le \tau^\star(\overline{u})$, but this inequality may be strict, for instance by truncating a function strictly before its blowup time. We briefly remark that the maximality of the strong solution implies that $\tau(\ED{\overline{U}})=\tau^\epsilon=\tau^\star(\ED{\overline{U}})$ almost surely. \paragraph{\textbf{Topology of $\XX_{\rm st, loc}$}} We now specify a topology on $\XX_{\rm st,loc}$. First, we write $\mathfrak{t}_{{\rm st},1}$ for the locally convex topology on $\XX_{{\rm st},1}$ induced by the seminorms\footnote{The choice of topology is somewhat arbitrary; the conclusions of the theorem would remain true so long as $\mathfrak{t}_{{\rm st},1}$ is stronger than the topology induced by the same seminorms as defining $\XX$ on compact time intervals $[0,t],t<1$, and is weak enough for the resulting topology on $\XX_{\rm st, loc}$ to be separable. This topology is chosen for concreteness, as it is the natural one with which to equip $\XX_{{\rm st},1}$.} $$ \overline{u}\mapsto \sup_{s\le t}\|\overline{u}(s)\|_{H^1(\mathbb{T}^3)}; \qquad \overline{u}\mapsto \left(\int_0^t \|\overline{u}(s)\|_{H^2(\mathbb{T}^3)}^2ds\right)^{1/2};\qquad t<1. $$ For $u\in \XX_{\rm st, loc}$ with $\tau(u)>\tau$, we define $u^\tau:=(u(s\tau):0\le s<1)\in \XX_{{\rm st},1}$. We define a base of open neighbourhoods of $\overline{u}\neq \star$ by setting \begin{equation}\label{eq: basic opens} \overline{\mathcal{U}}:=\left\{\overline{v}\in \XX_{\rm st, loc}: \tau(\overline{v})>\tau', \hspace{0.2cm}v^{\tau'}\in \overline{\mathcal{V}}, \hspace{0.2cm} \min(\tau_M(\overline{v}),\tau(\overline{v}))<\tau(\overline{u})+\delta  \right\} \end{equation} where $\tau'$ ranges over $(0, \tau(\overline{u}))$, $\overline{\mathcal{V}}$ runs over $\mathfrak{t}_{{\rm st},1}$-neighbourhoods of $u^{\tau'}$, $\delta>0$ and $M>0$. We take a set of open neighbourhoods of $\star$ to be $\{\overline{v}:\min(\tau(\overline{v}),\tau_M(\overline{v}))<\epsilon\}, \epsilon>0$. This topology is chosen so that $\overline{u}^n\to \overline{u}\neq \star$ if and only if $ \liminf_n \tau(\overline{u}^n)\ge \tau(\overline{u})$ and if $\overline{u}^n\to \overline{u}$ in $C_tH^1_{x}\cap L^2_{t}H^2_x$ on compact time intervals in $[0,\tau(\bar{u}))$ but on no larger time intervals. The sequences converging to the trivial path $\star$ are characterised by the property  that $\min(\tau(\overline{u}^n),\tau_M(\overline{u}^n))\to 0$ for any fixed $M$. The following basic properties can easily be checked, which allows us to use the entropy method Lemma \ref{entropymethod}. \begin{lemma}The topological space $\XX_{\rm st, loc}$ is separable and Hausdorff. \end{lemma}  

\subsection{Relationship of Strong and Weak Solutions}	Before we prove Theorem \ref{thm: LDPstrong}, we first show a technical lemma, which will allow us to make arguments about the large deviations behaviour of  $\ED{\overline{U}}$ related to the arguments in the rest of the paper, which relate only to stochastic Leray solutions.
	\begin{lemma}\label{lemma: wsuniqueness}
(Weak-Strong uniqueness) Let $T>0, M>0$. For any strong solution $(\ED{\widetilde{U}},\widetilde{W})$ defined on a stochastic basis $(\widetilde{\Omega},\widetilde{\mathcal{F}},\{\widetilde{\mathcal{F}}(t)\}_{t\in[0,T]},\widetilde{\mathbb{P}})$, there exists a strong solution $\ED{\overline{U}}$ and a stochastic Leray solution $\ED{U}$, defined on a stochastic basis $({\Omega},{\mathcal{F}},\{{\mathcal{F}}(t)\}_{t\in[0,T]},{\mathbb{P}})$, such that $\ED{\overline{U}}=_\mathrm{d}\ED{\widetilde{U}}$ and, almost surely, $\ED{\overline{U}}\preceq \ED{U}$.
\end{lemma} Since the laws of $\ED{\widetilde{U}}, \ED{\overline{U}}$ are the same, so are the large deviations, and we lose nothing moving from $\ED{\widetilde{U}}$, which was an \emph{ a priori } arbitrary strong solution, to the copy $\ED{\overline{U}}$ which is an initial segment of a stochastic Leray solution.  In the sequel we may therefore work only with strong solutions which are the initial segments of a stochastic Leray solution driven by the same divergence-free white noise $W$.
\begin{proof} Let us fix a given $\ED{\widetilde{U}}, \widetilde{W}$. For this divergence-free white noise $\widetilde{W}$, we can repeat Lemma \ref{lemma: properties of galerkin} to construct a sequence of Galerkin approximations $\EDM{\widetilde{U}}$ with initial data $\EDM{\widetilde{U}}(0):=P_m\ED{\widetilde{U}}(0)$. Repeating the argument of Proposition \ref{existence}, it follows that $$(\ED{\widetilde{U}}, \EDM{\widetilde{U}}, \{\widetilde{\beta}^{\zeta,i}:\zeta\in \mathcal{B}, 1\le i\le 3\}, \widetilde{M}^m)_{m\ge 1}$$ are tight in $\XX_{\rm st, loc}\times\XX^+\times C([0,T], \mathbb{R}^{\mathcal{B}\times \{1,2,3\}})\times C([0,T])$, where, as before, $\widetilde{\beta}^{\zeta,i}, \zeta\in \mathcal{B}, 1\le i\le 3$ are the 1d-Brownian motions driving $\widetilde{W}$, and $\widetilde{M}^m$ are the stochastic integrals appearing in the energy (in)equality. We may therefore use the Prohorov and Skorokhod theorems in this (further augmented) space to find a new stochastic basis $(\Omega, \mathcal{F},\{\mathcal{F}(t)\}_{t\in[0,T]}, \mathbb{P})$, on which are defined  $(\overline{U}^{\epsilon, (m)}_\delta, \EDM{{U}}, \{{\beta}^{\zeta,i, m}:\zeta\in \mathcal{B}, 1\le i\le 3\}, {M}^m)$ with the same laws as the original objects, and converging almost surely to a limit $(\ED{\overline{U}}, \ED{{U}}, \{{\beta}^{\zeta,i}:\zeta\in \mathcal{B}, 1\le i\le 3\}, {M})$. The identification of $\ED{U}$ as a stochastic Leray solution, driven by the divergence-free white noise $W$ corresponding to $\beta^{\zeta,i}, \zeta\in \mathcal{B}, 1\le i\le 3$ and starting at $\ED{\overline{U}}(0)$ is as in Proposition \ref{existence}. For the strong solution, since each $(\overline{U}^{\epsilon, (m)}_\delta, \{{\beta}^{\zeta,i,m}:\zeta\in \mathcal{B}, 1\le i\le 3\})$ has the same law under $\mathbb{P}$ as the given strong solution $(\ED{\widetilde{U}}, \{\widetilde{\beta}^{\zeta,i}, \zeta\in \mathcal{B}, 1\le i\le 3\})$ under $\widetilde{\mathbb{P}}$, we may pass to the limit in the definition to see that $\ED{\overline{U}}$ has the same law as $\ED{\widetilde{U}}$ and is a strong solution driven by $W$. \medskip \\ It remains to show that $\ED{\overline{U}}$ is almost surely an initial segment of $\ED{U}$, for which we will follow the deterministic weak-strong uniqueness argument in \cite[Theorem 6.10]{RRS}. Since $\ED{\overline{U}}$ is a strong solution, we have sufficient regularity to apply It\^o's formula \cite[Theorem 3.1]{Kry} to $\langle \ED{U}(t),\ED{\overline{U}}(t)\rangle$. Combining with the energy inequalities for $\ED{\overline{U}}, \ED{U}$ and calculating $\|\ED{U}-\ED{\overline{U}}\|_H^2$, both the martingale terms and the contributions from the It\^o correction cancel, and we find, $\mathbb{P}$-almost surely,
\begin{equation}\begin{split}
\frac{1}{2}\|\ED{U}(t)-\ED{\overline{U}}(t)\|_{H}^2& +\int_0^t\|\nabla(\ED{U}-\ED{\overline{U}})(s)\|_{H}^2ds\\
\le &\int_0^t\int_{\mathbb{T}^3}\nabla(\ED{U}-\ED{\overline{U}}):((\ED{U}-\ED{\overline{U}})\otimes\ED{\overline{U}})dxds,
\end{split}\end{equation}
for $t\in[0,\tau^\epsilon)$. The same functional estimates as for the deterministic weak-strong uniqueness, see for example \cite[Theorem 6.10]{RRS}, show that, $\mathbb{P}$-almost surely,
\begin{align*}
\int_0^t\int_{\mathbb{T}^3}\nabla(\ED{U}-\ED{\overline{U}}):((\ED{U}-\ED{\overline{U}})\otimes\ED{\overline{U}})dxds\leq&\int_0^t\|\ED{\overline{U}}\|_{L^{\infty}(\mathbb{T}^3)}\|\ED{U}-\ED{\overline{U}}\|_H\|\ED{U}-\ED{\overline{U}}\|_Vds\\
\leq& \frac{1}{2}\int_0^t\|\ED{\overline{U}}\|_{H^2(\mathbb{T}^3)}^2\|\ED{U}-\ED{\overline{U}}\|_H^2ds+\frac{1}{2}\int_0^t\|\ED{U}-\ED{\overline{U}}\|_V^2ds.
\end{align*}
Since $\ED{\overline{U}}\in L^2_{\rm loc}([0,\tau^\epsilon), H^2(\mathbb{T}^3))$, we may use Gr\"onwall's lemma to see that, almost surely, for all $t<\tau^\epsilon$, $\ED{\overline{U}}(t)=\ED{U}(t)$, as desired. \end{proof}	
		\subsection{Proof of the Theorem}
\begin{proof}[Proof of Theorem \ref{thm: LDPstrong}]
As in the theorem, let $\overline{u} \in \XX_{\rm st, loc}$, and let $\overline{v}\preceq \overline{u}$ be an initial segment. As remarked below Lemma \ref{lemma: wsuniqueness}, we may assume that the strong solution $\ED{\overline{U}}$ is an initial segment of a stochastic Leray solution of (\ref{SNS-1}), driven by a common divergence-free white noise $W$, and let $\overline{u}^\star$ be the maximal strong solution to the Navier-Stokes equations started from $\overline{u}^\star_0$. In Steps 1-2 below we will show that, if $\star \prec \overline{v}\prec \overline{u}$ and there exists a sequence of probability measures $\mathbb{Q}^\epsilon\ll \mathbb{P}$ under which $\ED{\overline{U}}\to \overline{u}$ in $\XX_{\rm st, loc}$ with $\limsup_{\epsilon}\epsilon {\rm Ent}(\mathbb{Q}^\epsilon|\mathbb{P})<\infty$, then we can construct $\tilde{\mathbb{Q}}^\epsilon$ under which $\ED{\overline{U}}\to \overline{v}$ with $$\limsup_{\epsilon}\epsilon {\rm Ent}(\tilde{\mathbb{Q}}^\epsilon|\mathbb{P})\le \limsup_{\epsilon}\epsilon {\rm Ent}(\mathbb{Q}^\epsilon|\mathbb{P}). $$  The assertion (\ref{eq: cheap blowup}) then follows by using both directions of the entropy method in Lemma \ref{entropymethod}, and observing that the conclusion is trivial if no such sequence of $\mathbb{Q}^\epsilon$ exist, since the right-hand side of (\ref{eq: cheap blowup}) is then $-\infty$. We will then discuss the necessary modification to show (\ref{eq: instant blowup}) at the end, and (\ref{eq: inf rate}) will follow from the ideas already established for the upper bound of Theorem \ref{thrm: LDP}.

\textbf{Step 1. Construction of Paths Rapidly Losing Regularity} Let us suppose that $\star\prec \overline{v}\prec \overline{u}$; to shorten notation, write $\tau:=\tau(\bar{u})$ and $\tau':=\tau(\overline{v})$. \bigskip \\ Let us observe that, under any change of measures $\mathbb{Q}^\epsilon$ which make $\ED{\overline{U}}\to \overline{u}$ in $\XX_{\rm st, loc}$, the stochastic Leray solution $\ED{U}$ which extends $\ED{\overline{U}}$ concentrates on paths $u\in \XX$ which extend $\overline{u}$. In particular, if $\mathbb{Q}^\epsilon$ are chosen so that $\epsilon {\rm Ent}(\mathbb{Q}^\epsilon|\mathbb{P})$ remains bounded, we may repeat the arguments of Theorem \ref{thrm: LDP} on $[0,\tau]$ to see that there exists $g\in L^2([0,\tau],\mathcal{M})$ such that $\overline{u}$ solves (\ref{control}) up until its time of existence $\tau$. \bigskip \\  We will now exhibit a sequence of trajectories $\overline{v}^{(n)}$ with the same blowup time $\tau(\overline{v}^{(n)})=\tau(\overline{u})$ and $\overline{v}^{(n)}\to \overline{v}$ with forcings $g^{(n)}-g\to 0$ in $L^2([0,\tau],\mathcal{M})$. Let $\overline{w}^{(n)}=\overline{w}^{(n,1)}-\overline{w}^{(n,2)}$ be the plane wave, defined on $t\in [\tau',\tau)$, given by setting
\begin{equation}\label{eq: bad u 1}
\overline{w}^{(n,1)}(t,x)=n^{-1/2}(0,\sin(2\pi nx_1),0)\exp\Big(-4\pi^2 n^2\Big|t-\tau'-\frac{1}{n}\Big|\Big);
\end{equation} \begin{equation}\label{eq: bad u 1'}
\overline{w}^{(n,2)}(t,x)=n^{-1/2}(0,\sin(2\pi nx_1),0)\exp\Big(-4\pi^2n-4\pi^2 n^2\Big|t-\tau'\Big|\Big)	
\end{equation} and consider the perturbed trajectory \begin{equation}
	\label{eq: bad} \bar{v}^{(n)}(t):=\begin{cases}\overline{u}(t), & 0\le t\le \tau'; \\ \overline{u}(t)+\bar{w}^{(n)}(t) & \tau'<t\le \tau. \end{cases}
\end{equation}
First, we note that $\overline{v}^{(n)}$ is continuous in time at the intermediate time $t=\tau'$ since the contributions from $\overline{w}^{(n,1)}, \overline{w}^{(n,2)}$ cancel. Since $\overline{u}(t)$ is bounded in $H^1(\mathbb{T}^3)$ away from $t=\tau$, it follows that  $\|\overline{v}^{(n)}(\tau'+n^{-1})\|_{H^1(\mathbb{T}^3)}\rightarrow \infty$ as $n\rightarrow\infty$.  Since $\overline{v}^{(n)}=\overline{u}=\overline{v}$ on $[0,\tau']$, it follows that $\overline{v}^{(n)}\to \overline{v}$ in $\XX_{\rm st, loc}$. Letting $g\in L^2([0,\tau],\mathcal{M})$ be the control for $\overline{u}$ as remarked above, a computation shows that $\overline{v}^{(n)}$ solves the skeleton equation (\ref{eq: sk}) on $[0,\tau)$ with control \begin{equation}\label{eq: gn} g^{(n)}(s,x)=g(s,x)+2\nabla w^{(n,1)}(s,x)I_{[\tau',\tau'+n^{-1}]}(s)+ (w^{(n)}\otimes \overline{u}(s,x)+\overline{u}\otimes w^{(n)}(s,x))I_{[\tau', \tau)}(s). \end{equation} We further see that \begin{equation}\label{eq: bad g} \|g^{(n)}-g\|_{L^2([\tau',\tau],\mathcal{M})}\to 0 \end{equation}  by direct computation on the first new term of (\ref{eq: gn}), and on the cross terms using that $\overline{u}\in L^\infty([0,\tau),H)$ and $\overline{w}^{(n)}\to 0$ in $L^\infty([0,T],L^\infty(\mathbb{T}^3))$.

\textbf{Step 2. Construction of a Change of Measures, Convergence and Entropy Estimate} We now start from the given change of measures $\mathbb{Q}^\epsilon\ll \mathbb{P}$ under which $\ED{\overline{U}}\to \overline{u}$ in $\XX_{\rm st, loc}$, and construct a new change of measures $\tilde{\mathbb{Q}}^\epsilon\ll \mathbb{P}$ under which $\ED{U}\to \overline{v}$. We set \begin{equation}
	Z^{n,\epsilon}_t:=\exp\left(-\frac{1}{\sqrt{\epsilon}}\int_{t\land \tau'}^{t}\langle \sqrt{Q_\delta}g^{(n)}, dW(s)\rangle-\frac1{2\epsilon}\int_{t\land \tau'}^t \|\sqrt{Q_\delta}g^{(n)}(s)\|_H^2 ds \right)
\end{equation} and define $\tilde{\mathbb{Q}}^{n,\epsilon}$ by setting, for $n>(\tau-\tau')^{-1}$, \begin{equation}
	\frac{d\tilde{\mathbb{Q}}^{n,\epsilon}}{d\mathbb{P}}:=\left.\frac{d\mathbb{Q}^\epsilon}{d\mathbb{P}}\right|_{\mathcal{F}(\tau)}Z^{n,\epsilon}_{\tau'+n^{-1}}.
\end{equation}  

\textbf{Step 2a. Convergence of the strong solution $\ED{\overline{U}}$}  We now deal with the convergence of the strong solution under the measures $\tilde{\mathbb{Q}}^{n(\epsilon), \epsilon}$ as $\epsilon\to 0$, for some sequence $n(\epsilon)\to \infty$ to be chosen. We consider the time intervals $[0,\tau'], [\tau', \tau'+n^{-1}]$ separately. On the initial time interval $[0,\tau']$, we observe that the law of $(\ED{\overline{U}}(t), 0\le t<\tau^\epsilon \land \tau')$ is the same under $\tilde{\mathbb{Q}}^{n, \epsilon}$ as under $\mathbb{Q}^\epsilon$ and, by definition of $\mathbb{Q}^\epsilon$ and the convergence in $\XX_{\rm st, loc}$, $\mathbb{Q}^\epsilon(\tau^\epsilon>\tau')\to 1$ and, for any $\mathfrak{t}_{{\rm st},1}$ neighbourhood  $\mathcal{V}\ni \overline{u}^{\tau'}$, $\mathbb{Q}^\epsilon((\ED{\overline{U}})^{\tau'}\in \mathcal{V})\to 1$. Both statements therefore remain true with $\tilde{\mathbb{Q}}^{n,\epsilon}$ in place of $\mathbb{Q}^\epsilon$, and indeed uniformly in $n$. \bigskip \\ We next consider the time interval $[\tau', \tau+n^{-1}]$, as $\epsilon\to 0$ with $n$ fixed. On this time interval, we repeat the arguments of Lemma \ref{lemma: existence of tilted} to see that, under the measure $\tilde{\mathbb{Q}}^{n,\epsilon}$, the weak solution $\ED{U}$ is a stochastic Leray solution to the equation \begin{equation}\label{scontrol-strong}
\partial_t\ED{U}=-A \ED{U}-\mathbf{P}(\ED{U}\cdot\nabla)\ED{U}-\sqrt{\epsilon}\nabla\cdot\tilde{\xi}_\delta-\nabla\cdot \Big(\sqrt{Q_{\delta}}g^{(n)}\Big)
\end{equation} for $g^{(n)}$ as in Step 1, for a $\tilde{\mathbb{Q}}^{n,\epsilon}$ divergence-free white noise $\tilde{\xi}$, and with $\ED{U}(\tau')\to \overline{u}(\tau')$ weakly in $H$ thanks to the argument of the previous paragraph. Since the paths $\overline{v}^{(n)}$ constructed belong to $L^\infty_tH^1_x$ on $[\tau', \tau'+n^{-1}]\subset \subset [0,\tau)$, we may use the Sobolev embedding $H^1(\mathbb{T}^3)\hookrightarrow L^6(\mathbb{T}^3)$ to repeat the proof of Lemmata \ref{weak-strong} and \ref{stationary} on the time interval $[\tau', \tau'+n^{-1}]$ to conclude that $\ED{U}\to \overline{v}^{(n)}$ as $\epsilon \to 0$ with $n$ fixed, in the topology of $\XX|_{[\tau', \tau'+n^{-1}]}$, in $\mathbb{Q}^{n, \epsilon}$-probability. Using the lower semicontinuity of $\|\cdot\|_{H^1(\mathbb{T}^3)}$ for the weak convergence in $H$, we may choose $\epsilon_n\downarrow 0$ such that, for all $0<\epsilon<\epsilon_n$, $$\tilde{\mathbb{Q}}^{n, \epsilon}\left(\|\ED{U}(\tau'+n^{-1})\|_{H^1(\mathbb{T}^3)}>M_n\right)>1-n^{-1} $$ where we choose $$M_n:=\|\overline{v}^{(n)}(\tau'+n^{-1})\|_{H^1(\mathbb{T}^3)}-n^{-1}. $$ We now relate this to the behaviour of the strong solution $\ED{\overline{U}}$. On the event displayed above, $\tau_{M_n}(\ED{{U}})<\tau'+n^{-1}$, and we consider separately the cases $\tau^\epsilon\le  \tau'+n^{-1}, \tau^\epsilon> \tau'+n^{-1}$. On the first event, since $\tau_M(\ED{\overline{U}})\uparrow \tau^\epsilon$, it follows that $\tau_{M_n}(\ED{\overline{U}})<\tau^\epsilon\le \tau'+n^{-1}$. On the other hand, if $\tau^\epsilon> \tau'+n^{-1}$, we use the fact that $\ED{\overline{U}}$ is an initial segment of the weak solution $\ED{U}$ up to $\tau^\epsilon$ to reach the same conclusion $\tau_{M_n}(\ED{\overline{U}})<\tau'+n^{-1}$. In particular, we have shown that $$\tilde{\mathbb{Q}}^{n, \epsilon}\left(\tau_{M_n}(\ED{\overline{U}})<\tau'+n^{-1}\right)\to 1 $$ as $\epsilon\to 0$ with $n$ fixed. We finally invert the relationship between $\epsilon, n$ by choosing $n(\epsilon)=n$ for $\epsilon\in (\epsilon_{n+1}, \epsilon_n)$ and take $\tilde{\mathbb{Q}}^\epsilon:=\tilde{\mathbb{Q}}^{n(\epsilon),\epsilon}$ to see that $M_{n(\epsilon)}\to \infty$ and \begin{equation}
	\label{eq: accumulation of blowups} \tilde{\mathbb{Q}}^\epsilon\left(\tau_{M_{n(\epsilon)}}(\ED{\overline{U}})<\tau(\overline{v})+n^{-1}\right)\to 1.
\end{equation} Together with the analysis on $[0, \tau(\overline{v})]$, we conclude that, under $\tilde{\mathbb{Q}}^\epsilon$, we have the full convergence $\ED{\overline{U}}\to \overline{v}$ in $\XX_{\rm st, loc}$.  \\

\textbf{Step 2b. Entropy Estimate} It remains to estimate the relative entropy of the new measures ${\rm Ent}(\tilde{\mathbb{Q}}^\epsilon|\mathbb{P})$ in terms of ${\rm Ent}(\tilde{\mathbb{Q}}^\epsilon|\mathbb{P})$, and to show that any increase is at most of sub-leading order. Using the definition of $\tilde{\mathbb{Q}}^\epsilon$, we find \begin{equation}
 \begin{split}\label{eq: new entropy}
 	{\rm Ent}\left(\tilde{\mathbb{Q}}^\epsilon|\mathbb{P}\right)& =\mathbb{E}_{\tilde{\mathbb{Q}}^\epsilon}\left[\log \left.\frac{d\mathbb{Q}^\epsilon}{d\mathbb{P}}\right|_{\mathcal{F}(\tau)}+ \log Z^{n(\epsilon), \epsilon}_{\tau'+n(\epsilon)^{-1}}\right] \\ &= \mathbb{E}_{\tilde{\mathbb{Q}}^\epsilon|_{\mathcal{F}(\tau)}}\left[\log \left.\frac{d\mathbb{Q}^\epsilon}{d\mathbb{P}}\right|_{\mathcal{F}(\tau)}\right]+ \mathbb{E}_{\tilde{\mathbb{Q}}^\epsilon}\left[\log Z^{n(\epsilon), \epsilon}_{\tau'+n(\epsilon)^{-1}}\right] \\[1ex] & ={\rm Ent}\left(\mathbb{Q}^\epsilon|_{\mathcal{F}(\tau)}\big|\mathbb{P}|_{\mathcal{F}(\tau)}\right)+\frac1{2\epsilon}\int_{\tau'}^{\tau'+n(\epsilon)^{-1}}\|\sqrt{Q_\delta}g^{(n(\epsilon))}\|_{\mathcal{M}}^2 dt.
 \end{split}	
 \end{equation} The first term is bounded above by ${\rm Ent}(\mathbb{Q}^\epsilon|\mathbb{P})$ thanks to the information processing inequality, while multiplying the second term by $\epsilon$ gives $$ \frac1{2}\int_{\tau'}^{\tau'+n(\epsilon)^{-1}}\|\sqrt{Q_\delta}g^{(n(\epsilon))}\|_{\mathcal{M}}^2 dt \le \int_{\tau'}^{\tau'+n(\epsilon)^{-1}} \|g\|_{\mathcal{M}}^2 dt + \|g^{(n(\epsilon))}-g\|_{L^2([\tau',\tau],\mathcal{M})}^2. $$ Of these, the first term converges to $0$ by dominated convergence in $dt$, and the latter converges to $0$ thanks to (\ref{eq: bad g}). Returning to (\ref{eq: new entropy}), we have proven that \begin{equation}
 	\limsup_{\epsilon} \epsilon{\rm Ent}\left(\tilde{\mathbb{Q}}^\epsilon|\mathbb{P}\right)\le \limsup_{\epsilon} \epsilon{\rm Ent}\left({\mathbb{Q}}^\epsilon|\mathbb{P}\right)
 \end{equation} and the proof of (\ref{eq: cheap blowup}) is complete in light of the reduction discussed above Step 1. In order to prove (\ref{eq: instant blowup}), one repeats the same construction, starting with the original probability measure $\mathbb{P}$ instead of $\mathbb{Q}^\epsilon$ and replacing $\tau'=0$, so that the dynamics are modified only on $[0, n(\epsilon)^{-1}]$. In this way, no term analagous to the first term in (\ref{eq: new entropy}) appears and the new measures $\tilde{\mathbb{Q}}^\epsilon$ which force convergence to the trivial path $\star$ have ${\rm Ent}(\tilde{\mathbb{Q}}^\epsilon|\mathbb{P})=\mathfrak{o}(\epsilon^{-1}).$
 
\textbf{Step 3. Large Deviations at Slower Speeds} We now deal with the other part of the dichotomy, where we show that any other speed $r_\epsilon \ll \epsilon^{-1}$ only allows limit paths $\overline{v}\preceq \overline{u}^\star$, which we recall is the unique maximal strong solution starting at ${u}^\star_0$. Let us fix $\overline{v}\in \XX_{\rm st, loc}$ with $\overline{v}\npreceq \overline{u}^\star$. Thanks to the uniqueness of strong solutions up to their $H^1(\mathbb{T}^3)-H^2(\mathbb{T}^3)$ blowup, as in Lemmata \ref{weak-strong}, it follows either that there exists $\psi \in H, \theta>0$ with $\langle \psi, \bar{v}(0)-u^\star_0\rangle>2\theta$, \textbf{or} there exists $\varphi\in C^\infty_{\mathrm{d.f.}}([0,T]\times\mathbb{T}^3)$, supported on $t\in [0,\tau(\overline{v}))$ such that $\Lambda(\varphi, \overline{v})>\theta$, \textbf{or} that there exists $t<\tau(\overline{v})$ and $\theta>0$ such that the energy inequality fails at time $t$ by at least $2\theta$: \begin{equation}
	\frac12\|\overline{v}(t)\|_H^2+\int_0^t\|\overline{v}(s)\|_V^2 ds>\frac12\|\bar{u}(0)\|_H^2+2\theta.
\end{equation} In each of the three cases, we now construct an open set $\overline{\mathcal{U}}\ni \overline{v}$ such that \begin{equation}\label{eq: strong ldp slower speed nontrivial}\limsup_{\epsilon\to 0}\epsilon \log \mathbb{P}\left(\ED{\overline{U}}\in \overline{\mathcal{U}}\right)<0. \end{equation} Since $\epsilon r_\epsilon\to 0$, it follows that $$\limsup_{\epsilon\to 0}r_\epsilon^{-1} \log \mathbb{P}\left(\ED{\overline{U}}\in \overline{\mathcal{U}}\right)=-\infty $$ which concludes the theorem. 

\textbf{Step 3a. Different Initial Data} We first deal with the case where $\overline{v}(0)\neq u^\star_0$. In this case, we pick $\theta>0, \psi\in H$ such that $\langle \psi, \overline{v}(0)-u^\star_0\rangle>2\theta$. Recalling that $\overline{v}\neq \star$, it follows that $\tau(\overline{v})>0$, and we choose the open set to be \begin{equation}
	\overline{\mathcal{U}}:=\left\{\overline{w}\in \XX_{\rm st, loc}: \tau(\overline{w})>\frac{\tau(\overline{v})}{2}, \langle \psi, \overline{w}(0)-u^\star_0\rangle>\theta \right\}
\end{equation} which can be checked to be open using the basic open sets specified in (\ref{eq: basic opens}). Since $\ED{\overline{U}}(0)$ satisfies a good LDP in the weak topology of $H$ with rate function $\mathcal{I}_0$, it follows that \begin{equation} \begin{split}
	\limsup_\epsilon \epsilon \log \PP\left(\ED{\overline{U}}\in \overline{\mathcal{U}}\right)& \le 	\limsup_\epsilon \epsilon \log \PP\left(\langle \psi, \ED{\overline{U}}(0)-u^\star_0\rangle\ge \theta\right) \\ & \le -\inf\left\{\mathcal{I}_0(u): u\in H, \langle \psi, u-u^\star_0\rangle \geq\theta\right\}<0\end{split}
\end{equation} where in the final line, we observe that the closed set does not contain $u^\star_0$, and we have used the goodness of the rate function $\mathcal{I}_0$ and the hypothesis that $\mathcal{I}_0$ has the unique zero at $u^\star_0$ to see that the infimum is strictly positive. 

\textbf{Step 3b. Evolution different from Navier-Stokes} The second possibility described above is if the partial path $\overline{v}$ fails to be a weak solution of the Navier-Stokes equations on some interval $[0, t]\subset \subset [0,\tau(\overline{v}))$. In this case, there exists a  test function $\varphi\in C^\infty_{\rm d.f.}([0,T]\times\mathbb{T}^3)$ supported on $[0,t]\times \mathbb{T}^3$ so that, extending the definition (\ref{Fu}) in a natural way to partial paths and functions on their supports, we have $\Lambda(\varphi, \overline{v})>2\theta>0$. We now take the open set \begin{equation}
	\overline{\mathcal{U}}:=\left\{\overline{w}\in \XX_{\rm st, loc}: \tau(\overline{w})>t, \Lambda(\varphi, \overline{w})>\theta\right\}.
\end{equation} To estimate the probability that $\ED{\overline{U}}\in \overline{\mathcal{U}}$, we observe that on this event, the stochastic Leray solution $\ED{U}$ with $\ED{\overline{U}}\preceq \ED{U}$ lies in the open set in $\XX$ \begin{equation}
	\mathcal{U}:=\left\{w\in \XX: \Lambda(\varphi, w)>\theta\right\}
\end{equation} and, repeating the arguments leading to (\ref{eq: UB with phi psi}), we have \begin{equation}
	\begin{split}\limsup_\epsilon \epsilon \log \mathbb{P}\left(\ED{\overline{U}}\in \overline{\mathcal{U}}\right) & \le \limsup_\epsilon \epsilon \log \mathbb{P}\left(\ED{{U}}\in \mathcal{{U}}\right) <0.\end{split}
\end{equation} 

\textbf{Step 3c. Violation of the Energy Identity} The previous two substeps prove the desired assertion (\ref{eq: strong ldp slower speed nontrivial}) in all cases \emph{except} when $\overline{v}$ is a weak solution to Navier-Stokes starting at $\overline{v}(0)=u^\star_0$. In the final case, in which the energy identity is violated at some time $t<\tau(\overline{v})$, we set \begin{equation}
	\overline{\mathcal{U}}:=\left\{\overline{w}\in \XX_{\rm st, loc}: \tau(\overline{w})>t, \frac12 \|\overline{w}(t)\|_H^2+\int_0^t\|\overline{w}(s)\|_V^2 ds>\|u^\star_0\|_H^2+\theta \right\}
\end{equation} which can be seen to be open in $\XX_{\rm st, loc}$ using the lower semicontinuity of $\overline{w}\mapsto \frac12 \|\overline{w}(t)\|_H^2+\int_0^t\|\overline{w}(s)\|_V^2 ds$ in $\XX_{{\rm st},\tau}$. The argument proceeds as in the previous case to yield, in light of Lemma \ref{lemma: improved UB'}, \begin{equation}
	\begin{split}
		\limsup_\epsilon \epsilon \log \PP\left(\ED{\overline{U}}\in \overline{\mathcal{U}}\right) &\le \limsup_\epsilon \epsilon \log \PP\left(\frac12\|\ED{U}(t)\|_H^2+\int_0^t\|\ED{U}(s)\|_V^2 ds> \frac12 \|u^\star_0\|_H^2+\theta\right) \\& <0.
	\end{split}
\end{equation} This completes the proof of (\ref{eq: strong ldp slower speed nontrivial}) in each of the three possible cases, and hence the theorem.  \end{proof}

\noindent{\bf  Acknowledgements}\quad Benjamin Gess acknowledges support by the Max Planck Society through the Research Group ``Stochastic Analysis in the Sciences". This work was funded by the Deutsche Forschungsgemeinschaft (DFG, German Research Foundation) via IRTG 2235 - Project Number 282638148, and cofunded by the European Union (ERC, FluCo, grant agreement No. 101088488). Views and opinions expressed are however those of the author(s)only and do not necessarily reflect those of the European Union or of the European Research Council. Neither the European Union nor the granting authority can be held responsible for them.

\bibliographystyle{alpha}
\bibliography{LDPNSre.bib}

\newcommand{\etalchar}[1]{$^{#1}$}
\begin{thebibliography}{RPDO{\etalchar{+}}21}

\bibitem[ABC22]{albritton2022non}
Dallas Albritton, Elia Bru{\'e}, and Maria Colombo.
\newblock Non-uniqueness of leray solutions of the forced navier-stokes
  equations.
\newblock {\em Annals of Mathematics}, 196(1):415--455, 2022.

\bibitem[Aub63]{aubin1963theoreme}
Jean-Pierre Aubin.
\newblock Un th{\'e}oreme de compacit{\'e}.
\newblock {\em CR Acad. Sci. Paris}, 256(24):5042--5044, 1963.

\bibitem[AX06]{AX}
Anna Amirdjanova and Jie Xiong.
\newblock Large deviation principle for a stochastic {N}avier-{S}tokes equation
  in its vorticity form for a two-dimensional incompressible flow.
\newblock {\em Discrete Contin. Dyn. Syst. Ser. B}, 6(4):651--666, 2006.

\bibitem[Aze78]{AZ}
R~Azencott.
\newblock Grandes d{\'e}viations.
\newblock {\em Ecole d'Et{\'e} de Probabilit{\'e}s de Saint-Flour VIII-1978},
  774, 1978.

\bibitem[BBMN10]{bellettini2010gamma}
Giovanni Bellettini, Lorenzo Bertini, Mauro Mariani, and Matteo Novaga.
\newblock $\gamma$-entropy cost for scalar conservation laws.
\newblock {\em Archive for rational mechanics and analysis}, 195:261--309,
  2010.

\bibitem[BC17]{BC}
Zdzislaw Brze\'{z}niak and Samdra Cerrai.
\newblock Large deviations principle for the invariant measures of the 2{D}
  stochastic {N}avier-{S}tokes equations on a torus.
\newblock {\em J. Funct. Anal.}, 273(6):1891--1930, 2017.

\bibitem[BC20]{berselli2020energy}
Luigi~C Berselli and Elisabetta Chiodaroli.
\newblock On the energy equality for the 3d navier--stokes equations.
\newblock {\em Nonlinear Analysis}, 192:111704, 2020.

\bibitem[BD23]{DA}
Dominic Breit and Alan Dodgson.
\newblock Space-time approximation of local strong solutions to the 3d
  stochastic navier-stokes equations. arxiv: 2211.17011, 2023.

\bibitem[BDSG{\etalchar{+}}15]{BDGJL}
Lorenzo Bertini, Alberto De~Sole, Davide Gabrielli, Giovanni Jona-Lasinio, and
  Claudio Landim.
\newblock Macroscopic fluctuation theory.
\newblock {\em Rev. Modern Phys.}, 87(2):593--636, 2015.

\bibitem[Ben95]{bensoussan1995stochastic}
Alain Bensoussan.
\newblock Stochastic navier-stokes equations.
\newblock {\em Acta Applicandae Mathematica}, 38:267--304, 1995.

\bibitem[BFH18]{BFH}
Dominic Breit, Eduard Feireisl, and Martina Hofmanov\'{a}.
\newblock {\em Stochastically forced compressible fluid flows}, volume~3 of
  {\em De Gruyter Series in Applied and Numerical Mathematics}.
\newblock De Gruyter, Berlin, 2018.

\bibitem[BGW07]{bell2007numerical}
John~B Bell, Alejandro~L Garcia, and Sarah~A Williams.
\newblock Numerical methods for the stochastic landau-lifshitz navier-stokes
  equations.
\newblock {\em Physical Review E}, 76(1):016708, 2007.

\bibitem[Bre18]{breit2018introduction}
Dominic Breit.
\newblock An introduction to stochastic navier--stokes equations.
\newblock {\em New Trends and Results in Mathematical Description of Fluid
  Flows}, pages 1--51, 2018.

\bibitem[BV20]{Bu20}
Tristan Buckmaster and Vlad Vicol.
\newblock Convex integration and phenomenologies in turbulence.
\newblock {\em EMS Surveys in Mathematical Sciences}, 6(1):173--263, 2020.

\bibitem[CD19]{CD}
Sandra Cerrai and Arnaud Debussche.
\newblock Large deviations for the two-dimensional stochastic {N}avier-{S}tokes
  equation with vanishing noise correlation.
\newblock {\em Ann. Inst. Henri Poincar\'{e} Probab. Stat.}, 55(1):211--236,
  2019.

\bibitem[CIS05]{CIS}
B.~S. Cirel'son, I.~A. Ibragimov, and V.~N.~\"{O}ttinger Sudakov.
\newblock {\em Norms of Gaussian sample functions}.
\newblock Springer Berlin Heidelberg, 2005.

\bibitem[CL21]{CL}
Alexey Cheskidov and Xiaoyutao Luo.
\newblock Anomalous dissipation, anomalous work, and energy balance for the
  {N}avier-{S}tokes equations.
\newblock {\em SIAM J. Math. Anal.}, 53(4):3856--3887, 2021.

\bibitem[CT08]{Cao-Titti}
Chongsheng Cao and Edriss~S. Titi.
\newblock Regularity criteria for the three-dimensional {N}avier-{S}tokes
  equations.
\newblock {\em Indiana Univ. Math. J.}, 57(6):2643--2661, 2008.

\bibitem[Der07]{Derrida}
Bernard Derrida.
\newblock Non-equilibrium steady states: fluctuations and large deviations of
  the density and of the current.
\newblock {\em J. Stat. Mech. Theory Exp.}, (7):P07023, 45, 2007.

\bibitem[DFG24]{DFG}
Nicolas Dirr, Benjamin Fehrman, and Benjamin Gess.
\newblock Conservative stochastic pde and fluctuations of the symmetric simple
  exclusion process. arxiv: 2012.02126, 2024.

\bibitem[DPZ14]{da2014stochastic}
Giuseppe Da~Prato and Jerzy Zabczyk.
\newblock {\em Stochastic equations in infinite dimensions}.
\newblock Cambridge university press, 2014.

\bibitem[dVY19]{da2019energy}
Hugo~Beir{\~a}o da~Veiga and Jiaqi Yang.
\newblock On the energy equality for solutions to newtonian and non-newtonian
  fluids.
\newblock {\em Nonlinear Analysis}, 185:388--402, 2019.

\bibitem[dVY20]{da2020shinbrot}
Hugo~Beirao da~Veiga and Jiaqi Yang.
\newblock On the shinbrot's criteria for energy equality to newtonian fluids: a
  simplified proof, and an extension of the range of application.
\newblock {\em Nonlinear Analysis}, 196:111809, 2020.

\bibitem[FG95]{FG95}
Franco Flandoli and Dariusz Gatarek.
\newblock Martingale and stationary solutions for stochastic {N}avier-{S}tokes
  equations.
\newblock {\em Probab. Theory Related Fields}, 102(3):367--391, 1995.

\bibitem[FG16]{FG16}
Peter~K. Friz and Benjamin Gess.
\newblock Stochastic scalar conservation laws driven by rough paths.
\newblock {\em Ann. Inst. H. Poincar\'{e} C Anal. Non Lin\'{e}aire},
  33(4):933--963, 2016.

\bibitem[FG21]{FG21}
Benjamin Fehrman and Benjamin Gess.
\newblock Well-posedness of the dean-kawasaki and the nonlinear dawson-watanabe
  equation with correlated noise. arxiv: 2108.08858, 2021.

\bibitem[FG23]{FG22}
Benjamin Fehrman and Benjamin Gess.
\newblock Non-equilibrium large deviations and parabolic-hyperbolic {PDE} with
  irregular drift.
\newblock {\em Invent. Math.}, 234(2):573--636, 2023.

\bibitem[Fla15]{FF}
Franco Flandoli.
\newblock A stochastic view over the open problem of well-posedness for the
  3{D} {N}avier-{S}tokes equations.
\newblock In {\em Stochastic analysis: a series of lectures}, volume~68 of {\em
  Progr. Probab.}, pages 221--246. Birkh\"{a}user/Springer, Basel, 2015.

\bibitem[FR08]{flandoliromito2008}
Franco Flandoli and Marco Romito.
\newblock Markov selections for the 3{D} stochastic {N}avier-{S}tokes
  equations.
\newblock {\em Probab. Theory Related Fields}, 140(3-4):407--458, 2008.

\bibitem[GH23]{GH}
Benjamin Gess and Daniel Heydecker.
\newblock A rescaled zero-range process for the porous medium equation:
  Hydrodynamic limit, large deviations and gradient flow. arxiv: 2303.11289,
  2023.

\bibitem[Gou07]{GM}
Mathieu Gourcy.
\newblock A large deviation principle for 2{D} stochastic {N}avier-{S}tokes
  equation.
\newblock {\em Stochastic Process. Appl.}, 117(7):904--927, 2007.

\bibitem[G{\v{S}}17]{guillod2017numerical}
Julien Guillod and Vladim{\'\i}r {\v{S}}ver{\'a}k.
\newblock Numerical investigations of non-uniqueness for the navier-stokes
  initial value problem in borderline spaces.
\newblock {\em arXiv preprint arXiv:1704.00560}, 2017.

\bibitem[Hey23]{DH}
Daniel Heydecker.
\newblock {Large deviations of Kac's conservative particle system and energy
  nonconserving solutions to the Boltzmann equation: A counterexample to the
  predicted rate function}.
\newblock {\em The Annals of Applied Probability}, 33(3):1758 -- 1826, 2023.

\bibitem[HW15]{HW}
Martin Hairer and Hendrik Weber.
\newblock Large deviations for white-noise driven, nonlinear stochastic {PDE}s
  in two and three dimensions.
\newblock {\em Ann. Fac. Sci. Toulouse Math. (6)}, 24(1):55--92, 2015.

\bibitem[HZZ19]{hofmanova2019non}
Martina Hofmanov{\'a}, Rongchan Zhu, and Xiangchan Zhu.
\newblock Non-uniqueness in law of stochastic 3d navier--stokes equations.
\newblock {\em arXiv preprint arXiv:1912.11841}, 2019.

\bibitem[HZZ23]{hofmanova2023global}
Martina Hofmanov{\'a}, Rongchan Zhu, and Xiangchan Zhu.
\newblock Global-in-time probabilistically strong and markov solutions to
  stochastic 3d navier--stokes equations: Existence and nonuniqueness.
\newblock {\em The Annals of Probability}, 51(2):524--579, 2023.

\bibitem[JS15]{jia2015incompressible}
Hao Jia and Vladimir Sverak.
\newblock Are the incompressible 3d {N}avier-{S}tokes equations locally
  ill-posed in the natural energy space?
\newblock {\em J. Funct. Anal.}, 268(12):3734--3766, 2015.

\bibitem[Kim10]{KJ}
Jong~Uhn Kim.
\newblock Strong solutions of the stochastic {N}avier-{S}tokes equations in
  {$\Bbb R^3$}.
\newblock {\em Indiana Univ. Math. J.}, 59(5):1853--1886, 2010.

\bibitem[KL57]{KL}
Alexander~A. Kiselev and Olga~A. Lady\v{z}enskaya.
\newblock On the existence and uniqueness of the solution of the nonstationary
  problem for a viscous, incompressible fluid.
\newblock {\em Izv. Akad. Nauk SSSR Ser. Mat.}, 21:655--680, 1957.

\bibitem[KL98]{kipnis1998scaling}
Claude Kipnis and Claudio Landim.
\newblock {\em Scaling limits of interacting particle systems}, volume 320.
\newblock Springer Science \& Business Media, 1998.

\bibitem[KOV89]{kipnis1989hydrodynamics}
Claude Kipnis, Stefano Olla, and SRS Varadhan.
\newblock Hydrodynamics and large deviation for simple exclusion processes.
\newblock {\em Communications on Pure and Applied Mathematics}, 42(2):115--137,
  1989.

\bibitem[Kry13]{Kry}
Nikolay~V. Krylov.
\newblock A relatively short proof of {I}t\^{o}'s formula for {SPDE}s and its
  applications.
\newblock {\em Stoch. Partial Differ. Equ. Anal. Comput.}, 1(1):152--174, 2013.

\bibitem[KS91]{KS}
Ioannis Karatzas and Steven~E. Shreve.
\newblock {\em Brownian motion and stochastic calculus}, volume 113 of {\em
  Graduate Texts in Mathematics}.
\newblock Springer-Verlag, New York, second edition, 1991.

\bibitem[KT01]{kochtataru}
Herbert Koch and Daniel Tataru.
\newblock Well-posedness for the {N}avier-{S}tokes equations.
\newblock {\em Adv. Math.}, 157(1):22--35, 2001.

\bibitem[Lio60]{Lionsjl}
Jacques-Louis Lions.
\newblock Sur la r\'egularit\'e et l'unicit\'e des solutions turbulentes des
  \'equations de {Navier} {Stokes}.
\newblock {\em Rendiconti del Seminario Matematico della Universit\`a di
  Padova}, 30:16--23, 1960.

\bibitem[LL87]{LL87}
Lew~D. Landau and Evgeny~M. Lifshitz.
\newblock {\em Course of theoretical physics. {V}ol. 6}.
\newblock Pergamon Press, Oxford, second edition, 1987.
\newblock Fluid mechanics, Translated from the third Russian edition by J. B.
  Sykes and W. H. Reid.

\bibitem[LS18]{leslie2018energy}
Trevor~M Leslie and Roman Shvydkoy.
\newblock The energy measure for the euler and navier--stokes equations.
\newblock {\em Archive for Rational Mechanics and Analysis}, 230:459--492,
  2018.

\bibitem[Mar10]{mariani2010large}
Mauro Mariani.
\newblock Large deviations principles for stochastic scalar conservation laws.
\newblock {\em Probability theory and related fields}, 147:607--648, 2010.

\bibitem[Mar18]{MD}
Davit Martirosyan.
\newblock Large deviations for invariant measures of the white-forced 2{D}
  {N}avier-{S}tokes equation.
\newblock {\em J. Evol. Equ.}, 18(3):1245--1265, 2018.

\bibitem[Mat03]{mattingly2003recent}
Jonathan Mattingly.
\newblock On recent progress for the stochastic navier stokes equations.
\newblock {\em Journ{\'e}es Equations aux d{\'e}riv{\'e}es partielles}, pages
  1--52, 2003.

\bibitem[Ner19]{N}
Vahagn Nersesyan.
\newblock Large deviations for the {N}avier-{S}tokes equations driven by a
  white-in-time noise.
\newblock {\em Ann. H. Lebesgue}, 2:481--513, 2019.

\bibitem[Pro59]{PG}
Giovanni Prodi.
\newblock Un teorema di unicit\`a per le equazioni di {N}avier-{S}tokes.
\newblock {\em Ann. Mat. Pura Appl. (4)}, 48:173--182, 1959.

\bibitem[QY98]{QY}
Jeremy Quastel and Horng-Tzer Yau.
\newblock Lattice gases, large deviations, and the incompressible
  {N}avier-{S}tokes equations.
\newblock {\em Ann. of Math. (2)}, 148(1):51--108, 1998.

\bibitem[RPDO{\etalchar{+}}21]{russo2021finite}
Antonio Russo, Sergio~P Perez, Miguel~A Dur{\'a}n-Olivencia, Peter Yatsyshin,
  Jos{\'e}~A Carrillo, and Serafim Kalliadasis.
\newblock A finite-volume method for fluctuating dynamical density functional
  theory.
\newblock {\em Journal of Computational Physics}, 428:109796, 2021.

\bibitem[RRS16]{RRS}
James~C. Robinson, Jos\'{e}~L. Rodrigo, and Witold Sadowski.
\newblock {\em The three-dimensional {N}avier-{S}tokes equations}, volume 157
  of {\em Cambridge Studies in Advanced Mathematics}.
\newblock Cambridge University Press, Cambridge, 2016.
\newblock Classical theory.

\bibitem[RZ12]{RZ2}
Michael R\"{o}ckner and Tusheng Zhang.
\newblock Stochastic 3{D} tamed {N}avier-{S}tokes equations: existence,
  uniqueness and small time large deviation principles.
\newblock {\em J. Differential Equations}, 252(1):716--744, 2012.

\bibitem[Ser62]{SJ}
James Serrin.
\newblock On the interior regularity of weak solutions of the {N}avier-{S}tokes
  equations.
\newblock {\em Arch. Rational Mech. Anal.}, 9:187--195, 1962.

\bibitem[Shi74]{shinbrot1974energy}
Marvin Shinbrot.
\newblock The energy equation for the navier--stokes system.
\newblock {\em SIAM Journal on Mathematical Analysis}, 5(6):948--954, 1974.

\bibitem[Spo12]{HS}
Herbert Spohn.
\newblock {\em Large scale dynamics of interacting particles}.
\newblock Springer Science \& Business Media, 2012.

\bibitem[SS06]{SS}
Sivaguru~S. Sritharan and Pichai Sundar.
\newblock Large deviations for the two-dimensional {N}avier-{S}tokes equations
  with multiplicative noise.
\newblock {\em Stochastic Process. Appl.}, 116(11):1636--1659, 2006.

\bibitem[Wie18]{WE}
Emil Wiedemann.
\newblock Weak-strong uniqueness in fluid dynamics.
\newblock In {\em Partial differential equations in fluid mechanics}, volume
  452 of {\em London Math. Soc. Lecture Note Ser.}, pages 289--326. Cambridge
  Univ. Press, Cambridge, 2018.

\bibitem[Zha12]{Z}
Hui~Yan Zhao.
\newblock Large deviations for 2-{D} stochastic {N}avier-{S}tokes equations
  with jumps.
\newblock {\em Acta Math. Sinica (Chin. Ser.)}, 55(3):499--516, 2012.

\bibitem[Zha19]{zhang2019remarks}
Zujin Zhang.
\newblock Remarks on the energy equality for the non-newtonian fluids.
\newblock {\em Journal of Mathematical Analysis and Applications},
  480(2):123443, 2019.

\bibitem[ZZ15]{ZZ}
Jianliang Zhai and Tusheng Zhang.
\newblock Large deviations for 2-{D} stochastic {N}avier-{S}tokes equations
  driven by multiplicative {L}\'{e}vy noises.
\newblock {\em Bernoulli}, 21(4):2351--2392, 2015.

\end{thebibliography}

\end{document}